\documentclass[12pt]{amsart}	
\usepackage[usenames]{color}

\usepackage{float}

\usepackage{soul}
\usepackage[normalem]{ulem}
	\usepackage{wasysym}

\usepackage{amsmath,amssymb,amsthm,amsfonts,enumerate,url,mathbbol}
\usepackage{mathrsfs,tikz,graphicx,pifont, svg,stackrel} 
\usepackage{hyperref}

\usepackage{mathtools}

\usepackage{aliascnt}

\usepackage{pgfplots}
\usepgfplotslibrary{polar}

\usetikzlibrary{automata,positioning,arrows}
\usetikzlibrary{arrows.meta,calc,decorations.markings,math,arrows.meta}
	
\usepackage[utf8]{inputenc}
\usepackage{verbatim}
\usepackage{array}

\def\be{\begin{equation}}
\def\ee{\end{equation}}
\def\beq{\begin{eqnarray}}
\def\eeq{\end{eqnarray}}
\def\beqs{\begin{eqnarray*}}
\def\eeqs{\end{eqnarray*}}
\def\ea{\end{array}}
\def\ea{\end{array}}
\def\bt{\begin{theo}}
\def\et{\end{theo}}
\def\bl{\begin{lemma}}
\def\el{\end{lemma}}
\def\br{\begin{rems}}
\def\er{\end{rems}}
\def\bc{\begin{coro}}
\def\ec{\end{coro}}

\theoremstyle{plain}
\newtheorem{theo}{Theorem}[section]

\newaliascnt{cor}{theo}
\newaliascnt{prop}{theo}
\newaliascnt{lemma}{theo}

\newtheorem{lemma}[lemma]{Lemma}
\newtheorem{prop}[prop]{Proposition}
\newtheorem{cor}[cor]{Corollary}

\aliascntresetthe{cor}
\aliascntresetthe{prop}
\aliascntresetthe{lemma}

\theoremstyle{definition} 

\newaliascnt{defi}{theo}
\newaliascnt{assum}{theo}
\newaliascnt{assums}{theo}
\newaliascnt{prob}{theo}

\newtheorem{defi}[defi]{Definition}
\newtheorem{assum}[assum]{Assumption}
\newtheorem{assums}[assums]{Assumptions}

\aliascntresetthe{defi}
\aliascntresetthe{assum}
\aliascntresetthe{assums}
\aliascntresetthe{prob}

\theoremstyle{remark}

\newaliascnt{rems}{theo}
\newaliascnt{rem}{theo}
\newaliascnt{exa}{theo}
\newaliascnt{exs}{theo}

\newtheorem{rems}[rems]{Remarks}
\newtheorem{rem}[rem]{Remark}
\newtheorem{exa}[exa]{Example}

\aliascntresetthe{rems}
\aliascntresetthe{rem}
\aliascntresetthe{exa}
\aliascntresetthe{exs}

\numberwithin{theo}{section}
\numberwithin{equation}{section}
\numberwithin{figure}{section}

\DeclareMathOperator{\Real}{\Re}
\DeclareMathOperator{\Id}{Id}

\DeclareMathOperator{\Tr}{Tr}
\DeclareMathOperator{\diag}{diag}

\DeclareMathOperator{\lin}{span}

  \def\mG{\mathsf{G}} 
    
  \def\mV{\mathsf{V}}
  
  \def\mE{\mathsf{E}}

  \def\mmu{\mathsf{u}}

 \def\mv{\mathsf{v}}
 \def\me{\mathsf{e}}
 \def\mw{\mathsf{w}}
  \def\mW{\mathsf{W}}
  \def\mf{\mathsf{f}}

\newcommand{\R}{\mathbb{R}}
\newcommand{\C}{\mathbb{C}}
\newcommand{\N}{\mathbb{N}}

\newcommand{\bL}{\mathbf{L}}

\title{Linear Hyperbolic Systems on Networks:\\
 Well-posedness and qualitative properties} 

\author[M.~Kramar Fijav\v{z}]{Marjeta Kramar Fijav\v{z}}
\address{Marjeta Kramar Fijav\v{z}, University of Ljubljana, Faculty of Civil and Geodetic Engineering, Jamova 2, SI-1000 Ljubljana, Slovenia / Institute of Mathematics, Physics, and Mechanics, Jadranska 19, SI-1000 Ljubljana, Slovenia}
\email{marjeta.kramar@fgg.uni-lj.si}

\author[D.~Mugnolo]{Delio Mugnolo}
\address{Delio Mugnolo, Lehrgebiet Analysis, Fakult\"at Mathematik und Informatik, Fern\-Universit\"at in Hagen, D-58084 Hagen, Germany}
\email{delio.mugnolo@fernuni-hagen.de}

\author[S.~Nicaise]{Serge Nicaise}
\address{Serge Nicaise, Universit\'e Polytechnique Hauts-de-France, LAMAV,
FR CNRS 2956, 
F-59313 - Valenciennes Cedex 9 France}
\email{Serge.Nicaise@uphf.fr}



\thanks{The work of M.K.F.~was partially supported by the Slovenian Research Agency, Grant No.~P1-0222, and the work of D.M.~by the Deutsche Forschungsgemeinschaft (Grant 397230547). This article is based upon work from COST Action 18232 MAT-DYN-NET, supported by COST (European Cooperation in Science and Technology), www.cost.eu.}

\begin{document}

\begin{abstract}
{
We study hyperbolic systems of one-dimensional partial differential equations under general, possibly non-local boundary conditions. A large class of evolution equations, either on individual 1-dimensional intervals or on general networks, can be reformulated in our rather flexible formalism, which generalizes the classical technique of first-order reduction. We study forward and backward well-posedness; furthermore, we provide necessary and sufficient conditions on both the boundary conditions and the coefficients arising in the first-order reduction for a given subset of the relevant ambient space to be invariant under the flow that governs the system. Several examples are studied.
}
\end{abstract}

\maketitle

\section{Introduction}

This paper is devoted to the study of systems of partial differential equations in 1-dimensional setting, more precisely, on collections of intervals: not only internal couplings are allowed, but also interactions at the endpoints of the intervals. It is then natural to interpret these systems as networks, and in fact, we will dwell on this viewpoint throughout the paper. 

Partially motivated by investigations in quantum chemistry since the 1950s, differential operators of second order on networks have been often considered in the mathematical literature since the pioneering investigations by Lumer~\cite{Lum80} and Faddeev and Pavlov~\cite{PavFad83}: in these early examples, either heat or Schrödinger equations were of interest. This has paved the way to a manifold of investigations, see e.g.\ the historical overview in~\cite{Mug14}.

The equations we are going to study in this paper will, however, be rather hyperbolic; more precisely, the \textit{hyperbolic systems} of partial differential equations of our interest  are of the form
\[
\dot{u} = M u' + Nu,
\] 
where, here and below, we denote everywhere by 
$\dot{u}$ and $u'$
the partial derivative  of a function $u$ with respect to the time variable $t$ and to the space variable $x$, respectively.

Each of these equations models a physical system: we consider several of these systems  and allow them to interact at their boundaries, thus producing a collection of hyperbolic systems on a network.  Hyperbolic evolution equations of different kinds taking place on the edges of a network have been frequently considered in the literature, we refer to~\cite{nicaise:93,LagLeuSch94,dagerzuazua,DorKraNag10,Mug14} for an overview. 
{Let us emphasize that we shall only consider \emph{linear} systems: for a survey on some recent developments of the theory for  nonlinear hyperbolic systems and many practical applications see e.g.~\cite{BCGHP14}.}

On each edge of the network we allow for possibly different dynamics (say, Dirac-like, wave-like, beam-like, etc.), thus it would be more precise to write
\begin{equation}\label{eq:acp0}
\dot{u}_\me=M_\me u'_\me+ N_\me u_\me,\qquad \me\in\mE,
\end{equation}
where $\mE$ is the edge set of the considered network. In particular,  in the easiest cases $M_\me$ may be a diagonal matrix of coefficients of a transport-like equations, but $M_\me$ may well have off-diagonal entries, or even have a symplectic structure: additionally, we allow all these $M_\me$'s to have different size, which of course has to be taken into account by the boundary conditions. 

We are \textit{not} going to assume the matrices $M_\me$ to be either positive or negative semidefinite -- in fact, not even Hermitian; therefore, it is at a first glance not clear at which endpoints the boundary conditions should be imposed at all. Indeed, the choice of appropriate transmission conditions in the vertices of the network is the biggest difficulty one has to overcome.

\medskip
While Ali Mehmeti began the study of wave equations on networks already in~\cite{Ali89}, it was to the best of our knowledge only at the end of the 1990s that first order differential operators on networks began to be studied. In~\cite{Car99}, Carlson defined on a network the \emph{momentum operator} -- i.e., the operator defined edgewise as $\imath\frac{d}{dx}$ -- and gave a sufficient condition -- in terms of the boundary conditions satisfied by functions in its domain -- for self-adjointness, hence for generation of a unitary group governing a system of equations
\begin{equation}\label{eq:trans-carls}
\dot{u} = \pm u'
\end{equation}
with couplings in the boundary (i.e., in the nodes of the networks). Similar ideas were revived in~\cite{EndSte10,Exn13}, where different sufficient conditions of combinatorial or algebraic nature were proposed. A characteristic equation and the long-time behavior of the semigroup governing~\eqref{eq:acp0} as well as further spectral and extension theoretical properties were discussed in~\cite{Klo10,JorPedTia13}, respectively, in dependence of the boundary conditions. {While all the above mentioned  authors -- as well as the present manuscript -- apply Hilbert space techniques, a semigroup approach to study simple transport equations in Banach spaces (like the space of $L^1$-functions  along the edges of a network) was presented in \cite{KraSik05,MatSik07,DorKraNag10}, see also \cite[Sec.~18]{BatKraRha17} and the references given there.}

All these above mentioned papers treat essentially the same parametrization of boundary conditions, namely 
\[
\underline{u(0)}=T\underline{u(\ell)}
\]
for a suitable matrix $T$ (possibly consisting of diagonal blocks that correspond to the network's vertices), where $\underline{u(0)}$ and $\underline{u(\ell)}$ denote the vectors of boundary values of $u$ at the initial and terminal endpoints of all intervals, respectively. 

Bolte and Harrison studied in~\cite{BolHar03} the Dirac equation on networks. The 1D Dirac equation consists of a system of two coupled first order (both in time and space) equations, much like~\eqref{eq:acp0}; the matrix $M_\me$ is Hermitian, which allows for simple integration by parts and, in turn, for the emergence of a convenient symplectic structure. Both internal and boundary couplings had to be considered, and the relevant coupling matrix is indefinite. They thus adopted the parametrization
\[
A\underline{u(0)}+B\underline{u(\ell)}=0
\]
for the boundary conditions, for suitable matrices $A,B$: 
mimicking ideas from~\cite{KosSch99}, they were able to characterize those $A,B$ that lead to self-adjoint extensions. Self-adjointness of more general first-order differential operator matrices has been studied in~\cite{SchSeiVoi15}.

In this paper, we opt for yet another parametrization of the boundary conditions, inspired by a classical Sturm--Liouville formalism borrowed by Kuchment to discuss self-adjoint extensions of Laplacians on networks in~\cite{Kuc02}  (see also  \cite{Nic:2017} for the  ``telegrapher's equation''  on networks with similar boundary conditions). More precisely, we impose boundary conditions of the form
\[
\begin{pmatrix}
\underline{u(0)}\\ \underline{u(\ell)}
\end{pmatrix}\in Y
\]
for a subspace $Y$ of the space of boundary values; and find sufficient conditions on $Y$ that, in dependence on $M$ and an auxiliary matrix $Q$, guarantee that the abstract Cauchy problem associated with~\eqref{eq:acp0} is governed by a (possibly unitary, under stronger assumptions) group, or a (possibly contractive, under stronger assumptions) semigroup.

The auxiliary matrix $Q$ -- often called a \textit{Friedrichs symmetrizer} in the literature, see~\cite[Def.~2.1]{BenGav07} --  will play a fundamental role in our approach. Roughly speaking, its role is \textit{not} to diagonalize $M$, but only to make it Hermitian; this is done by suitably modifying the inner product of the $L^2$-space over the network by means of $Q$, which therefore has in turn to be positive definite; especially for this reason, our whole theory is essentially relying upon the Hilbert space structure. 
Our approach allows us in particular to prove generation of unitary $C_0$-groups and contractive $C_0$-semigroups (and, by perturbation, of general $C_0$-(semi)groups). {This has a long tradition that goes back to Lax and Phillips~\cite{LaxPhi60}, who already propose the idea 
 of transforming boundary conditions into the requirement} that  at each boundary point $\mv$
the boundary values belong to a given subspace $Y_\mv$. Indeed, while our well-posedness results are not surprising \textit{once the correct boundary conditions are found}, the actually tricky task -- as long as $M$ is not diagonalizable, the standard assumption among others in~\cite{BenGav07,
JacZwa12,Rau12,Eng13,BasCor16} -- is to actually find the right dimension of the space $Y_\mv$. In this paper, we pursue this task by a fair amount of linear algebra that eventually allows us to parametrize the boundary conditions leading to \textit{contractive} (semi)groups. This should be compared with the more involved situation in higher dimension, see e.g.~\cite{Rau85}, which allows for less explicit representation of the boundary conditions.
Our setting is thus arguably more general than the approaches to hyperbolic systems on networks that have recently emerged, including port-Hamiltonian systems~\mbox{\cite{ZwaLeGMasVil10,JacMorZwa15,{WauWeg19}}}  
and hyperbolic systems that can be transformed into characteristic forms via Riemann coordinates~\mbox{\cite{BasCor16}}, both based on diagonalization arguments.

{A relevant by-product of our approach is the possibility to characterize in terms of $Q,M,N$ positivity and further qualitative properties of the solutions of the initial value problem associated with~\eqref{eq:acp0}.}
In this context, we regard as particularly relevant~\autoref{prop:posit} and~\autoref{lem:serge-diagon}, which roughly speaking state that the semigroup governing~\eqref{eq:acp0} can only be positive if $M$ is diagonal, up to technical assumptions (including that $Q$ is diagonal too; this is not quite restrictive, as e.g.\ all of the examples we will discuss in Section~\ref{sec:examples} will satisfy it); this negative result essentially prevents most evolution equations of non-transport type arising in applications from being governed by a positive semigroup. 

The obtained results also form a basis for studies of different stability and control problems related to the presented hyperbolic systems. 
In \autoref{rem:control} we give some immediate indications for these studies but leave further problems  for possible future considerations.   

Let us sketch the structure of our paper.
In Section~\ref{sec:general} we present our general assumptions and discuss their role by showing that a broad class of examples fits into our scheme. In Sections~\ref{sec:iso} and~\ref{sec:contr} we then show that our description of boundary condition allows for easy description of realizations that generate (semi)groups. We also find necessary and sufficient conditions for qualitative properties of these semigroups, including reality and positivity.

We conclude this paper by reviewing in Section~\ref{sec:examples} several applications of our method; among other we discuss forward and backward well-posedness of different equations modeling wave phenomena on networks, including 1D Saint--Venant, Maxwell, and Dirac equations. We study different regimes for the Saint-Venant equation and discuss transmission conditions in the vertices that imply forward, but not backward well-posedness of the Dirac equation. We also study in detail an interesting model of mathematical physics for heat propagation in supercold molecules; we extend the results from~\cite{Racke} by providing physically meaningful classes of transmission conditions implying well-posedness and proving nonpositivity of the semigroup governing this system.

Some technical results, which seem to be folklore, are recalled in the appendices.
  
\medskip
{\bf Acknowledgment.} The authors would like to thank Roland Schnaubelt (Karlsruhe) for interesting suggestions concerning early literature devoted to the topic of hyperbolic systems.

\section{General setting and main examples}\label{sec:general}

Let $\mE$ be a nonempty finite set, which we will identify with the edges of a network upon associating a length $\ell_\me$ with each $\me\in \mE$.
To fix the ideas, take $\me\in \mE$ and $\ell_\me>0$.  (We restrict for simplicity to the case of a network consisting of edges of finite length only, although our results can be easily extended to the case of networks consisting of finitely many leads -- semi-bounded intervals -- attached to  a ``core'' of finitely many edges of finite length.)  We will consider evolution equations of the form 
\begin{equation}\label{eq:max1}
\dot{u_\me }(t,x)=M_\me(x)u'_\me(t,x)+ N_\me(x) u_\me (t,x), \quad t\ge 0,\ x\in (0,\ell_\me),
\end{equation}
where $u_\me$ is a vector-valued function of size $k_\me\in \N_1:=\{1,2, \ldots\}$, and $M_\me$ and $N_\me$ are 
 matrix-valued functions of size $k_\me\times k_\me$. 
We will couple equations \eqref{eq:max1} for different $\me\in \mE$ via boundary conditions given later on.

If $M_\me(x)$ is Hermitian for all $x$, then integrating by parts we obtain   for all $u\in \bigoplus_{\me\in\mE} H^1(0,\ell_\me)^{k_\me}$ \footnote{ Throughout this paper $\oplus_{\me\in\mE} H_\me$ denotes the direct sum of the Hilbert spaces $H_\me$, $\me\in \mE$, i.e., $\bigoplus_{\me\in\mE} H_\me:=\{(h_\me)_{\me\in \mE}: h_\me\in H_\me\}$.}
\begin{equation}\label{eq:ur-ibp}
\begin{split}
2\Re\sum_{\me\in\mE}\int_0^{\ell_\me}
M_\me u'_\me\cdot \bar u_\me \,dx
=-\sum_{\me\in\mE}\int_0^{\ell_\me}
  u_\me \cdot   
M'_\me \bar u_\me  \,dx+\sum_{\me\in\mE} \left[  M_\me   u_\me \cdot \bar u_\me \right]\Big|_0^{\ell_\me},
\end{split}
\end{equation}
which -- provided $M'_\me$ is essentially bounded -- allows for an elementary dissipativity analysis of the operator  that governs the abstract Cauchy problem associated with~\eqref{eq:max1} in a natural Hilbert space. Also the case of diagonalizable matrices $M_\me$ is benign enough, see e.g.~\cite[\S~7.3]{Eva10}.
In the case of general $M_\me$, however, it is not easy to control all terms that arise when integrating against test functions and we have to resort to different ideas.

\begin{assums}\label{assum} 
\begin{enumerate}[(1)] For each $\me\in\mE$ the following holds.
\item The matrix $M_\me(x)$ is invertible for  each $x\in [0,\ell_\me]$ and the mapping $[0,\ell_\me]\ni x\mapsto M_\me(x)\in M_{k_\me}(\C)$ is  Lipschitz continuous, in other words, $M_\me\in W^{1,\infty}(0,\ell_\me)$.
\item The mapping $[0,\ell_\me]\ni x\mapsto N_\me(x)\in M_{k_\me}(\C)$ is of class $L^\infty$.
\item\label{assum:Q} There exists a  Lipschitz continuous function
$[0,\ell_\me]\ni x\mapsto Q_\me(x)\in M_{k_\me}(\C)$ such that 
\begin{enumerate}[(i)]
\item\label{assumMe}  $Q_\me(x)$ and $Q_\me(x) M_\me(x)$ are Hermitian for all $x\in [0,\ell_\me]$,
\item
\label{sn:29/1:1} $Q_\me(\cdot)$ is \emph{uniformly positive definite}, i.e.,
there exists $q>0$ such that  
\[Q_\me(x)\xi \cdot \bar \xi \geq q \|\xi \|^2
\text{ for all }\xi \in \C^{k_\me} \text{  and } x\in [0,\ell_\me].\footnote{
 Throughout this paper, $\C^{k_\me}$ is equipped with its Euclidean norm $\|\cdot \|$ and $M_{k_\me}(\C)$ 
with the induced operator norm, also denoted by $\|\cdot \|$, since no confusion is possible.}\]
\end{enumerate}
\end{enumerate}
\end{assums}

If $M_\me(x)$ is Hermitian for all $x$, then  \autoref{assum}.(\ref{assumMe}) are trivially satisfied by taking $Q_\me$ to be the $k_\me\times k_\me$ identity matrix, {although this is not the only possible choice and, in fact, it is sometimes actually possible and convenient to take non-diagonal $Q_\me$}.  \autoref{assum}.(\ref{assum:Q}) holds if and only if the system \eqref{eq:max1} is  hyperbolic in the sense of \cite{BasCor16}, see \autoref{l:hyperbolic_Q-new} below. 
But we prefer this formulation because the matrices $Q_\me$  will be involved in the boundary conditions.

The fact that $Q_\me(x) M_\me(x)$ is Hermitian for all $x$ greatly simplifies our analysis. At the same time, many examples from physics, chemistry, biology, etc., fit in this framework.

The most trivial examples are obtained by taking $M_\me$ as a diagonal matrix with spatially constant entries: this choice leads to classical (vector-valued) transport problems on networks. For $k_\me\equiv 1$ they were considered in \cite{KraSik05} and subsequent papers, cf.\ the literature quoted in~\cite[Sec.~18]{BatKraRha17}.

\begin{exa}\label{exa:saintv}
 The $2\times 2$ hyperbolic system
\begin{equation}\label{syst2time2}
\left\{\begin{array}{ll}
\dot{p}+L q' +G p+Hq=0  \quad\hbox{ in  }  (0,\ell) \times(0,+\infty),\\
\dot{q}+P p' +K q+Jp=0 \quad\hbox{ in  }  (0,\ell)\times(0,+\infty),
\end{array}
\right.
\end{equation}
on a real interval $(0,\ell)$ generalizes the first order reduction of the wave equation and offers a general framework to treat models that appear in several applications.  The analysis of this system on networks
with different boundary conditions 
has been performed in  \cite{Nic:2017}.

In electrical engineering
\cite{MaffucciMiano:06,ImpJol14},
$p$ (resp. $q$) represents the voltage $V$ (resp. the electrical  current $I$) at $(\ell-x, t)$,
 $H=J=0$, $L=\frac{1}{C}$, $P=\frac{1}{L}$,
$G=\frac{\hat{G}}{C}$, $K=\frac{R}{L}$, where  $C>0$ is the capacitance, $L>0$ the inductance, $\hat{G}\geq 0$ the conductance,
and $R\geq 0$ the resistance: \eqref{syst2time2} is then referred to as ``telegrapher's equation''. 

This system also models arterial blood flow \cite{Bressanetal,Carlson:11}
for which  $p$ is the pressure and $q$ the flow rate at $(x, t)$, $L=\frac{1}{C}$, $P=A$, $K=-\frac{2\alpha}{\alpha-1} \frac{\nu}{A}$,
$G=0$, where $A>0$ is the vessel cross-sectional area, $C>0$ is the vessel compliance, $\nu\geq 0$ is the kinematic viscosity coefficient ($\nu\approx 3.2 10^{-6} m^2/s$ for blood)
and $\alpha>1$ is the Coriolis coefficient or correction coefficient ($\alpha=4/3$ for  Newtonian fluids, while $\alpha=1.1$ for non-Newtonian fluids, like blood). 

Given $L,P\in \C$, the \autoref{assum}.(3) hold for system
\eqref{syst2time2} with
\begin{equation}\label{eq:Qhypsys}
M_\me=
-
\begin{pmatrix}
0 & L\\ P & 0
\end{pmatrix},\quad
N_\me=-\begin{pmatrix}
G & H\\ K & J
\end{pmatrix},
\quad \hbox{and}\quad
Q_\me=\begin{pmatrix}
a & b\\ c & d
\end{pmatrix}
\end{equation}
if and only if \footnote{ For a complex number $z$, $\bar z$ denotes its complex conjugate.}
\begin{equation}\label{eq:abcd}
a,d\in \R,\quad b=\overline{c},\quad a>0, \quad ad>|b|^2,\quad aL=d\overline{P},\quad bP,b\overline{L}\in \R,\quad LP\ne 0.
\end{equation}
\end{exa}

\begin{exa}\label{exa:momentum}
The momentum operator $i\frac{d}{d x}$~\cite{Car00,Exn13} does \textit{not} satisfy  \autoref{assum}. 
More generally, if $M_\me$ is skew-Hermitian, then~\autoref{assum} imply that there exists a Hermitian, positive definite matrix $Q_\me$ that anti-commutes with $M_\me$. But then $\Tr Q_\me$, the trace of the matrix $Q_\me$, satisfies
\[\Tr Q_\me = \Tr (M_\me^{-1} Q_\me M_\me) = -\Tr Q_\me ,\] hence $\Tr Q_\me=0$  which is in contradiction with positive definiteness of $Q_\me$.
 \end{exa}

\begin{exa}\label{Saint-Venant}
 The linearized Saint-Venant equation gives rise to a case where 
  $N_\me\ne0$, see \cite[Eq.~(1.27)]{BasCor16}. Indeed, it corresponds to the $2\times 2$ system
\begin{equation}\label{systSaint-Venant}
\left\{
\begin{array}{rcll}
\dot{h}&=&-V h'-H u'-V' h-H' u \quad &\hbox{ in  }  (0,\ell) \times(0,+\infty),\\
\dot{u} &=&- V  u'-g h'+C_f \frac{V^2}{H^2} h
-(V'+2C_f \frac{V}{H}) u\quad&\hbox{ in  }  (0,\ell)\times(0,+\infty),
\end{array}
\right.
\end{equation}
where $h$ is the water depth and  $u$  the water velocity,
and corresponds to 
the linearization around a steady state $(H,V)$
of the Saint-Venant model,
that in particular satisfies
\begin{equation}
\label{sn:15/4:1}
H\ne 0, (HV)'=0 \hbox{ and } gH-V^2\ne 0.
\end{equation}
Here, $g$ is the constant of gravity and $C_f$ is a (positive) constant  friction coefficient. 

Note that \autoref{assum} holds for system
\eqref{systSaint-Venant} with 
\begin{equation}
\label{aq:defMQ-Saint-Venant}
M_\me :=
\left(\begin{array}{ll}
-V  & -H  \\ -g & -V 
\end{array}
\right)\quad\hbox{and}\quad 
Q_\me :=\left(\begin{array}{lll}
g&0\\
0& H 
\end{array}
\right)
\end{equation}
whenever $H$ and $V$ are of class $H^1$ and $H>0$, which holds as soon as we consider a non trivial and smooth enough steady state $(H,V)$, see \eqref{sn:15/4:1}. 

Finally $N_\me$ is clearly given by
\begin{equation}
\label{aq:defN-Saint-Venant}
N_\me :=
\left(\begin{array}{ll}
-V'  & -H'  \\  C_f \frac{V^2}{H^2} & -(V'+2C_f \frac{V}{H})
\end{array}
\right).
\end{equation}
\end{exa}

\section{Parametrization of the realizations: the isometric case\label{sec:iso}}

The catchiest application of our general theory arises whenever we discuss hyperbolic equations (or even systems thereof) on \textit{networks} (also known as \textit{metric graphs} in the literature); in this case, it is natural to interpret $\mE$ as a set of intervals, each with length $\ell_\me$; and boundary conditions in the endpoints $0,\ell_\me$ turn into transmission  conditions in the ramification nodes. Indeed, for each edge $\me$, $k_\me$ boundary conditions  are required. In general, they are expressed in Riemann (characteristic) coordinates, see for instance \cite{BasCor16}. This  means that system \eqref{eq:max1} is transformed into an equivalent  system with a diagonal  matrix $\tilde M_\me$ with $k_\me^+$ (resp. $k_\me^-$)  positive (resp. negative) eigenvalues with   $k_\me^++k_\me^-=k_\me$ and $k_\me^+$ (resp. $k_\me^-$) boundary conditions are imposed at $0$ (resp. $\ell_\me$), which allows to fix the incoming information.
Here, we prefer to write them in the original unknowns. Furthermore it is  \textit{a priori} not clear how these conditions should be adapted to the case of a network, so we will conversely try to parametrize all those transmission conditions in the network's vertices that lead to an evolution governed by a semigroup (of isometries).

In particular, we are going to look for dissipativity, hence $m$-dissipativity of the operator $\pm\mathcal A$ whose restriction 
to the edge $\me$ is given by
\begin{equation}\label{eq:maxop}
({\mathcal A} u)_\me:=  M_\me u'_\me+ N_\me  u_\me 
\end{equation}
 in a natural Hilbert space, see  \eqref{eq:max1}.

To begin with, let us impose the following assumption.
\begin{assum}\label{assum:basic}
Let $\mathcal G$ be a finite network
 (or metric graph) with underlying (discrete) graph $\mG$, i.e., $\mG=(\mV,\mE)$ is a finite, directed  graph with node set $\mV$ and edge set $\mE$ and each $\me\in \mE$ is identified with an interval $(0,\ell_\me)$ whereby the parametrization of the interval agrees with the orientation of the edge.
\end{assum}

We are going to study the problem~\eqref{eq:max1} in the vector space
\[
\bL^2(\mathcal G):=\bigoplus_{\me\in\mE} L^2(0,\ell_\me)^{k_\me}.
\]
Clearly, $\bL^2(\mathcal G)$ becomes a Hilbert space once equipped with the inner product
\begin{equation}\label{eq:prod?}
(u,v):=\sum_{\me\in\mE}\int_0^{\ell_\me}
Q_\me(x) u_\me (x)\cdot \bar v_\me(x)\,dx, \qquad 
u,v\in \bL^2(\mathcal G), 
\end{equation}
(where, {here and below, $z  \cdot \bar z_1$ means the Euclidean  inner product in $\C^{k_\me}$ between $z$ and $z_1$}),
which is equivalent to the canonical one. {The associated norm will be denoted by   $\|\cdot \|$, because no confusion is possible with the Euclidean and matrix norms introduced before.}

Of course, if $M_\me$ is diagonal, then \autoref{assum}.(\ref{assumMe}) is satisfied e.g.\ whenever  $Q_\me(x)$ is the identity for all $x$; however, \autoref{exa:saintv} shows that $Q_\me M_\me$ may be Hermitian even when $M_\me$ is not. It thus turns out that such an alternative inner product is tailor-made for the class of hyperbolic systems we are considering. The main reason for restricting to the Hilbert space setting is that checking dissipativity in $L^p$-spaces is less immediate.

\subsection{Transmission conditions in the vertices}
In order to tackle the problem of determining the correct transmission conditions on $\mathcal A$, let us first introduce  the maximal domain
\begin{equation}\label{eq:maxbc-2-true}
D_{\max}:=\bigoplus_{\me\in \mE} \{u\in L^2(0,\ell_\me)^{k_\me}:M_{\me}u'\in L^2(0,\ell_\me)^{k_\me}\}.
\end{equation}
We want to explicitly state the following, whose easy proof we leave to the reader. Recall that invertibility of $M_\me(x)$ is   assumed for all $x\in [0,\ell_\me]$.
\begin{lemma}\label{lem:cr}
It holds
\begin{equation}\label{eq:maxbc-2}
D_{\max}=\bigoplus_{\me\in \mE}H^1(0,\ell_\me)^{k_\me},
\end{equation}
and therefore $D_{\max}$ is densely and  compactly embedded in $\bL^2(\mathcal G)$.
\end{lemma}
We stress that compactness of the embedding can actually fail if 0 is an eigenvalue of $M_\me(\cdot)$ at the endpoints of $(0,\ell_\me)$: to see this, take over -- with obvious changes -- the proof of \cite[Lemma~4.2]{KurMugWol19}.

Let us see why we have chosen to define an alternative inner product on $\bL^2(\mathcal G)$. Under our standing assumption $Q_\me(x) M_\me(x)$ is for  all $x$ a Hermitian matrix, so it can be diagonalized -- although these matrices need not commute, so they will in general not be simultaneously diagonalizable. If however there exists a diagonal matrix $D_\me$ such that 
\begin{equation}\label{eq:DMQ}
D_\me =Q_\me(x) M_\me(x)\quad\hbox{ for all }x,
\end{equation}
then the semigroup generated by $M_\me\frac{d }{d x}$ on $\bL^2(\mathcal G)$ with respect  to the inner product in~\eqref{eq:prod?} agrees with the semigroup generated by $Q_\me(\cdot)M_\me(\cdot)\frac{d }{d x}=D_\me\frac{d }{d x}$ on $\bL^2(\mathcal G)$ with respect to the canonical inner product: the latter one is simply the shift semigroup, up to taking into account the boundary conditions, cf.~\cite [Prop.~3.3]{Dor08b}  or \cite[Prop.~18.7]{BatKraRha17} for a special case  where $D_\me = I$. So, the complete operator $\mathcal A$ will generate a semigroup $(e^{t\mathcal A})_{t \ge 0}$ that can be semi-explicitly written down by means of the Dyson--Phillips Series, since the perturbation $Q_\me N_\me$ is bounded,
see \cite[Prop.\ 3.1.2, p. 77]{Paz83} or \cite[Thm.\ 1.10]{EngNag00}; and also by means of Trotter's Product Formula, see~\cite[Exer.~III.5.11]{EngNag00}.

Even when~\eqref{eq:DMQ} does not hold we are still in a commendable situation:  indeed,
\beq\label{identityA}
\left(\mathcal A u, v\right)&=&
\sum_{\me\in\mE}\int_0^{\ell_\me}
Q_\me  \left( M_\me u'_\me+ N_\me  u_\me \right)\cdot \bar v_\me \,dx\qquad\hbox{for all }u,v\in D_{\max}.
\eeq
Due to our standing assumptions, $Q_\me M_\me $ and hence its space derivative are hermitian: integrating by parts we hence find
\begin{equation}\label{eq:big-int-by-p}
\begin{split}
\sum_{\me\in\mE}\int_0^{\ell_\me}
Q_\me   M_\me u'_\me\cdot \bar v_\me \,dx
&=
-\sum_{\me\in\mE}\int_0^{\ell_\me}
  u_\me \cdot   
  \left(\overline{Q_\me M_\me  v_\me} \right)' \,dx+\sum_{\me\in\mE} \left[   Q_\me M_\me u_\me \cdot \bar v_\me \right]\Big|_0^{\ell_\me}.
\\
&=
-\sum_{\me\in\mE}\int_0^{\ell_\me}
  u_\me \cdot
\overline{Q_\me M_\me  v'_\me} \,dx -\sum_{\me\in\mE}\int_0^{\ell_\me}
  u_\me \cdot   
  \overline{\left( Q_\me M_\me  \right)'   v_\me}  \,dx
  \\
&\qquad +\sum_{\me\in\mE} \left[ Q_\me   M_\me   u_\me \cdot \bar v_\me \right]\Big|_0^{\ell_\me}
\end{split}
\end{equation}
Now, for $u=v$ \eqref{eq:big-int-by-p} can be equivalently written as
\begin{equation}\label{ibp}
\begin{split}
2\Re\sum_{\me\in\mE}\int_0^{\ell_\me}
Q_\me   M_\me u'_\me\cdot \bar u_\me \,dx
&=-\sum_{\me\in\mE}\int_0^{\ell_\me}
 \left( Q_\me M_\me\right)'   u_\me \cdot   
     \bar u_\me  \,dx +\sum_{\me\in\mE} \left[ Q_\me   M_\me   u_\me \cdot \bar u_\me \right]\Big|_0^{\ell_\me};
\end{split}
\end{equation}
both addends on the right hand side are real, in view of our standing assumptions on $Q_\me, M_\me$.

For all $\mv\in \mV$, let us denote by $\mE_\mv$,  the set of all edges incident in $\mv$. We introduce for each $\mv\in\mV$
 the trace operator $\gamma_\mv:\bigoplus_{\me\in\mE}H^1(0,\ell_\me)^{k_\me}\to \C^{k_\mv}$ defined by \[
\gamma_\mv(u)= 
\left(u_\me(\mv)\right)_{\me\in\mE_\mv},\qquad \mv\in\mV,
\]
where $k_\mv:=\sum_{\me\in\mE_\mv} k_\me$, and the $k_\mv\times k_\mv$ block-diagonal matrix $T_\mv$ with $k_\me\times  k_\me$  diagonal blocks 
\begin{equation}
\label{eq:defTv}
T_\mv:= 
\diag \left( Q_\me(\mv) M_\me(\mv) {\iota}_{\mv \me}
\right)_{\me\in\mE_\mv},\qquad \mv\in\mV,
\end{equation}
where we recall that the $|\mV|\times |\mE|$ \textit{incidence matrix} $\mathcal I=(\iota_{\mv\me})$ of the graph $\mG$ is defined by
\begin{equation}\label{eq:incidence}
{\iota}_{\mv \me}:=\left\{
\begin{array}{ll}
-1 & \hbox{if } \mv \hbox{ is initial endpoint of } \me, \\
+1 & \hbox{if } \mv \hbox{ is terminal endpoint of } \me, \\
0 & \hbox{otherwise.}
\end{array}\right.
\end{equation}
With these notation, we see that the identity \eqref{ibp}
is equivalent to
\begin{equation}\label{eq:ibpequiv}
\begin{split}
2\Re\sum_{\me\in\mE}\int_0^{\ell_\me}
Q_\me   M_\me u'_\me\cdot \bar u_\me \,dx
&=
-\sum_{\me\in\mE}\int_0^{\ell_\me}
 \left( Q_\me M_\me\right)' u_\me \cdot   
       \bar u_\me  \,dx   \\
&\quad +
\sum_{\mv\in\mV}  T_\mv \gamma_\mv(u)\cdot\gamma_\mv (\bar u).
\end{split}
\end{equation}
Taking the real part of \eqref{identityA} and using the last identity we find
\begin{equation}\label{eq:Adiss}
\begin{split}
\Re\left(\mathcal A u, u\right)
&=
\Re\sum_{\me\in\mE}\int_0^{\ell_\me}
 Q_\me  N_\me  u_\me \cdot \bar u_\me \,dx\\
&\quad -\frac{1}{2}  \sum_{\me\in\mE}\int_0^{\ell_\me}
  \left( Q_\me M_\me\right)' u_\me \cdot      \bar u_\me  \,dx
+\frac{1}{2}\sum_{\mv\in\mV}  T_\mv \gamma_\mv(u)\cdot\gamma_\mv(\bar u).
\end{split}
\end{equation}

The boundary terms vanish if so does $\gamma_\mv(u)$ for all $\mv\in\mV$; however, { 
upon introducing the quadratic form
\begin{equation}\label{eq:qv-def}
q_\mv(\xi ):=T_\mv \xi \cdot \bar \xi,\qquad  \xi\in \C^{k_\mv},
\end{equation}
it is more generally sufficient to impose that  $\gamma_\mv(u)$ belongs to a \textit{totally isotropic subspace  $Y_\mv$
associated with $q_\mv$}, i.e., to a vector space $Y_\mv$ such that the restriction of $q_\mv$ to $Y_\mv$ vanishes identically, see \autoref{def:nonpositiveisocone}.}
  (Observe that $q_\mv(\xi)\in \R$, due to our standing assumptions on $Q_\me, M_\me$.)
This means that it suffices to assume that
\begin{equation}\label{eq:mainbc}
\gamma_\mv(u) \in Y_\mv\quad \hbox{for all } \mv\in \mV.
\end{equation}

\begin{rem}
Introducing
\[
\omega(f,g):=\sum_{\me\in\mE}\int_0^{\ell_\me}
Q_\me   M_\me u'_\me\cdot \bar v_\me \,dx,\qquad f,g\in D_{\min},
\]
defines a skew-symmetric form in the sense of~\cite[Def.~2.2]{SchSeiVoi15}, where $D_{\min}:=\bigoplus_{\me\in \mE}H^1_0(0,\ell_\me)^{k_\me}$. All skew-adjoint extensions of $\mathcal A_{|D_{\min}}$ can be then parametrized by~\cite[Thm.~3.6]{SchSeiVoi15}. We are however rather interested in the general case of possibly variable coefficients and therefore prefer to pursue an approach based on the classical Lumer--Phillips Theorem. 
\end{rem}
 
\begin{exa}\label{ex:Maxwell}
For the system  \eqref{syst2time2} with the choice of $Q_\me$ from  {\eqref{eq:Qhypsys} in which we take $a=P$, $b=c=0$, $d=L$,} 
the matrix $Q_\me M_\me$ is given by
\[
Q_\me M_\me =-\left(\begin{array}{ll}
0&PL\\
LP&0
\end{array}
\right)
\]
and therefore, with {$u_\me=(p_\me,q_\me)^\top$}, the expression $Q_\me   M_\me   u_\me \cdot \bar u_\me $ takes the form
\[
Q_\me   M_\me   u_\me \cdot \bar u_\me =
-2 \Real  (PL q_\me \bar p_\me).
\]
One may e.g.\ consider the vertex transmission conditions (see \cite[p. 56]{Nic:2017})
\begin{itemize}
\item $p$ is continuous across the vertices and $\sum_{\me\in \mE_\mv} q_\me(\mv){\iota}_{\mv \me}=0$ for all $\mv\in \mV$; or
\item $q$ is continuous across the vertices and $\sum_{\me\in \mE_\mv} p_\me(\mv){\iota}_{\mv \me}=0$ for all $\mv\in \mV$.
\end{itemize}
They both fit to our framework. Indeed, if for simplicity we write
$\gamma_\mv(u)=((p_\me(\mv))_{\me\in \mE_\mv},(q_\me(\mv))_{\me\in \mE_\mv})^\top$, then
in the first case it suffices to take
\[
Y_\mv=\lin \{{\mathbf 1}_{\mE_\mv} \} \oplus \lin \{ \iota_{\mE_\mv}\}^\perp,
\]
{where ${\mathbf 1}_{\mE_\mv}:= (1,\dots,1)\in\C^{|\mE_\mv|}$ and $\iota_{\mE_\mv}$ is the vector in $\C^{|\mE_\mv|}$ whose $\me$-th entry equals  $\iota_{\mv\me}$, the appropriate nonzero entry of the incidence matrix defined in \eqref{eq:incidence}. Observe that  \eqref{eq:mainbc} now yields that the values  $p_\me(\mv)$ coincide for all $\me\in\mE_\mv$ while the vectors $(q_\me(\mv))_{\me\in \mE_\mv}$ and $\iota_{\mE_\mv}$ are orthogonal.
To cover the second set of transmission conditions we on the contrary let}
\[
Y_\mv= \lin\{ \iota_{\mE_\mv}\}^\perp \oplus \lin\{ {\mathbf 1}_{\mE_\mv}\}.
\]
\end{exa}

Before proving our first well-posedness result, we reformulate the condition \eqref{eq:mainbc} for constant vector fields $u_\me$. 
Namely, if we assume that 
$u_\me\equiv K_\me\in \C^{k_\me}$ for all edges $\me$,  
\eqref{eq:mainbc} is equivalent to
\begin{equation}\label{eq:mainbcequiv00}
(K_\me)_{\me\in \mE_\mv} \in Y_\mv\quad \hbox{for all } \mv\in \mV.
\end{equation}
Denoting   $I_\mv:=\{1,2,\ldots, \dim Y_\mv^\perp\}$ and fixing a basis $\{\mw^{(\mv, i)}\}_{i\in I_\mv}$ of $Y_\mv^\perp\subset \C^{k_\mv}$, \eqref{eq:mainbcequiv00} is equivalent to
\be\label{eq:mainbcequiv1}
(K_\me)_{\me\in \mE_\mv}\cdot \overline{\mw^{(\mv, i)}}=0\quad \hbox{for all } i\in I_\mv, \mv\in \mV.
\ee

To write this in a global way, we first let $k:=\sum_{\me\in\mE} k_\me $. 
Now recall that each $\mw^{(\mv, i)}$ is an element of $\C^{k_\mv}$, hence it can be identified with
the vector 
$(\mw^{(\mv, i)}_\me)_{\me\in \mE_\mv}$. We denote by 
$\widetilde \mw^{(\mv, i)}\in\C^k$ its extension to the whole set of edges, namely,
\be\label{eq:wtilde}
\widetilde \mw^{(\mv, i)}_\me :=
\left\{
\begin{array}{ll}
\mw^{(\mv, i)}_\me, &\hbox{ if } \me\in \mE_\mv,\\
0 &\hbox{ else. }
\end{array}
\right.
\ee
With this notation we see that \eqref{eq:mainbcequiv1}, hence also \eqref{eq:mainbc} in this case, is equivalent to,
\be\label{eq:mainbcequiv2}
(K_\me)_{\me\in \mE}\cdot \overline{\widetilde \mw^{(\mv, i)}}=0\quad \hbox{for all } i\in I_\mv, \mv\in \mV.
\ee
In the same way each coordinate of an element of $Y_\mv, Y_{\mv}^\perp\subset \C^{k_\mv}$ corresponds to some $\me\in E_{\mv}$ and as above we can extend these spaces to $\C^k$ by putting a $0$ to the  coordinate corresponding to  $\me$ whenever $\me\notin E_\mv$. Denote these extensions  by  $\widetilde Y_\mv,$ and $\widetilde{Y_{\mv}^\perp}$, respectively. 

\begin{lemma}\label{lem:basis}
 The set $\{\widetilde \mw^{(\mv, i)}\}_{i\in I_\mv, \mv\in \mV}$ is a basis of $\C^{k}$
 if and only if
\begin{equation}\label{eq:basis} 
 \dim \sum_{\mv\in \mV} \widetilde {Y_{\mv}^\perp}= k = \sum_{\mv\in \mV}  \dim Y_\mv.
 \end{equation}
\end{lemma}

\begin{proof} 
By construction, $\{\widetilde \mw^{(\mv, i)}\}_{i\in I_\mv}$ is a basis of $\widetilde {Y_{\mv}^\perp}$.
Therefore, $\{\widetilde \mw^{(\mv, i)}\}_{i\in I_\mv, \mv\in \mV}$ is a basis of $\C^{k}$ if and only if  
\[\sum_{\mv\in\mV} \widetilde {Y_{\mv}^\perp} = \bigoplus_{\mv\in\mV} \widetilde {Y_{\mv}^\perp} = \C^{k} \]
which is further equivalent to the dimensions condition
\begin{equation}\label{eq:dim}
 \dim\sum_{\mv\in\mV} \widetilde {Y_{\mv}^\perp}= \sum_{\mv\in\mV} \dim\widetilde {Y_{\mv}^\perp} = k. 
 \end{equation}
Now, observe that $ \dim\widetilde {Y_{\mv}^\perp} =  \dim{Y_{\mv}^\perp} = k_\mv -  \dim Y_{\mv}$. Moreover, by the hand-shaking lemma, $ \sum_{\mv\in\mV} k_\mv = 2 k$, hence
\[  \sum_{\mv\in\mV} \dim\widetilde {Y_{\mv}^\perp} = 2k -  \sum_{\mv\in\mV} \dim Y_{\mv}.\]
Plugging this into \eqref{eq:dim} yields \eqref{eq:basis}.
\end{proof}

\begin{rem}\label{rem-dim}
The equivalent assertions in \autoref{lem:basis} mean that 
 the number of boundary conditions in \eqref{eq:mainbc} 
 (that is equivalent to \eqref{eq:mainbcequiv2} in the special case) 
 is exactly equal to $k$  and that these  boundary conditions are linearly independent. Furthermore, as  the support of the vector $\widetilde \mw^{(\mv, i)}$ corresponds to the set of the edges incident to $\mv$, the vectors $\widetilde \mw^{(\mv, i)}$ and $\widetilde \mw^{(\mv', i')}$, and hence also the subspaces  $\widetilde {Y_{\mv}^\perp}$ and $\widetilde {Y_{\mv'}^\perp}$, are linearly independent if $\mv$ and $\mv'$ are not adjacent. 
 However, the first equality in \eqref{eq:basis} is equivalent to the mutual linear independence of all  $\widetilde {Y_{\mv}^\perp}$, that is,
\begin{equation}\label{eq:independent-Y} 
\widetilde{ Y_{\mv}^\perp} \cap \sum_{\mv' \ne \mv}\widetilde{ Y_{\mv'}^\perp}=\{0\}\quad\text{for all }\mv\in\mV.
\end{equation}
 \end{rem}

We are finally in the position to formulate a well-posedness result in terms of the transmission conditions in~\eqref{eq:mainbc}.
\begin{theo}\label{prop:main1}
For all $\mv\in \mV$, let $Y_\mv$ be  a
totally isotropic subspace    
associated  with the quadratic form $q_\mv$ defined by~\eqref{eq:qv-def}
and assume that \eqref{eq:basis} holds.
Then both $\pm\mathcal A$,  defined {as \eqref{eq:maxop} }on the domain
\begin{equation}\label{eq:maxop-dom}
\begin{split}
D(\mathcal A):=\left\{
u\in D_{\max}:  \gamma_\mv(u)\in Y_\mv\hbox{ for all }\mv\in\mV
\right\},
\end{split}
\end{equation}
are quasi-$m$-dissipative operators.
In particular, both $\pm \mathcal A$ generate a strongly continuous semigroup and hence a strongly continuous group in $\bL^2(\mathcal G)$. {The operator $\mathcal A$ has compact resolvent, hence pure point spectrum.}
\end{theo}

\begin{proof}
Under \autoref{assum}.(2), $u\mapsto Nu$ is a bounded perturbation of $\mathcal A$. {By the Bounded Perturbation Theorem (cf.~\cite[Thm.~III.1.3]{EngNag00}), we may
without loss of generality in this proof assume} that $N=0$. Under this assumption and 
in view of~\eqref{eq:Adiss} and the definition of $D(\mathcal A)$, we see that 
\beq\label{eq:Adiss3}
\Re\left(\pm\mathcal A u, u\right)
=
\mp \frac{1}{2}  \sum_{\me\in\mE}\int_0^{\ell_\me}
  \left( Q_\me M_\me  \right)' u_\me \cdot \bar u_\me  \,dx
\eeq
for all $u\in D(\mathcal A)$. By the assumptions made on the matrices $Q_\me,M_\me$, we deduce that there exists a positive constant $C$ such that
\be\label{eq:Adiss2}
|\Re\left(\mathcal A u, u\right)|
\leq  \frac12  
\left(\max_{\me\in\mE} \sup_{x\in (0,\ell_\me)} \| \left(Q_\me M_\me  \right)'(x)\|_2\right)
\sum_{\me\in\mE}\int_0^{\ell_\me}
\|u_\me(x)\|_2^2\, dx \leq C\|u\|^2,
\ee
for all $u\in D(\mathcal A)$
which means that   $\pm \mathcal A$ with domain $D(\mathcal A)$ are both quasi-dissipative. 
By the Lumer--Phillips Theorem, it remains to check their maximality. To this aim, for any $\mf\in \bL^2(\mathcal G)$, we first look for a solution $u\in D(\mathcal A)$ of
\be\label{eq:sergepb1}
M_\me(x)u'_\me(x)=\mf_\me(x)\quad \hbox{for $x\in (0,\ell_\me)$ and all } \me\in \mE.
\ee
Such a solution is given by
\be\label{eq:varpar}
u_\me(x)=K_\me+\int_0^xM^{-1}_\me(y)\mf_\me(y)\,dy,
\hbox{for all } x\in [0,\ell_\me],  \hbox{for all } \me\in \mE,
\ee
with $K_\me\in \C^{k_\me}$.
It then remains to fix the vectors $K_\me$ in order to enforce the condition $u_\me\in D(\mathcal A)$. Since \eqref{eq:mainbc} is in our situation a $k_\me\times k_\me$ linear system 
in $(K_\me)_{\me\in \mE}$, the existence of this vector is equivalent to its uniqueness. By the previous considerations, this means that it suffices to show that system
\eqref{eq:mainbcequiv2} has the sole solution $K_\me=0$, which holds due to our assumption \eqref{eq:basis}  in \autoref{lem:basis}.
This shows that the operator
$\mathcal A$  
is an isomorphism from $D(\mathcal A)$ into $\bL^2(\mathcal G)$ and proves that $\pm\mathcal{A}$ is maximal.

To conclude, we observe that  \autoref{lem:cr} directly implies that $\mathcal A$ has compact resolvent, since $D(\mathcal A)$ is continuously embedded in $D_{\max}$.
\end{proof}

\begin{exa}
Imposing Dirichlet conditions on \textit{all} endpoints is a possibility allowed for by our formalism, taking 
\[
\gamma_\mv (u)\in \left\{( 0, \dots, 0)^\top\right\}=:Y_\mv,\qquad \mv\in\mV.
\]
However, our dimension condition rules it out, as in this case $\dim Y_\mv=0$ for all $\mv$, hence~\eqref{eq:basis} is not satisfied.
\end{exa}

By the Lumer--Phillips Theorem and~\eqref{eq:ibpequiv}, the semigroup generated by $\mathcal A$ is isometric if and only if
\begin{equation}\label{eq:isometry}
\Re\sum_{\me\in\mE}\int_0^{\ell_\me}
\left (Q_\me  N_\me -\frac{1}{2}   ( Q_\me M_\me)'  \right) u_\me \cdot \bar u_\me  \,dx=0\qquad\hbox{for all }u\in D(\mathcal A).
\end{equation}
As the next result shows, this condition is  easy to characterize using Lemma \ref{l:appendixB}.

\begin{cor}\label{cor:main1}
Under the assumptions of \autoref{prop:main1},
\be\label{cond:iso}
 Q_\me  N_\me+(Q_\me  N_\me)^\ast=  
\left(Q_\me  M_\me  \right)'
  \ee
 if and only if the system \eqref{eq:max1} on $\mathcal G$ with transmission conditions~\eqref{eq:mainbc} is governed by  a unitary group on $\bL^2(\mathcal G)$;  in particular, the energy
\begin{equation}\label{eq:energydef}
\mathcal E(t):=\frac12 \sum_{\me\in\mE}\int_0^{\ell_\me} Q_\me  u(t)\cdot \bar u(t)\; dx,\qquad t\in \R,
\end{equation}
is conserved.
\end{cor}

\begin{proof}
Under the assumptions of \autoref{prop:main1}, the identity \eqref{eq:Adiss}
guarantees that
\[
\Re\left(\mathcal A u, u\right)=0  \qquad\hbox{for all }u\in D(\mathcal A)
\]
if and only if \eqref{eq:isometry} holds.
But simple calculations show that  \eqref{eq:isometry} is equivalent to
\begin{equation}\label{eq:isometryequiv}
\sum_{\me\in\mE}\int_0^{\ell_\me}
\left (Q_\me  N_\me+(Q_\me  N_\me)^\ast -    ( Q_\me M_\me)'  \right) u_\me \cdot \bar u_\me  \,dx=0\qquad\hbox{for all }u\in D(\mathcal A).
\end{equation}
This obviously shows that \eqref{cond:iso} is a sufficient condition for the unitarity property of the   semigroup generated by $\mathcal A$.
For the necessity,  let us observe that $Q_\me  N_\me+(Q_\me  N_\me)^\ast -    ( Q_\me M_\me)'$
is hermitian. Since the test functions vanishing at each endpoint satisfy all boundary conditions, $\bigoplus_{\me\in \mE}\mathcal{D}(0,\ell_\me)^{k_\me}$
is included in $D(\mathcal A)$ and we deduce that \eqref{eq:isometryequiv} implies that
\[
 \int_0^{\ell_\me}
\left (Q_\me  N_\me+(Q_\me  N_\me)^\ast -    ( Q_\me M_\me)'  \right) u_\me \cdot \bar u_\me  \,dx=0\qquad\hbox{for all }u_\me \in \mathcal{D}(0,\ell_\me)^{k_\me} \hbox{ and all }\me\in\mE.
\]
By \autoref{l:appendixB} we conclude that \eqref{cond:iso} holds.
Finally, because
\[
\frac{d}{dt}\mathcal E(t)=\Re \left(\mathcal A u(t), u(t)\right)
\]
holds along classical solutions $u$ of~\eqref{eq:max1}, the second assertion is valid as well.
\end{proof}

Note that condition~\eqref{cond:iso} is  satisfied in the special case when $Q_\me M_\me$ is spatially constant and $Q_\me N_\me$ has zero or purely imaginary entries.

We will see,  however, that the condition~\eqref{eq:basis} is not  satisfied in   some relevant applications (in \autoref{sec:dirac}, for example). Therefore, let us present 
an alternative  approach to prove well-posedness   based on the dissipativity of $\mathcal A$ and its adjoint $\mathcal A^\ast$  that is used in \cite[Appendix A]{BasCor16} for diagonal systems.   The first step is to characterize  the adjoint operator.

{
\begin{lemma}\label{lem:adjoint}
The adjoint of the operator $\mathcal A$ defined in  \eqref{eq:maxop}-\eqref{eq:maxop-dom} is given by
 \begin{equation}\label{eq:defAast}
\begin{split}
D({\mathcal A}^\ast)&=\{v\in D_{\max}: \gamma_\mv (v)\in T_\mv^{-1} Y_\mv^\perp\hbox{ for all }\mv\in\mV\},\\
({\mathcal A}^\ast v)_\me&=- M_\me  v'_\me  -Q_\me^{-1}\left( Q_\me M_\me  \right)'   v_\me
+Q_\me^{-1}N_\me^\ast Q_\me v_\me,\qquad \me\in\mE.
\end{split}
\end{equation}
\end{lemma}

\begin{proof}
The identity~\eqref{eq:big-int-by-p} is equivalent to
\begin{equation}\label{eq:big-int-by-p-equiv}
\begin{split}
\sum_{\me\in\mE}\int_0^{\ell_\me}
Q_\me   M_\me u'_\me\cdot \bar v_\me \,dx
&=
-\sum_{\me\in\mE}\int_0^{\ell_\me}
  u_\me \cdot
\overline{Q_\me M_\me  v'_\me} \,dx -\sum_{\me\in\mE}\int_0^{\ell_\me}
  u_\me \cdot   
  \overline{\left( Q_\me M_\me  \right)'   v_\me}  \,dx
  \\
&\qquad +
\sum_{\mv\in\mV}  T_\mv \gamma_\mv(u)\cdot\gamma_\mv (\bar v).
\end{split}
\end{equation}

Since 
\[
\left(\mathcal A u, v\right)=
\sum_{\me\in\mE}\int_0^{\ell_\me}
Q_\me  \left( M_\me u'_\me+ N_\me  u_\me \right)\cdot \bar v_\me \,dx\qquad\hbox{for all }u,v\in D_{\max},
\]
we get

\begin{equation}\label{eq:Auv}
\begin{split}
\left(\mathcal A u, v\right)
&=
\sum_{\me\in\mE}\int_0^{\ell_\me}
  u_\me \cdot 
\overline{\left(-Q_\me M_\me  v'_\me  -\left( Q_\me M_\me  \right)'   v_\me
+N_\me^\ast Q_\me v_\me\right)}  \,dx
  \\
&\qquad +
\sum_{\mv\in\mV}  T_\mv \gamma_\mv(u)\cdot\gamma_\mv (\bar v)
 \qquad \hbox{for all } u, v \in D_{\max}.\end{split}
\end{equation}
Now, $v\in \bL^2(\mathcal G)$ belongs to $D({\mathcal A}^\ast)$ if and only if there exists 
$g\in \bL^2(\mathcal G)$ such that
\[
\left(\mathcal A u, v\right)=\left(u, g\right)\qquad  \hbox{for all } u\in D(\mathcal A)
\]
and in this case ${\mathcal A}^\ast v=g$.

First, by taking $u$ such that $u_\me\in \mathcal{D}(0,\ell_\me)$ in this identity 
we find
\begin{equation}\label{DEadjoint}
-Q_\me M_\me  v'_\me  -\left( Q_\me M_\me  \right)'   v_\me
+N_\me^\ast Q_\me v_\me=Q_\me g_\me 
\end{equation}
in the distributional sense. Since, under our \autoref{assum}.(1), $Q_\me M_\me$ is invertible (with a bounded inverse), we find
that $v_\me \in H^1(0,\ell_\me)$, hence $v$ belongs to $D_{\max}.$

Now we can apply  the  identity \eqref{eq:Auv} and get
\begin{equation}\label{eq:A*uv}
\begin{split}
\left(u, g\right)
&=
\sum_{\me\in\mE}\int_0^{\ell_\me}
  u_\me \cdot 
\overline{\left(-Q_\me M_\me  v'_\me  -\left( Q_\me M_\me  \right)'   v_\me
+N_\me^\ast Q_\me v_\me\right)}  \,dx
  \\
&\qquad +
\sum_{\mv\in\mV}  T_\mv \gamma_\mv(u)\cdot\gamma_\mv (\bar v)
 \qquad \hbox{for all }u\in D(\mathcal A)\end{split}
\end{equation}
and, by \eqref{DEadjoint}, we find
\[
\sum_{\mv\in\mV}  T_\mv \gamma_\mv(u)\cdot\gamma_\mv (\bar v)=0 \qquad \hbox{for all }u\in D(\mathcal A).
\]
Since the trace operator $\gamma_\mv: D_{\max}\to \C^{k_\mv}$ is surjective, so is its restriction from $D(\mathcal A)$ to $Y_\mv$. Hence we get
\[
y\cdot T_\mv \gamma_\mv (\bar v)=0\qquad \hbox{for all }y\in Y_\mv,
\]
and therefore 
\begin{equation}\label{eq:tvy}
T_\mv \gamma_\mv (v)\in Y_\mv^\perp \ .
\end{equation}
(Here, $Y_\mv^\perp$ denotes the orthogonal  complement of $Y_\mv$ in $\C^{k_\mv}$ for the Euclidean inner product.) Since $T_\mv$ is invertible,~\eqref{eq:tvy} is equivalent to
\[
  \gamma_\mv (v)\in T_\mv^{-1} Y_\mv^\perp.
\]
This concludes the proof.
 \end{proof}

\begin{theo}\label{thm:new-for-dirac}
For all $\mv\in \mV$, let both $Y_\mv$ and $T^{-1}_\mv Y_\mv$ be 
totally isotropic subspaces associated  with the quadratic form $q_\mv$ defined by~\eqref{eq:qv-def}.
Then both $\pm\mathcal A$,  defined on the domain
\[
\begin{split}
D(\mathcal A):=\left\{
u\in D_{\max}:  \gamma_\mv(u)\in Y_\mv\hbox{ for all }\mv\in\mV
\right\},
\end{split}
\]
are quasi-$m$-dissipative operators.
In particular, both $\pm \mathcal A$ generate a strongly continuous semigroup and hence a strongly continuous group in $\bL^2(\mathcal G)$. If, additionally, \eqref{cond:iso} holds, then the group is unitary.
\end{theo}

\begin{proof}
We already know that $\mathcal A$ is densely defined.
By the first part of the proof of  \autoref{prop:main1} 
(that does not depend on condition~\eqref{eq:basis}), we see that both $\pm\mathcal A$ are quasi-dissipative.
{It is easy to verify that $\mathcal A$ is closed, see, e.g., the proof of \cite[(A.24)]{BasCor16}.}
In view of~\cite[Cor.~II.3.17]{EngNag00}, it  thus suffices to show
that $\pm\mathcal A^\ast$ {is quasi-}dissipative, too. 
Observe that~\eqref{eq:ibpequiv}, which is known to hold for all $u\in D_{\max}$, implies that 
\begin{equation}\label{eq:A*diss}
\begin{split}
\Re\left({\mathcal A}^\ast u, u\right)
&=
\Re\sum_{\me\in\mE}\int_0^{\ell_\me}
\left(-\left( Q_\me M_\me  \right)'   u_\me
+N_\me^\ast Q_\me u_\me\right)\cdot  \bar u_\me \,dx\\
&\quad +\frac{1}{2}  \sum_{\me\in\mE}\int_0^{\ell_\me}
  \left( Q_\me M_\me\right)' u_\me \cdot      \bar u_\me  \,dx
-\frac{1}{2}\sum_{\mv\in\mV}  T_\mv \gamma_\mv(u)\cdot\gamma_\mv(\bar u)
\end{split}
\end{equation}
for all $u\in D({\mathcal A}^\ast)$.
Hence quasi-dissipativity of both $\pm\mathcal A^\ast$ holds if
\begin{equation}\label{eq:A*dissalgebraiccond}
 q_\mv( \xi)= 0,\qquad \hbox{for all } \xi\in T_\mv^{-1} Y_\mv^\perp,
 \end{equation}
i.e., if $T_\mv^{-1} Y_\mv^\perp$ is a totally isotropic subspace associated with the quadratic form $q_\mv$.
\end{proof} 
}

In the remainder of this section, we are going to comment on the possibility of an alternative way of formulating transmission conditions. This is especially relevant in applications to models of theoretical physics when the stress is on unitary well-posedness, see e.g.\ Section~\ref{sec:dirac}, rather than accurate description of the network structure.

\subsection{Global boundary conditions}\label{rk:globalbc}

It is natural to choose transmission conditions that reflect the connectivity of the network, that is the reason of the local boundary condition ~\eqref{eq:mainbc}. However, nonlocal boundary conditions can be imposed as well by re-writing the term
$\sum_{\mv\in\mV}  T_\mv \gamma_\mv(u)\cdot\gamma_\mv(\bar u)$ in a global way as 
\[
T \gamma(u)\cdot\gamma(\bar u),
\]
where 
\[
\gamma (u):= \left(\underline{u(0)}, \underline{u (\ell)}\right)^\top:=
\left(\left(u_\me(0)\right)_{\me\in\mE}, \left(u_\me (\ell_\me)\right)_{\me\in\mE}\right)^\top
\]
and the $2k\times 2k$ matrix $T$ is given by
\begin{equation}\label{eq:Tdef}
T:= \begin{pmatrix}
-\diag \left(
Q_\me(0) M_\me(0)\right)_{\me\in\mE}& 0
\\
0& \diag \left(Q_\me(\ell_\me) M_\me(\ell_\me)\right)_{\me\in\mE}
\end{pmatrix},
\end{equation}
without any reference to the structure of the network.
With this notation,
\[
\begin{split}
\Re\left(\mathcal A u, u\right)&=
\Re\sum_{\me\in\mE}\int_0^{\ell_\me}
\left (\left( Q_\me  N_\me -\frac12 (Q_\me M_\me)'\right) u_\me \cdot \bar u_\me 
\right)\,dx +\frac{1}{2}T \gamma(u)\cdot\gamma(\bar u),
  \end{split}
  \]
and this suggests to replace \eqref{eq:mainbc} by
\begin{equation}\label{eq:mainbc_global}
\gamma(u) \in Y,
\end{equation}
i.e., to consider $\mathcal A$ with domain
\[
\begin{split}
D(\mathcal A):=\left\{
u\in D_{\max}:  \gamma(u)\in Y
\right\},
\end{split}
\]
where $Y\subset  \C^{2k}$   is a subspace of the null isotropic cone associated with the quadratic form 
\be\label{eq:quadglobal}
T \xi\cdot \bar \xi , \quad \hbox{for all } \xi \in \C^{2k}.
\ee
This corresponds to glue all vertices together, thus forming a so-called flower graph, and to impose general transmission conditions in the only vertex of such a flower.

\begin{figure}[H]
\begin{minipage}{5cm}
\begin{tikzpicture}[scale=0.8]
\foreach \x in {45,135,225,315}{
\foreach \y in {45,135,225,315}{
\draw (\x:2cm) -- (\y:2cm);
\draw[fill] (\x:2cm) circle (2pt);
}}
\end{tikzpicture}
\end{minipage}
\begin{minipage}{3cm}
\begin{tikzpicture}[scale=0.5]
  \begin{polaraxis}[grid=none, axis lines=none]
     \addplot[mark=none,domain=0:360,samples=300] { abs(cos(6*x/2))};
   \end{polaraxis}
\draw[fill] (3.43,3.43) circle (4pt);
 \end{tikzpicture}
\end{minipage} 
     \caption{Gluing all the vertices of a graph: from a complete graph on four vertices (left) to a flower graph on six edges (right).}
\end{figure}
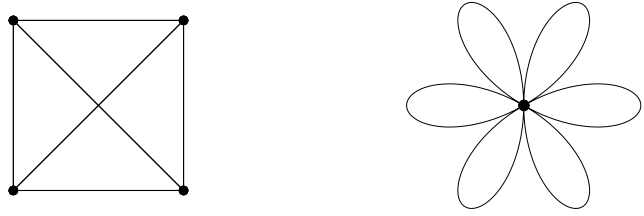

In this setting the well-posedness conditions  is different from \eqref{eq:basis}  and  sums up to
\begin{equation}\label{eq:basis_global}
\dim Y^\perp  = k=  \dim P_K Y^\perp,
\end{equation}
where $K$ is the $k$-dimensional subspace of  $\C^{2k}$
defined by
\[
K=\left\{\left(\left(K_\me\right)_{\me\in\mE}, \left(K_\me\right)_{\me\in\mE}\right)^\top:
K_\me\in \mathbb{C} \text{ for all }\me\in\mE\right\}
\]
and $P_K$ is the orthogonal projection on $K$ with respect to the Euclidean inner product of $\C^{2k}$. {To see this, we elaborate on the proof of \autoref{prop:main1} and
write   $I:=\{1,2,\ldots, \dim Y^\perp\}$. Once fixed a basis $\{\mW^{(i)}\}_{i\in I}$ of $Y^\perp\subset \C^{2k}$, using \eqref{eq:varpar} the condition \eqref{eq:mainbc_global} is equivalent to
\[
\left(\left(K_\me\right)_{\me\in\mE}, \left(K_\me\right)_{\me\in\mE}\right)^\top\cdot 
\overline{\mW^{(i)}}=\left(\left(0 \right)_{\me\in\mE}, 
\left(-\int_0^{\ell_\me}M^{-1}_\me(y)\mf_\me(y)\,dy\right)_{\me\in\mE}\right)^\top
\cdot 
\overline{\mW^{(i)}}\quad  \hbox{ for all } i\in I.
\]
Since the vector $\left(\left(K_\me\right)_{\me\in\mE}, \left(K_\me\right)_{\me\in\mE}\right)^\top$ belongs to $K$, we have
\[
\left(\left(K_\me\right)_{\me\in\mE}, \left(K_\me\right)_{\me\in\mE}\right)^\top\cdot 
\overline{\mW^{(i)}}=
\left(\left(K_\me\right)_{\me\in\mE}, \left(K_\me\right)_{\me\in\mE}\right)^\top\cdot 
\overline{P_K\mW^{(i)}},
\]
and therefore this last condition is equivalent to
\[
\left(\left(K_\me\right)_{\me\in\mE}, \left(K_\me\right)_{\me\in\mE}\right)^\top\cdot 
\overline{P_K\mW^{(i)}}=\left(\left(0 \right)_{\me\in\mE}, 
\left(-\int_0^{\ell_\me}M^{-1}_\me(y)\mf_\me(y)\,dy\right)_{\me\in\mE}\right)^\top
\cdot 
\overline{\mW^{(i)}} \quad \hbox{ for all } i\in I.
\]
Our assumptions \eqref{eq:basis_global} guarantee that we are dealing with a $k\times k$ linear system for which its associated matrix is invertible. This then proves the maximality condition.}

{
Note that the second condition in \eqref{eq:basis_global} admits different formulations stated in the next Lemma.
\begin{lemma}\label{lem:equivSerge} We have
\begin{equation}
\label{eq:equivSerge}
 \dim P_K Y^\perp = k \Leftrightarrow P_K Y^\perp=K \Leftrightarrow  Y\cap K=\{0\}.
 \end{equation}
\end{lemma}
\begin{proof}
The first equivalence directly follows from the fact that $K$ is of dimension $k$.

Let us now prove the implication $P_K Y^\perp=K\Rightarrow  Y\cap K=\{0\}$. Indeed, each $y\in Y\cap K$ satisfies
\[
y\cdot \bar z=0\quad  \hbox{for all } z\in {Y^\perp}.
\]
Since $y$ is also in $K$, we have $y\cdot \bar z=y\cdot \overline{P_K z}$, thus the last property is equivalent to
\[
y\cdot \overline{P_K z}=0\quad  \hbox{for all } z\in  {Y^\perp}.
\]
By our assumption $P_K Y^\perp=K$ and we deduce that
\[
y\cdot \overline{k}=0\quad  \hbox{for all }  k\in K,
\]
which yields $y=0$ since $y$ is in $K$.

Let us show the converse implication. Take $k\in K$ orthogonal to $P_K Y^\perp$, i.e.,
\[
k\cdot \overline{P_Kz}=0\quad  \hbox{for all } z\in Y^\perp.
\]
As $k$ is in $K$, we get equivalently
\[
k\cdot \overline{z}=0\quad  \hbox{for all  } z\in Y^\perp.
\]
Therefore $k$ belongs to $Y$ and hence to $Y\cap K$. Thus $k=0$, due to the assumption $Y\cap K=\{0\}$,  which proves that $P_K Y^\perp=K$.
\end{proof}

Using $\dim Y =2k-\dim Y^\perp$   we can find different   equivalent conditions to 
\eqref{eq:basis_global}. One of them is
\begin{equation}\label{eq:basis_global_2}
\dim Y=  \dim P_K Y^\perp =k.
\end{equation}

This formalism also makes possible to compare  solutions of the same system under different transmission conditions in the vertices. We will come back to this in the next section.}

{
\begin{rem}[Local boundary conditions versus global ones]
Assume that $\mathcal{A}$ is defined by the local boundary conditions
\eqref{eq:mainbc}, then they  can be viewed as  global boundary conditions
with $Y$ defined here below. Namely for all $\mv\in \mV$, introduce the matrix operator $\mathcal{P}_\mv$
(from $\C^{2k}$ to $\C^{k_\mv}$) defined by
\[
\gamma_\mv(u)=\mathcal{P}_\mv \gamma(u).
\]
Then, given the basis   $\{\mw^{(\mv, i)}\}_{i\in I_\mv}$ of $Y_\mv^\perp\subset \C^{k_\mv}$, we define
\[
\mW^{(\mv, i)}=\mathcal{P}_\mv^\top \mw^{(\mv, i)}.
\]
We may notice that
\begin{equation}\label{eq:prop1locversusglo}
\gamma(u)\cdot \overline{\mW^{(\mv, i)}}=\gamma_\mv(u)\cdot \overline{\mw^{(\mv, i)}},
\end{equation}
hence $Y^\perp$
is simply the space spanned by $\mW^{(\mv, i)}$, for $i\in I_\mv$ and $\mv\in \mV$.
Let us also  remark  that 
\begin{equation}\label{eq:prop2locversusglo}
\mW^{(\mv, i)}\cdot \overline{\mW^{(\mv', i')}}=0, \hbox{ if } \mv\ne \mv'.
\end{equation}

These two properties imply that \eqref{eq:basis_global} is equivalent to \eqref{eq:basis}.
Indeed, by \eqref{eq:prop2locversusglo}, and as $\mathcal{P}_\mv$ being surjective,  $\mathcal{P}_\mv^\top$ is injective, 
we deduce that
\begin{equation}\label{eq:prop3locversusglo}
\dim Y^\perp=\sum_{\mv\in \mV} \dim Y_\mv^\perp,
\end{equation}
while the first one guarantees that
\begin{equation}\label{eq:prop4locversusglo}
\left(\left(K_\me\right)_{\me\in\mE}, \left(K_\me\right)_{\me\in\mE}\right)^\top\cdot 
\overline{\mW^{(\mv, i)}}=
\left(K_\me\right)_{\me\in\mE_\mv} \cdot \overline{\mw^{(\mv, i)}}
=\left(K_\me\right)_{\me\in\mE} \cdot \overline{\widetilde{\mw}^{(\mv, i)}}.
\end{equation}

Hence if \eqref{eq:basis}, then \eqref{eq:prop3locversusglo} direclty implies that
$\dim Y^\perp=k$. 
Further by \eqref{eq:prop4locversusglo},
one has
\begin{equation}\label{eq:prop5locversusglo}
\left(\left(K_\me\right)_{\me\in\mE}, \left(K_\me\right)_{\me\in\mE}\right)^\top
\cdot \overline{P_K\mW^{(\mv, i)}}=\left(K_\me\right)_{\me\in\mE} \cdot \overline{\widetilde{ \mw}^{(\mv, i)}},
\end{equation}
and since the vectors $\widetilde{ \mw}^{(\mv, i)}$ span the whole $\C^k$, if
the left-hand side is zero for all $i, \mv$, 
$\left(K_\me\right)_{\me\in\mE}$ will be zero, hence the vectors $P_K\mW^{(\mv, i)}$ span the whole $K$.

Conversely if \eqref{eq:basis_global} holds, then by \eqref{eq:prop3locversusglo}, we have 
 $\sum_{\mv\in \mV} \dim Y_\mv^\perp=k$, that is by the hand-shaking lemma equivalent to  $\sum_{\mv\in \mV} \dim Y_\mv=k$. As before, due to \eqref{eq:prop5locversusglo},
 if  the vectors $P_K\mW^{(\mv, i)}$ span the whole $K$, then the vectors $\widetilde{ \mw}^{(\mv, i)}$ span the whole $\C^k$.
\end{rem}
}

\begin{exa}\label{ex:transp-group}
The arguably easiest application of our theory is the model of flows on networks discussed e.g.\ in~\cite{KraSik05}. It consists of a system of $k=|\mE|$ scalar equations 
\[
\dot{u}_\me = c_\me u'_\me 
,\qquad \me\in\mE,
\]
where $c_\me$ are positive constants. Hence, $M_\me=c_\me, $ and we can take $Q_\me=1$, for all $\me\in \mE$.
The associated matrix $T_\mv$ 
is diagonal and takes the form 
\begin{equation*}
T_\mv= 
\diag \left(  c_\me(\mv) {\iota}_{\mv \me}
\right)_{\me\in\mE_\mv},\qquad \mv\in\mV,
\end{equation*}
where $ {\iota}_{\mv \me}$ are the entries of the incidence matrix of the graph, see \eqref{eq:defTv}.

However, in order to treat more general boundary conditions, we switch to the global setting and take
\[
T=\begin{pmatrix}
-\diag(c_\me)_{\me\in\mE}& 0
\\
0&\diag(c_\me)_{\me\in\mE}
\end{pmatrix}.
\]
A rather general way of writing the transmission conditions in the vertices is 
$V_0 \underline{u(0)}=V_\ell \underline{u(\ell)}$, where $V_0,V_\ell$ are $k\times k$ matrices,
see  \cite{Eng13}. They can be equivalently expressed in our formalism by
\[
\gamma(u)\in \ker \begin{pmatrix}
V_0 & -V_\ell
\end{pmatrix}=:Y.
\]
The relevant conditions in \autoref{prop:main1} are hence whether
\begin{enumerate}[(i)]
\item the space $\ker\begin{pmatrix}V_0 & -V_\ell \end{pmatrix}=\{(\psi,\theta)\in \C^{2k}:V_0\psi=V_\ell \theta\}$ has dimension $k$ and
\item $V_0 \psi =V_\ell \theta$ implies $\sum_{\me\in \mE} c_\me |\psi_\me|^2 = \sum_{\me\in \mE} c_\me |\theta_\me|^2$  for all $\psi,\theta\in \C^k$.
\end{enumerate}
Both conditions are e.g.\ satisfied if $V_0,V_\ell=\Id$, which corresponds to transport on $k$ disjoint loops, in which case \autoref{prop:main1} confirms one's intuition that $\mathcal A$ with 
\begin{equation}\label{eq:dmaxVV}
D(\mathcal A):=\{u\in D_{\max}:u_\me(0)=u_\me(\ell_\me),\ \me\in\mE\}
\end{equation}
generates a strongly continuous group on    ${\bf L}^2(\mathcal G)$. With $k=2$, condition (i) is also fulfilled for $V_\ell=\Id$ if e.g.\ $V_0= \begin{psmallmatrix}0 & 1\\ 1 & 0\end{psmallmatrix}$, or $V_0=\frac12 \begin{psmallmatrix}1 & 1\\ 1 & 1\end{psmallmatrix}$; condition (ii) is satisfied in the former case if $c_{\me_1}=c_{\me_2}$ (hence we have by~\autoref{prop:main1} a group generator on a network which can be regarded as a loop of length $\ell_{\me_1}+\ell_{\me_2}$), but not in the latter:
  it will follow from the results in the next section that this operator, which is a prototype of those considered in~\cite{KraSik05}, still generates a strongly continuous \textit{semi}group on ${\bf L}^2(\mathcal G)$. This example shows   that our conditions   on $Y_\mv$ are tailored for \textit{unitary} group generation, as Corollary~\ref{cor:main1} shows. Some remedies to avoid such a problem will be discussed below.

We stress that the second above condition implies 
the invertibility of both $V_0,V_\ell$, which is proved in~\cite[Cor.~3.8]{Eng13} to be equivalent to the assertion that $\mathcal A$ with domain
\begin{equation*}
D(\mathcal A):=\{u\in D_{\max}:\gamma(u)\in \ker \begin{pmatrix}
V_0 & -V_\ell
\end{pmatrix}\}
\end{equation*}
is a group generator.
\end{exa}

\section{Contractive well-posedness and qualitative properties\label{sec:contr}}

Let us now discuss the more general situation in which the solutions to \eqref{eq:max1} are given by semigroups that are merely contractive. In this case,  the above computations show that much more general boundary conditions can be studied.   
Furthermore, we are also able to describe qualitative behavior of these solutions. We refer to \autoref{sec:examples} for several illustrative examples. 

Recall that the \emph{nonpositive   (resp.~nonnegative) isotropic cone} associated with a quadratic form $q:\C^k \to \R$  is the set of vectors $\xi\in \C^{k}$ such that
$q(\xi)\leq 0$   (resp.~$q(\xi)\geq 0$), see \autoref{def:nonpositiveisocone}.

\begin{theo}\label{prop:main2}
For all $\mv\in \mV$, let $Y_\mv$ be  a
 subspace of the nonpositive  isotropic cone associated with   the quadratic form 
$q_\mv$ given in \eqref{eq:qv-def} 
and assume that one of the following conditions is satisfied:
\begin{itemize}
\item \eqref{eq:basis} holds or
\item for all $\mv\in \mV$, the space  $T^{-1}_\mv Y^\perp_\mv$ is a subspace of the nonnegative isotropic cone associated with $q_\mv$.
\end{itemize}
Then the following assertions hold.
\begin{enumerate}[(1)]
\item $\mathcal A$ with domain
\[
\begin{split}
D(\mathcal A):=\left\{
u\in D_{\max}:~\gamma_\mv(u) \in Y_\mv \hbox{ for all } \mv\in \mV\right\}
\end{split}
\]
  generates a strongly continuous {quasi-contractive} semigroup $(e^{t\mathcal A})_{t\ge 0}$  in $\bL^2(\mathcal G)$
and the system \eqref{eq:max1} on $\mathcal G$ with transmission conditions
\be\label{eq:mainbcdissipative}
\gamma_\mv(u) \in Y_\mv,\quad \mv\in \mV,
\ee
is well-posed.

\item This semigroup is contractive if, additionally, ${(Q_\me N_\me)(x)}+{(Q_\me N_\me)^\ast (x)}-(Q_\me M_\me)'(x)$ is negative semi-definite for all $\me\in\mE$ and a.e.\ $x\in [0,\ell_\me]$.

\item Under the assumptions of (2), the energy in~\eqref{eq:energydef} is monotonically decreasing; in this case, there exists a projector commuting with $(e^{t\mathcal A})_{t\ge 0}$ whose null space is the set of strong stability of the semigroup and whose range is the closure of the set of periodic vectors under $(e^{t\mathcal A})_{t\ge 0}$.
\end{enumerate}
\end{theo}
 
\begin{proof}
The proof  of the assertion leading to well-posedness is exactly the same as the ones of \autoref{prop:main1}, \autoref{cor:main1},  and \autoref{thm:new-for-dirac} for operator $\mathcal A$; the sufficient condition for contractivity of the semigroup can be read off {\eqref{eq:Adiss}.}
The {third} assertion is a direct consequence of~\cite[Cor.~V.2.15]{EngNag00} and \autoref{lem:cr}.
\end{proof}

\begin{rem}
{In the previous Theorem we may replace local boundary condition by global ones and then replacing either   \eqref{eq:basis} by \eqref{eq:basis_global}
or the second local condition by a global version. Hence}
Theorem~\ref{prop:main2} especially applies to the setting of~\cite[Appendix~A]{BasCor16},  since their boundary condition (A.8) can be written as $\gamma(u)\in \ker \mathcal K$ for some  $k\times 2k$-matrix $\mathcal K$. The proof of \cite[Thm.~A.1]{BasCor16}  shows that $\ker \mathcal K$ is contained in the nonpositive isotropic cone of the quadratic form \eqref{eq:quadglobal}.  Also the boundary condition~(2.2) in~\cite{Eng13} can be written as $\gamma(u)\in \ker \mathcal K$, and \eqref{eq:basis} is easily seen to be satisfied. Hence, the generation result in~\cite[Cor.~2.2]{Eng13} follows from Theorem~\ref{prop:main2}; this shows in particular that the well-posedness results for the transport equations in~\cite{KraSik05,MatSik07} are special cases of our general theory. Let us further stress that, as in 
\cite[Cor.~3.8]{Eng13}, we can chose  different  $Q_\me$ for $M_\me$ and -$M_\me$
to obtain generation of a group. We will review further, more advanced examples in Section~\ref{sec:examples}.
\end{rem}

\begin{rem} {(Global boundary conditions revisited)
  It turns out that the global condition \eqref{eq:basis_global} 
may fail, see \autoref{ex:posit} below.  But if all matrices $M_\me$ are diagonal or spatially constant and symmetric positive definite, 
we can replace the surjectivity of $\mathcal{A}$ by the surjectivity of
$\lambda \mathbb{I}-\mathcal{A}$ for $\lambda>0$ large enough. Indeed, using the same approach as in \autoref{prop:main1}, 
for any $\mf\in \bL^2(\mathcal G)$, we  look for a solution $u\in D(\mathcal A)$ of
\be\label{eq:sergepb1lambda}
\lambda u'_\me-M_\me(x)u'_\me(x)=\mf_\me(x)\quad \hbox{for $x\in (0,\ell_\me)$ and all } \me\in \mE.
\ee
By our assumptions, such a solution is given by
\[
u_\me(x)=e^{\lambda P_\ell(x)}K_\me-\int_0^x e^{\lambda (P_\ell(x)-P_\ell(y))}M^{-1}_\me(y)\mf_\me(y)\,dy,
\hbox{for all } x\in [0,\ell_\me],  \hbox{for all } \me\in \mE,
\]
with $K_\me\in \C^{k_\me}$ to be fixed, where 
\[
P_\ell(x):=\int_0^xM^{-1}_\me(y)\,dy,\quad x\in [0,\ell_\me], \ \me\in \mE.
\]
Obviously,  $P_\ell(0)=0$
and $P_\ell(\ell_\me)=\int_0^{\ell_\me}M^{-1}_\me(y)\,dy$ is symmetric positive definite. Hence, using the notation from \autoref{rk:globalbc},
the boundary condition
\eqref{eq:mainbc_global}
is equivalent to
\be\label{eq:sergewplambda1bis}
 \left(K, E_\lambda K\right)^\top\cdot \overline{\mW^{(i)}}=
  \left(\left(0 \right)_{\me\in\mE}, 
\left(\int_0^{\ell_\me} e^{\lambda (P_\ell({\ell_\me})-P_\ell(y))}M^{-1}_\me(y)\mf_\me(y)\,dy\right)_{\me\in\mE}\right)^\top
\cdot 
\overline{\mW^{(i)}}\quad  \hbox{for all } i\in I.
\ee
where $K=(K_\me)_{\me\in \mE}$ and
$E_\lambda=\diag \left(e^{\lambda P_\ell(\ell_\me)}\right)_{\me\in\mE}$. Setting $L=E_\lambda K$, we have equivalently
\be\label{eq:sergewplambda1ter}
 \left( E_{-\lambda} L, L\right)^\top\cdot \overline{\mW^{(i)}}=
  \left(\left(0 \right)_{\me\in\mE}, 
\left(\int_0^{\ell_\me} e^{\lambda (P_\ell(\ell_\me)-P_\ell(y)}M^{-1}_\me(y)\mf_\me(y)\,dy\right)_{\me\in\mE}\right)^\top
\cdot 
\overline{\mW^{(i)}}\quad \hbox{for all } i\in I.
\ee
Hence if $\dim Y^\perp = k$, we find a $k\times k$ linear system and a unique solution $L$ and then $K$ exists if and only if the determinant of the associated matrix $A_\lambda$ is different from zero.
But if for any vectors $(y, z)^\top\in \mathbb{C}^{2k}$, with $y, z \in \mathbb{C}^{k}$
we set 
\[
 {\Pi_0}(y, z)^\top=y\quad \hbox{and}\quad  {\Pi_\ell}(y, z)^\top=z,
\]
we see that 
\[
 \left( E_{-\lambda} L, L\right)^\top\cdot \overline{\mW^{(i)}}=
 E_{-\lambda} L\cdot \overline{{\Pi_0}\mW^{(i)}}+
L\cdot \overline{{\Pi_\ell}\mW^{(i)}}=L\cdot \overline{ E_{-\lambda}^\star {\Pi_0}\mW^{(i)}+ {\Pi_\ell}\mW^{(i)}},
\]
where $E_{-\lambda}^\star=\diag \left(e^{-\lambda {P_\ell}(\ell_\me)^\star}\right)_{\me\in\mE}$.
From this expression we see that $A_\lambda$ can be split up as
\[
A_\lambda= R_\lambda+A_\infty,
\]
where  $A_\infty$ is the limit of $A_\lambda$ as $\lambda$ goes to infinity and is simply the matrix  whose $i$th row is $({\Pi_\ell}\mW^{(i)})^\top$.
Further, since ${P_\ell}(\ell_\me)^\star$ are symmetric positive definite, one can easily see that
\[
\lim_{\lambda\to \infty} \|R_\lambda\|=0.
\]
Hence, if we assume that $A_\infty$ is invertible, $A_\lambda$ will be invertible as well for $\lambda$ large enough, which allows to conclude the surjectivity of $\lambda \mathbb{I}-\mathcal{A}$ for $\lambda>0$ large enough. In conclusion, as $A_\infty$  is invertible if and only if $\dim  \Pi_\ell Y^\perp = k$,
 maximality holds  if 
\begin{equation}
\label{eq:basis_global_infty}
\dim Y^\perp = k =  \dim  \Pi_\ell Y^\perp.
\end{equation}
This actually means that maximality holds as soon as maximality holds if we replace the boundary condition 
\eqref{eq:mainbc_global} by the one
\[
\underline{u (\ell)}\in \Pi_\ell Y^\perp,
\]
at the endpoint $\ell$ only; in particular, dropping any condition at $0$. 
A local version of this approach can be stated, we let the details to the reader.
}
\end{rem}
 
\begin{rem}\label{rem:control}
Our well-posedness results can be also regarded as a first step towards applications to control theory. For example, \cite[Prop. VI.8.5]{EngNag00}  prevents the control system 
\begin{equation}\tag{$\Sigma(\mathcal A,\mathcal B)$}
\left\{
\begin{split}
\dot{u}_\me(t) &= M_\me u'_\me(t) +N_\me u_\me(t) + B_\me y_\me(t),\qquad  t\ge 0, \me\in\mE,\\
u_\me(0)&=f_\me
\end{split}
\right.
\end{equation} from being exactly 2-controllable for any control operator $\mathcal B=(B_\me)_{\me\in\mE}$ mapping the vertex space $\C^{{|\mV|}}$ to $L^2(\mathcal G)$, for any finite time horizon $t<\infty$. However, \cite[Cor. VI.8.11]{EngNag00} can be invoked to show that, provided the assumptions of Theorem~\ref{prop:main2}.(2) are satisfied and hence $\mathcal A$ generates a contractive semigroup, the control system $(\Sigma(\mathcal A,\mathcal B))$ will be approximately 2-controllable if and only if $\bigcap_{n\in\mathbb N} \ker \mathcal B^\ast (\Id-\mathcal A^\ast)^n = {0}$, where $\mathcal A^*$ is the adjoint of $\mathcal A$ given in Lemma~\ref{lem:adjoint} and $\mathcal B:\C^{{|\mV|}} \to L^2(\mathcal G)$ is any control operator.

Studying processes on a metric graph, boundary control problems with control imposed in the vertices are of special interest. A semigroup approach to such problems that enables an explicit description of the associated (even exact or positive) reachability space as presented in \cite{EKNS08,EKKNS10,EK17} could be adapted for the present situation. 

Further, as in \cite{Nic:2017} in case of generation of a contraction semigroup, some stabilization results can be considered using different approaches, like the use of observability estimates \cite{dagerzuazua,AMbook}, frequency domain approach \cite{pruss:84,borichev:10}, or port-Hamiltonian techniques \cite{JacZwa12}.
\end{rem}

Let us continue by studying qualitative properties of the semigroup generated by $\mathcal A$. In particular, let $C$ be a closed and convex subset of $\C$ and write {
\begin{equation}\label{eq:lpc}
\bL^2(\mathcal G;C):=\{u\in \bL^2(\mathcal G): u_\me(x)\in C^{k_\me}\ \hbox{ for a.e. }x\in (0,\ell_\me)\hbox{ and all }\me \in\mE\}
\end{equation}
A semigroup $(T(t))_{t\ge 0}$ on $L^2(\mathcal G)$  is called \emph{real} (resp., \emph{positive})
if each operator $T(t)$ leaves $\bL^2(\mathcal G;\R)$ (resp., $\bL^2(\mathcal G;\R_+)$)
invariant.} Moreover, for  a closed and convex subset $K$ of a Hilbert space $H$
 {with inner product $ \langle \cdot, \cdot \rangle_H$ and associated norm $\|\cdot\|_H$},  the \emph{minimizing projector  $P_K$ onto $K$} assigns to each $u\in H$ the unique element $P_K u\in K$ satisfying 
 \[ \| u - P_K u \|_{ H} = \min \{\| u - w \|_{H} : w\in K\}\]
 or, equivalently (see \cite[Thm. 5.2]{Bre10}), $P_K u=z\in K$ is the unique element in $K$ such that
 \begin{equation}\label{eq:P_K}
\Re \langle w-z, u-z \rangle_{H}  \le 0\qquad \hbox{for all }w\in K.
 \end{equation}

We will use a generalization of Brezis' classical result for invariance 
 under the semigroup generated by a subdifferential. In the linear case,~\cite[Thm.~2.4]{Yok01} can be formulated as follows.
\begin{lemma}\label{lem:yok}
Let $H$ be a complex Hilbert space {with inner product $ \langle \cdot, \cdot \rangle_H$ and associated norm $\|\cdot\|_H$}, $K$  a closed and convex subset of $H$ and $P_K$ the minimizing projector onto $K$.
Let $A$ be an $\omega$-quasi-m-dissipative operator on $H$ for some $\omega\in \mathbb R$ and $(T(t))_{t\ge 0}$ the strongly continuous semigroup generated by $A$. Then $K$ is invariant under $(T(t))_{t\ge 0}$ if and only if 
\begin{equation}\label{eq:yok}
\Re \langle Au,u-P_K u\rangle_{H}  \le \omega \|u-P_K u\|^2_{ H} \qquad \hbox{for all }u\in D(A).
\end{equation}

If in particular $\omega=0$, i.e., $A$ is dissipative, and $P_K u\in D(A)$ for all $u\in D(A)$, then the invariance of $K$ under $(T(t))_{t\ge 0}$ is equivalent to
\[
\Re \langle A P_K u,u-P_K u\rangle_{H}  \le 0\qquad \hbox{for all }u\in D(A).
\]
\end{lemma}
We can thus describe further properties of the semigroup generated by $\mathcal A$ in terms of the matrices $Q_\me,M_\me,N_\me$, and the boundary conditions. 
We are interested in the convex subsets of the form  $K=\bL^2(\mathcal G;C)$ where  $C\subset \C$ is a closed interval, e.g., $C=\R$ or $C=\R_+$. 
Let us first  relate the minimizing projector $P_K^Q$ with respect to the inner product \eqref{eq:prod?} on ${\bf L}^2(\mathcal G)$ to the minimizing projector $P_K$ with respect to the standard inner product
\begin{equation}\label{eq:st-prod}
\langle u,v\rangle:=\sum_{\me\in\mE}\int_0^{\ell_\me}
 u_\me (x)\cdot \bar v_\me(x)\,dx, \qquad 
u,v\in \bL^2(\mathcal G).
\end{equation}

\begin{lemma}\label{lem:project}
Assume $Q_\me^{\frac{1}{2}}(x)$ to be bijective for all $\me\in\mE$ and all $x\in [0,\ell_\me]$  as a map on $C$, i.e., $Q_\me^{\frac{1}{2}}(x)(C)=C$.
Then
the  minimizing projector   $P^Q_K$ with respect to the inner product \eqref{eq:prod?} onto $K=\bL^2(\mathcal G;C)$
is given by
\begin{equation}\label{eq:PQK}
P^Q_K =Q^{-\frac{1}{2}} P_K Q^{\frac{1}{2}} 
\end{equation}
where $Q:=\diag(Q_\me)_{\me\in \mE}$ is a block diagonal matrix and $P_K$ is the minimizing projector with respect to the standard inner product  \eqref{eq:st-prod}.
\end{lemma}
\begin{proof} 	
By \eqref{eq:P_K}, $P^Q_K u=:z$ is the unique element in $K$ such that
\begin{equation}\label{eq:sergepos1}
\sum_{\me\in\mE}\int_0^{\ell_\me}
Q_\me(x) (w_\me (x)-z_\me(x))\cdot (\bar u_\me (x)- \bar z_\me(x))\,dx\leq 0 \qquad\hbox{for all }
w\in K.
\end{equation}
As $Q_\me(x)$ is symmetric positive definite, it admits a square root $Q_\me(x)^{\frac{1}{2}}$ which is still 
symmetric positive definite (and is of class $H^1$ as a function of $x$). Hence~\eqref{eq:sergepos1} is equivalent to
\begin{equation}\label{eq:sergepos2}
\sum_{\me\in\mE}\int_0^{\ell_\me}
Q_\me(x)^{\frac{1}{2}} (w_\me (x)-z_\me(x))\cdot Q_\me(x)^{\frac{1}{2}}(\bar u_\me (x)-\bar z_\me(x))\,dx\leq 0 \qquad 
\hbox{for all } w\in K.
\end{equation}
Since $Q_\me^{\frac{1}{2}}$ is bijective on $C$, by setting $v_\me:=Q_\me^{\frac{1}{2}}u_\me$ and  $\tilde w_\me:=Q_\me^{\frac{1}{2}}w_\me$, 
we may equivalently re-write~\eqref{eq:sergepos2} as
\[ \Re \langle \tilde w-Q^{\frac{1}{2}} z, v-Q^{\frac{1}{2}} z \rangle \le 0\qquad \hbox{for all }\tilde w\in K.
\]
By \eqref{eq:P_K} we obtain
$Q^{\frac{1}{2}} z = P_K v$ and \eqref{eq:PQK} follows.
\end{proof}

In many applications $Q_\me$, and thus also $ Q_\me^{\frac{1}{2}} $ and $Q_\me^{-\frac{1}{2}}$, are  real-valued. 
In such a  case, \eqref{eq:PQK} for  $C=\R$ and hence $K=L^2(\mathcal G;\R)$ yields
\begin{equation}\label{eq:qrr}
P^Q_{K} u=Q^{-\frac{1}{2}} \Re\left( Q^{\frac{1}{2}} u\right) = \Re u.
\end{equation}

\begin{prop}\label{prop:real}
Under the assumptions of \autoref{prop:main2}, let
\begin{equation}\label{eq:realcond-0}
\Re \xi\in  \bigoplus_{\mv\in\mV} Y_\mv \hbox{ for all } \xi\in  \bigoplus_{\mv\in\mV
} Y_\mv
\end{equation}
and the matrix-valued mapping $Q_\me $ be real-valued for all $\me\in \mE$. Then  the semigroup generated by $\mathcal A$ is real if and only if 
the matrix-valued mappings $M_\me,N_\me$ are real-valued for all $\me\in\mE$.
\end{prop}
\begin{proof}
We use \autoref{lem:yok} for the convex subset of real-valued functions $K=\bL^2(\mathcal G,\R)$ and the projector $P^Q_K u = \Re u$ obtained in \eqref{eq:qrr}. Also observe that we can without loss of generality assume 
$\omega=0$, since reality of a semigroup is invariant under scalar perturbations of its generator. Now, it follows from~\eqref{eq:realcond-0}  that $P^Q_K u\in D(\mathcal A)$ whenever $u\in D(\mathcal A)$.  We deduce that reality of the semigroup is equivalent to
\[
\Re ( \mathcal A \Re u, \imath \Im u)\le 0\quad \hbox{for all }u\in D(\mathcal A).
\]
By applying the same trick as in the proof of~\cite[Prop.~2.5]{Ouh05}, that is by plugging $-\Re u + \imath \Im u$ into the above inequality,  we obtain
\[
 ( \mathcal A \Re u, \Im u)=\sum_{\me\in\mE}\int_0^{\ell_\me}
\left( \left( Q_\me M_\me \right)\frac{d\Re u_\me }{d x}+  \left( Q_\me N_\me \right) \Re u_\me \right)\cdot  \Im u_\me \,dx \in \R \quad \text{for all }u\in D(\mathcal A),
\]
which by a simple localization argument is in turn equivalent to
\begin{equation}\label{eq:cond-real-ima}
(Q_\me M_\me) \frac{du_\me}{dx} + (Q_\me N_\me)  u_\me\hbox{ is real-valued  for all real-valued }u\in D(\mathcal A) \hbox{ and all }\me\in\mE.
\end{equation}
The conclusion then follows from \autoref{lem:cond4.4} below that yields an equivalent, but easier to check, formulation of \eqref{eq:cond-real-ima}.
\end{proof}

\begin{lemma}\label{lem:cond4.4}
Assume  the matrix-valued mapping
$Q_\me$ to be real-valued for all $\me\in \mE$. Then \eqref{eq:cond-real-ima} holds if and only if the matrix-valued mappings $M_\me,N_\me$ are real-valued for all $\me\in\mE$.
\end{lemma}
\begin{proof}
As all $Q_\me $ are real-valued, it suffices to show that~\eqref{eq:cond-real-ima} holds if and only if 
the matrices $Q_\me M_\me$ and $Q_\me N_\me$ are real for all $\me\in\mE$.
As this second property is clearly sufficient for~\eqref{eq:cond-real-ima} to hold,
it suffices to prove the converse implication.
For that purpose, fix a real-valued function $\varphi\in \mathcal{D}(\mathbb{R})$
with a support included into $[-1,1]$ and such that
$\varphi'(0)=1$. 
Now fix one edge $\me\in\mE$, one point $x_0\in (0,\ell_\me)$ and one $i\in\{1,\dots,  k_\me\}$. Then
 for all $n$ large enough, define $u_n$ as follows: 
 $u_{n,\me'}=0$, for all edges ${\me'}\ne \me$
 and
 $u_{n,\me}= \varphi(n(x-x_0)) e_i$,
 where $e_i=(\delta_{ij})_{j=1}^{k_\me}$ is the $i$-th vector of the canonical basis of $\R^{k_\me}$.
 The parameter $n$ is chosen large enough so that the support of $u_{n,\me}$ is included into 
 $(0,\ell_\me)$ so that 
  $u_n\in D(\mathcal A)$ (and is real valued). Taking this function in~\eqref{eq:cond-real-ima}, we find
 \[
n \varphi'(n(x-x_0))  (Q_\me M_\me)(x) e_i + \varphi(n(x-x_0)) (Q_\me N_\me)(x) e_i  \in \R^{k_\me}\quad \hbox{for all } x\in (0,\ell_\me).
\]
By evaluating this expression at $x_0$, dividing by $n$ and  and letting $n$ goes to infinitiy, we find that
 \[
 (Q_\me M_\me)(x_0) e_i  \in \R^{k_\me}.
\]
In other words, $Q_\me M_\me$ is real-valued. Once this property holds, ~\eqref{eq:cond-real-ima}  reduces to
\[
 (Q_\me N_\me)  u_\me\hbox{ is real-valued  for all real-valued }u\in D(\mathcal A) \hbox{ and all }\me\in\mE.
\]
This directly implies that  $Q_\me N_\me$ is real-valued 
since $\bigoplus_{\me\in\mE}(\mathcal{D}(0,\ell_\me))^{k_\me}$ is included into $D(\mathcal A)$.
\end{proof}

Before going on let us mention that  condition \eqref{eq:realcond-0}   can be simplified in the following way.
\begin{lemma}\label{conj:noreal}
The condition \eqref{eq:realcond-0} holds if and only if $Y_\mv$, for each $\mv\in\mV$, is spanned by vectors with real entries only.
\end{lemma}
\begin{proof}
Let $Y_\mv$ be spanned by entry-wise real vectors $y_1,\ldots,y_n$ and let $y\in Y_\mv$. Then there exist $\alpha_1,\ldots,\alpha_n\in \C$ such that $y=\sum_{i=1}^n \alpha_i y_i$. Because all entries of each $y_i$ are real, it follows that $\Re y$ is again a linear combination of these basis vectors,
\[
\Re y=\sum_{i=1}^n (\Re \alpha_i)y_i \in Y_\mv.
\]
For the converse first note that \eqref{eq:realcond-0}  implies that for any $y\in Y_\mv$ also $\Re y, \Im y \in Y_\mv$. If now $y_1,\ldots,y_n$  is any basis of $Y_\mv$, then the entry-wise real vectors  $\Re y_i, \Im y_i$, $ i=1,\dots n,$ span $Y_\mv$.
\end{proof}

Let us now continue with the study of positivity of the semigroup.
Here, we will without loss of generality restrict ourselves to real Hilbert space $L^2(\mathcal G, \mathbb R)$.
We first notice that  each $Q_\me(x)^{\frac{1}{2}}$ is a lattice isomorphism  if and only if $Q_\me(x)^{\frac{1}{2}}$ and hence $Q_\me(x)$ are diagonal.
\begin{lemma}\label{lem:diagonal}
Let $P$ be a real $k\times k$ ($k\geq 1$) matrix that is symmetric and positive definite.
 Then $P$ is a lattice isomorphism, i.e., $P(\R_+^k)=\R_+^k$,
 if and only if $P$ is diagonal.
\end{lemma}
\begin{proof}
The diagonal character of $P$ added with its positive definiteness trivially imply that $P(\R_+^k)=\R_+^k$. Hence, we only need to prove the converse implication.
For that purpose denote by
$p_{ij}, 1\leq i,j\leq k$ (resp. $q_{ij}, 1\leq i,j\leq k$) the entries of $P$ (resp. $P^{-1}$).
Now notice that from our assumption directly follows $p_{ij}, q_{ij}\geq 0$, for all $i,j$. 
Moreover, for all $i,j$ with $i\ne j$, we have
\[
\sum_{\ell=1}^k p_{i\ell} q_{\ell j}=0,
\]
or, equivalently,
\[
p_{i\ell} q_{\ell j}=0 \quad \text{for all } \ell=1, \dots, k.
\]
Taking $\ell=j$, we find that
$p_{ij}=0$ since the diagonal entries of $P^{-1}$ are strictly positive.
This shows that $P$ is diagonal as requested.
\end{proof}
  
We continue by applying \autoref{lem:yok} to the convex subset of positive-valued functions $\bL^2(\mathcal G,\R_+)$.
In the following we will adopt the notation
\[
{\bf 1}_{\{u\ge 0\}}:=({\bf 1}_{\{u_{\me}\ge 0\}})_{\me\in\mE},\quad {\bf 1}_{\{u\le 0\}}:=({\bf 1}_{\{u_{\me}\le 0\}})_{\me\in\mE},
\]
for the vector-valued characteristic function of the nonnegative and nonpositive support of $u$, respectively. Here each ${\bf 1}_{\{u_{\me}\ge 0\}}$ is the diagonal $k_\me\times k_\me$ matrix associated to $u_{\me}=\left(u_\me^{(1)}, \cdots, u_\me^{(k_\me)}\right)^\top$,
\[
{\bf 1}_{\{u_{\me}\ge 0\}}:=
\diag\left({\bf 1}_{\{u_\me^{(1)}\ge 0\}}, \cdots, {\bf 1}_{\{u_\me^{(k_\me)}\ge 0\}}\right).
\]
Recall also,  that for any $f\in H^1$ one can write
\[
f^+ =  {\bf 1}_{\{f\ge 0\}} f,\quad  {-}f^- = {\bf 1}_{\{f\le 0\}} f,\quad\text{ and }\quad 
(f^+)'= {\bf 1}_{\{f\ge 0\}}  f'.
\]

\begin{prop}\label{prop:posit}
In addition to the assumptions of \autoref{prop:main2},  let $M_\me, N_\me, Q_\me $  be real-valued for all $\me\in\mE$, $Q_\me(x)$ be diagonal 
for $\me\in\mE$ and all $x\in [0,\ell_\me]$, and
\begin{equation}\label{eq:cond-positive-0}
\xi^+ \in  {\bigoplus_{\mv\in\mV} Y_\mv } \hbox{ for all } \xi\in  {\bigoplus_{\mv\in\mV} Y_\mv }.
\end{equation}
Then the semigroup generated  by $\mathcal A$  on $L^2(\mathcal G, \mathbb R)$ is positive if and only if
for all $\me\in\mE$ the matrices 
$M_\me (x)$ 
are diagonal for all $x\in [0,\ell_\me]$, and all off-diagonal entries of the matrices
$N_\me (x)$ 
are nonnegative for a.e.\ $x\in [0,\ell_\me]$.
\end{prop}

\begin{proof} 
By \autoref{lem:diagonal} we may apply \autoref{lem:yok} to the closed convex subset $K=\bL^2(\mathcal G,\R_+)$. Therefore,  by diagonality and positivity of $Q_\me(x)$, the minimizing projector $P_K $ given by \eqref{eq:PQK}  takes the simpler form
\begin{equation}\label{eq:defPK}
P^Q_K u =Q^{-\frac{1}{2}}  \left( Q^{\frac{1}{2}} u \right)^+  = u^+.
\end{equation}
Since real scalar perturbations of the generator do not affect positivity of the semigroup, we may assume that $\omega = 0$. As \eqref{eq:cond-positive-0}  yields $P_K u\in D(\mathcal A)$ for all $u\in D(\mathcal A)$, the semigroup is thus positive if and only if
 \[( {\mathcal A}u^+,(u-u^+)) \le 0\qquad \hbox{for all }u\in D(\mathcal A).
\]
By applying \eqref{identityA} and~\eqref{eq:defPK} we obtain the equivalent condition 
\begin{equation}\label{eq:cond-positive}
\sum_{\me\in\mE}\int_0^{\ell_\me}
Q_\me \left(  M_\me  (u_\me^+)' +  N_\me u_\me^+ \right)\cdot 
( u_\me- u_\me^+) \,dx \le 0\quad\hbox{for all }u\in D(\mathcal A),
\end{equation}
which we rewrite using the characteristic functions as 
\begin{equation}\label{eq:cond-positive-main}
\sum_{\me\in\mE}\int_0^{\ell_\me}
\left( Q_\me  M_\me  {\bf 1}_{\{u_\me\ge 0\}}  u'_\me + Q_\me  N_\me  {\bf 1}_{\{u_\me\ge 0\}}  u_\me \right)\cdot 
(  {\bf 1}_{\{u_\me\le 0\}} u_\me) \,dx   \le  0\quad\hbox{for all }u\in D(\mathcal A).
\end{equation}
The conclusion then follows from \autoref{lem:serge-diagon} below that furnishes an equivalent, but easier to check, formulation of \eqref{eq:cond-positive-main}.
\end{proof}

\begin{rem} In the same way as in \autoref{conj:noreal}, by replacing the real and imaginary by the positive and negative part, respectively,  we can see that \eqref{eq:cond-positive-0} holds if and only if, for each $\mv\in\mV$, $Y_\mv$ is spanned by vectors with positive entries only.
\end{rem}

Let us prove the last step that is still missing in the proof of \autoref{prop:posit}.

\begin{lemma}\label{lem:serge-diagon}
Let   $Q_\me(x)$ be diagonal 
for all $\me\in\mE$ and  for all $x\in [0,\ell_\me]$.
Then \eqref{eq:cond-positive-main} holds if and only if 
for all $\me\in\mE $ the matrices 
$M_\me (x)$ 
are diagonal for all $x\in [0,\ell_\me]$, and all off-diagonal entries of the matrices
$N_\me (x)$ 
are nonnegative for a.e.\ $x\in [0,\ell_\me]$.
\end{lemma}

\begin{proof}
First observe that since $Q_\me(x) $ are all diagonal matrices with strictly positive diagonal elements, $M_\me (x)$ is diagonal if and only if  $Q_\me (x)M_\me (x)$ is diagonal and $N_\me (x)$  has  {nonnegative} off-diagonal entries if and only if  the same holds for $Q_\me(x)N_\me (x)$.
Obviously, these properties imply that
\[
\begin{split}
\sum_{\me\in\mE}\int_0^{\ell_\me}
 \left( Q_\me M_\me  {\bf 1}_{\{u_\me\ge 0\}}  u'_\me  \right)\cdot 
(  {\bf 1}_{\{u_\me\le 0\}} u_\me) \,dx =& 0\quad \hbox{and}\\
\sum_{\me\in\mE}\int_0^{\ell_\me}
 \left( Q_\me N_\me  {\bf 1}_{\{u_\me\ge 0\}}  u_\me \right)\cdot 
(  {\bf 1}_{\{u_\me\le 0\}} u_\me) \,dx   \le  & 0.
\end{split}
\]
yielding  \eqref{eq:cond-positive-main}.
 
 Conversely, assume that  \eqref{eq:cond-positive-main} holds.
Introduce function
 $\varphi\colon \R\to\R_+$,
 \[
 \varphi(y) :=
 \begin{cases}
 y+1 & \hbox{ if } y\in [-1,0],\\
 1- y & \hbox{ if } y\in [0,1],\\
 0& \hbox{ else, }
 \end{cases}
 \]
 and define 
  \[
  \psi_-:\R\to \R_+, y\mapsto \varphi(2y+1), \quad \psi_+:\R\to \R_+, y\mapsto \varphi(2y-1).
  \]
  Notice that 
  \begin{equation}\label{eq:serge/8/10:1}
  \int_{-1}^1 \varphi'(y) \psi_-(y)\,dy>0, \quad  \int_{-1}^1 \varphi'(y) \psi_+(y)\,dy<0.
  \end{equation}
 Now fix one edge $\me\in\mE$, one point $x_0\in (0,\ell_\me)$ and $1\leq i,j\leq k_\me$, $i \ne j$. Then
 for all $n$ large enough, define $v_{n,\pm}$ as follows: 
 $v_{n,\pm,\me'}=0$, for all edges ${\me'}\ne \me$
 and
 $v_{n,\pm, \me}= (v_n^{(1)}, \cdots, v_n^{(k_\me)})^\top$
with only two nonzero entries: $v^{(i)}, v^{(j)}$ defined by
\begin{eqnarray*}
v_n^{(i)}(x)=\varphi(n(x-x_0)) \hbox{ and } v_n^{(j)}(x)=-\psi_\pm(n(x-x_0)).
\end{eqnarray*}
 The parameter $n$ is chosen large enough so that the support of $v_n^{(i)}, v_n^{(j)}$ is included into 
 $(0,\ell_\me)$ and
$u_{n,\pm} \in D(\mathcal A)$, as it vanishes in a neighborhood of each vertex. Plugging this test-function into \eqref{eq:cond-positive-main}, dividing the expression by $n$ and letting $n$ goes to infinitiy, we find that
 \[
 \left( Q_\me (x_0) M_\me (x_0) \right)_{ji} \int_{-1}^1 \varphi'(y) \psi_\pm(y)\,dy\geq 0.
 \]
 With the help of \eqref{eq:serge/8/10:1}, we deduce that $Q_\me(x_0) M_\me(x_0)$
 is diagonal.
 
Taking this into account,   \eqref{eq:cond-positive-main} reduces to
 \begin{equation}\label{eq:cond-positive-main-reduced}
\sum_{\me\in\mE}\int_0^{\ell_\me}
 \left( Q_\me N_\me  {\bf 1}_{\{u_\me\ge 0\}}  u_\me \right)\cdot 
(  {\bf 1}_{\{u_\me\le 0\}} u_\me) \,dx    \le  0.
\end{equation}
Now fix one edge $\me\in\mE$, and
$1\leq i<j\leq k_\me$. Define $v$ as follows: 
 $v_{\me'}=0$, for all edges ${\me'}\ne \me$
 and
 $v_{\me}= (v^{(1)}, \cdots, v^{(k_\me)})^\top\in  \mathcal{D}(0,\ell_\me)^{k_\me}$
with only two nonzero entries: $v_n^{(i)}, v_n^{(j)}$ defined by
\begin{eqnarray*}
v^{(i)}(x)= - |\chi(x)| \hbox{ and } v^{(j)}(x)=|\chi(x)|,
\end{eqnarray*}
with $\chi\in  \mathcal{D}(0,\ell_\me)$.
As before this function belongs to $D(\mathcal A)$, and with this choice of $v$ in 
\eqref{eq:cond-positive-main-reduced} we find that
\[
\int_0^{\ell_\me}
(Q_\me N_\me)_{ij} \chi^2 \,dx    \ge   0.
 \]
 Since this holds for all $\chi\in  \mathcal{D}(0,\ell_\me)$, by {\autoref{l:appendixB1}}, we conclude
 that $(Q_\me (x)N_\me (x))_{ij}   \ge   0$ for all $i\ne j$ and all $x \in [0,\ell_\me]$.
\end{proof}

\begin{rem}
As a simple corollary, if $Q_\me(x)$, $M_\me(x)$, and $N_\me(x)$ are all diagonal, \eqref{eq:cond-positive-main} is automatically satisfied and we are left with~\eqref{eq:cond-positive-0}, a condition depending on the boundary conditions only, a result reminiscent of~\cite[Prop.~5.1]{CarMug09} in the case of \textit{parabolic} systems. 
\end{rem}

 \begin{exa}\label{ex:Dirichlet}
If Dirichlet conditions are imposed on all endpoints $\ell_\me$ (resp., on all endpoints 0) and $Q_\me(0) M_\me(0)$ is positive semidefinite (resp. $Q_\me(\ell) M_\me(\ell)$ is negative semidefinite) for all $\me$, then the corresponding nonpositive isotropic subspaces $Y_\mv$  are  isomorphic to a direct product of $k_\mv$ blocks of size ${k_\me}$ that are either zero or $\C^{k_\me}$. The same holds for subspaces $Y_\mv^\perp$ but with the opposite pattern. Since each edge $\me$ has exactly one initial endpoint 0 and one terminal endpoint $\ell_\me$, we get exactly one corresponding nonzero block ($\equiv \C^{k_\me}$)  in all the spaces $ \widetilde {Y_{\mv}^\perp}$, $\mv\in \mV$. Hence,
\[ \dim \sum_{\mv\in \mV} \widetilde {Y_{\mv}^\perp}= \sum_{\mv\in \mV}  \dim Y_\mv = \sum_{\me\in\mE} k_\me = k.\] 
Since also $\gamma_\mv(u) \in Y_\mv$ for all $ \mv\in \mV$, by \autoref{prop:main2}, $\mathcal A$ generates a strongly continuous quasi-contractive semigroup.
 \end{exa}

\begin{exa}\label{ex:posit}
The boundary conditions considered in \autoref{ex:Dirichlet} do \textit{de facto} turn the network into a collection of decoupled intervals. 
Instead, we now use the setting described  in \autoref{rk:globalbc} 
and  impose conservation of mass conditions, which result in a global boundary condition $u(\ell)=Wu(0)$ for some column stochastic block  diagonal matrix $W$ with $|\mE |$ blocks of sizes $k_\me\times k_\me$. We have
\[
\begin{split}
Y&=\left\{\left(\alpha_{\me_1},\alpha_{\me_2},\ldots,\alpha_{\me_{|\mE|}},\sum_{\me\in\mE}w_{1\me}\alpha_{\me},\ldots,\sum_{\me\in\mE}w_{|\mE|\me}\alpha_{\me}\right)^\top: \alpha_{\me_i}\in \C^{k_\me}, \ \me_i\in\mE \right\}\\
&=\left\{\left(\alpha,W\alpha\right)^\top: \alpha\in \C^k\right\}  = \ker \begin{pmatrix}-W & I\\ \end{pmatrix},
\end{split}
\] 
hence condition \eqref{eq:basis_global} {is not satisfied: indeed, 1 is an eigenvalue of $W$ and therefore its associated eigenvector $\alpha_0$ yields a non trivial element
$\left(\alpha_0,\alpha_0\right)^\top$ in $Y\cap K$, hence \eqref{eq:basis_global} cannot hold due to \autoref{lem:equivSerge}. But if $M$ is diagonal or spatially constant and positive definite, then we can use condition \eqref{eq:basis_global_infty} that in the case of our choice of $Y$ reduces to the invertibility of the matrix $W$, {which agrees with \cite[Prop. 2.1]{KraPuc20}.}} 
Then, {under these additional assumptions,} $\mathcal A$  generates a contractive semigroup provided
\[
\left(-Q(0)M(0) +  W^\top Q(\ell)M(\ell)W\right)\alpha\cdot \overline{\alpha}\le 0\qquad \hbox{for all }\alpha\in \C^k.
\]
{where $Q(0):=\diag \left(
Q_\me(0) \right)_{\me\in\mE}$ (resp. $Q(\ell):=\diag \left(
Q_\me(\ell_\me) \right)_{\me\in\mE}$) and $M(0):=\diag \left(
M_\me(0)\right)_{\me\in\mE}$ (resp. $M(\ell):=\diag \left(
M_\me(\ell_\me) \right)_{\me\in\mE}$).}

If the matrices $Q_\me,M_\me$ are diagonal and real-valued, and if moreover  the diagonal and off-diagonal entries of $N_\me(x)$ are for all $\me\in\mE$ and a.e.\ $x\in [0,\ell_\me]$ real and nonpositive, respectively, then  the semigroup generated by $\mathcal A$ is positive. The case of spatially constant,  and diagonal matrices {with positive entries} on $L^1(\mathcal G)$ corresponds to flows in networks as studied in \cite{KraSik05}, see also  \cite{KraPuc20} and the references there. There, $W$ is the generalized adjacency matrix of the line graph.
\end{exa}

{
We conclude this section by elaborating on a comparison principle between semigroups. Let $A_1,A_2$ be two operators, each generating a positive semigroup -- say, $(T_1(t))_{t\ge 0},(T_2(t))_{t\ge 0}$. Then $(T_2(t))_{t\ge 0}$ is said to \emph{dominate} $(T_1(t))_{t\ge 0}$ if 
\[
T_1(t)f\le T_2(t)f\qquad \hbox{for all }f\ge 0 \hbox{ and all }t\ge 0.
\]

We can now formulate the following.
\begin{prop}
Let two operators $A_1,A_2$ with domain
\[
\begin{split}
D(A_1)&:=\{u\in D_{\max}:\gamma(u)\in Y_1\},\\
D(A_2)&:=\{u\in D_{\max}:\gamma(u)\in Y_2\}
\end{split}
\]
satisfy our standing assumptions as well as the assumptions of \autoref{prop:main2}. Let the semigroups generated by $A_1,A_2$,  say, $(T_1(t))_{t\ge 0},(T_2(t))_{t\ge 0}$, be both real and positive. Then $(T_2(t))_{t\ge 0}$ dominates $(T_1(t))_{t\ge 0}$ if 
$Y_1$ is an \textit{ideal} of $Y_2$, i.e., $Y_1\subset Y_2$ and
\[
0\le \xi \le \chi \hbox{ with }\xi\in Y_2\hbox{ and }\chi \in Y_1 \hbox{ implies }\xi\in Y_1.
\]
and additionally 
\[
( A_1 u,v) \le ( A_2 u,v)\quad \hbox{for all }u,v\in D(A_1)\hbox{ s.t. }u,v\ge 0.
\]
\end{prop}
\begin{proof}
First of all, observe that domination is not affected if $A_1,A_2$ are rescaled by the same real scalar $\omega$; hence, we can without loss of generality assume both $A_1,A_2$ to be dissipative. The proof can then be performed combining Lemma~\ref{lem:yok} and~\cite[Thm.~2.24]{Ouh05}.
\end{proof}
An interesting case arises when $A_1,A_2$ are defined by means of the same matrices $M,Q,N$, and only the boundary conditions are different, i.e., $Y_1\ne Y_2$. This is e.g.\ the case if $Y_1$ is the space spanned by $\Pi v_1,\ldots, \Pi v_m$, where $v_1,\ldots,v_m$ are the vectors spanning $Y_2$ and the projector $\Pi$ is a block diagonal operator matrix whose diagonal blocks are a zero matrix (of any size $\ge 1$) and an identity matrix  (of any size $\le 2|\mE|$). In the case of the transport equation on a network studied in~\cite{KraSik05}, this corresponds to comparing a given network with a new network with additional Dirichlet conditions in some vertices.

We may discuss in a similar way the issue of $L^\infty$-contractivity, i.e., the invariance of ${\bf L}^2(\mathcal G;C)$ for $C:=[-1,1]$ under the semigroup generated by $\mathcal A$. We omit the details.

\section{Examples}\label{sec:examples}

\subsection{Linearized Saint-Venant models}

Here we study a system where
the  linearized Saint-Venant  model   \eqref{systSaint-Venant} is considered on  all the $J$ edges of a network, and hence the unknown $(h,u)^\top$ is replaced by $(h_\me,u_\me)^\top_{\me\in\mE}$.
 While our approach applies to arbitrary networks
 and  variable  functions $H, V$ that  may differ across the edges, for the sake of simplicity  we here restrict to the case of constant (real) coefficients $H, V$
 that are independent of the edges, to $H>0$ (which is physically reasonable), and  to a star-shaped network { with $J$ edges, for some integer $J\geq 2$, see \autoref{fig:star}. More precisely, we let $\mE:=\{\me_1,\dots, \me_J\}$, and identify each edge $\me_j$ with $(0,\ell_j)$ (the parametrization of $\me_j$ is determined by the arrow in \autoref{fig:star}):} $\mv_0$ will correspond to the endpoint 
$\ell_1$ for $\me_1$ and to the   initial point $0$ for all other edges.
The external vertex of $\me_j$ will be denoted by $\mv_j$, $j=1,\dots,J$.

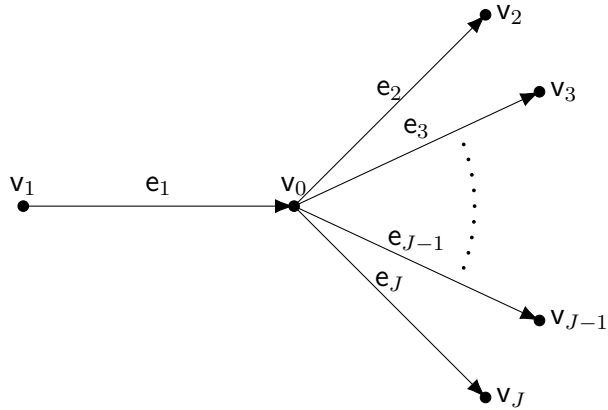
\begin{figure}[H]
\begin{center}
\begin{tikzpicture}
\foreach \x in {25,45,-25,-45}{
\draw[-{Latex[scale=1.5]}]  (\x:0cm) -- (\x:3.6cm);
\draw[fill] (\x:3.6cm) circle (2pt);
}
\draw[-{Latex[scale=1.5]}]  (180:3.6cm) -- (180:0cm);
\draw[fill] (180:3.6cm) circle (2pt);
\draw[fill] (0:0cm) circle (2pt) node[above]{$\mv_0$};
\node at (180:3.6) [anchor=south] {$\mv_1$};
\node at (180:1.8) [anchor=south] {$\me_1$};
\node at (45:3.6) [anchor=west] {$\mv_2$};
\node at (45:1.8) [anchor=south] {$\me_2$};
\node at (25:3.6) [anchor=west] {$\mv_3$};
\node at (25:1.8) [anchor=south] {$\me_3$};
\node at (-25:3.6) [anchor=west] {$\mv_{J-1}$};
\node at (-25:1.8) [anchor=south] {$\me_{J-1}$};
\node at (-45:3.6) [anchor=west] {$\mv_J$};
\node at (-45:1.8) [anchor=south] {$\me_J$};
\foreach \y in {20,15,10,5,0,-5,-10,-15,-20}{
\draw[fill] (\y:2.4cm) circle (.5pt);
}
\end{tikzpicture}
\caption{A star-shaped network with one incoming and $J-1$ outgoing edges.}\label{fig:star}
\end{center}
\end{figure}
Now we need to fix the boundary conditions at all vertices. According to our approach they are related to the operators $T_\mv$ defined in \eqref{eq:defTv}. In our case we have
(see \eqref{aq:defMQ-Saint-Venant})
\[
Q_\me  M_\me =
-
\left(\begin{array}{ll}
gV  & gH  \\ gH & HV 
\end{array}
\right)
\equiv: - B.
\]
Because the matrix $B$ is independent of $\me$,  and also  symmetric and invertible due to  condition \eqref{sn:15/4:1}, we thus have 
\begin{equation*}\label{eq:star-ext}
T_{\mv_1}=  B, \hbox{ and } T_{\mv_i}= - B \text{ for all } i\geq 2
\end{equation*}
at the external vertices $\mv_i$, $i\geq 1$, 
while at the interior vertex $\mv_0$  $T_{\mv_0}$ is a block diagonal matrix:
\begin{equation*}\label{eq:star-int}
T_{\mv_0}=\diag(B, -B, \dots, -B).
\end{equation*}

Now, let us notice that
the two eigenvalues $\lambda_\pm$ of $B$
satisfy 
\[
\lambda_+ \lambda_-=gH(V^2-gH).
\]
Hence, under the subcritical flow condition $gH-V^2>0$ (see \cite[p. 14]{BasCor16}), 
$\lambda_+$ and  $\lambda_-$ are of opposite sign. On the contrary,
under the supercritical flow condition $gH-V^2<0$, $\lambda_+$ and  $\lambda_-$ have the same sign;
but as $\lambda_++ \lambda_-= (g+H)V>0$, 
in this case they are both strictly positive.

Now, we distinguish between these two flow conditions.
\begin{enumerate}[1)]
\item If $gH-V^2<0$, 
then at the external vertices
the only choice for a totally isotropic subspace $Y_{\mv_i}$ associated with $T_{\mv_i}$ is $\langle (0,0)^\top\rangle$. This already yields $2J$ boundary conditions and there is no more freedom to manage
the internal vertex $\mv_0$. In other words, under this choice of boundary conditions the associated operator $\mathcal{A}$ cannot generate a group.

But we may hope for the generation of a semigroup.
Hence  as $T_{\mv_1}=  B$ has two positive eigenvalues we surely need to  impose that $Y_{\mv_1}=\{ (0,0)^\top\}$,   i.e.,
\begin{equation}\label{eq:stvn-bcv1}
h(\mv_1)=u(\mv_1)=0,
\end{equation}
while we are free to impose 
boundary conditions  or not at $\mv_j$, for all $j\geq 2$.
Since $T_{\mv_0}$ has 2 positive eigenvalues and $2(J-1)$ negative eigenvalues,
by \autoref{l:nonpositiveiso}, the maximal dimension of a subspace  $Y_{\mv_0}$ of the nonpositive isotropic cone
associated with $T_{\mv_0}$ is $2(J-1)$. 
Choosing $Y_{\mv_0}$ as the subspace associated with the negative eigenvalues leads to a decoupled system and is of {less} interest.
Letting instead
\begin{equation*}\label{eq:condtransmissionnatural}
Y_{\mv_0}=\lin\{(1,0,1,0, \dots, 1,0)^\top, (0,1,0,1, \dots, 0, 1)^\top\},
\end{equation*}
we observe that
\[
T_{\mv_0}\xi \cdot \bar \xi=(2-J) B(x,y)^\top\cdot (\bar x, \bar y)^\top\leq 0,
\]
for all $\xi\in Y_{\mv_0}$, i.e., $\xi=x(1,0,1,0, \dots, 1,0)^\top+y (0,1,0,1, \dots, 0, 1)^\top$ for some $x,y\in \mathbb{C}$, since $B$ is positive definite.
This choice   corresponds to the 
continuity of the water depth and the velocity at $\mv_0$, namely
\[
h_j(0)=h_1(\ell_1), \quad v_j(0)=v_1(\ell_1)\quad \text{for all }j\geq 2
\]
and yields $2(J-1)$ boundary conditions. They are complemented by the two conditions~\eqref{eq:stvn-bcv1} at $\mv_1$,   and by no conditions at $\mv_i$ for $i\ge 2$. This leads to $k =2J$ boundary conditions for which 
\[
T_\mv \xi\cdot \bar \xi\leq 0, \quad \text{for all }\xi\in Y_\mv,  \mv\in \mV.
\]
To conclude the generation of a semigroup
by \autoref{prop:main2} it suffices to notice that 
\eqref{eq:basis}  is valid because  $\widetilde {Y_{\mv_1}^\perp} = \mathbb{C}^2\times \{(0,0)\}^{J-1}$ is clearly independent of  $\widetilde {Y_{\mv_0}^\perp}$, 
while $\widetilde {Y_{\mv_j}^\perp}$ are trivial for all $j\ge 2$, so \eqref{eq:independent-Y} holds, and
 \[
\dim Y_{\mv_0}= 2,\quad
 \dim Y_{\mv_1}= 0,\quad  \hbox{ and }\quad  \dim Y_{\mv_j}= 2 \text{ for all }j\geq 2.
\]
Furthermore, by \autoref{lem:cond4.4} and \autoref{prop:real}, the semigroup is real.
Finally, by \autoref{lem:serge-diagon} condition  \eqref{eq:cond-positive-main} does not hold, hence the semigroup is not positive.
\item If we are in the subcritical case $gH-V^2>0$, then the eigenvalue $\lambda_+$
 (resp. $\lambda_-$ ) is positive (resp. negative). 
 
 Let us first analyze the possibility to have a group.
 In that case, by \autoref{l:iso}, 
 at any external vertex $\mv$ a totally isotropic subspace of $q_\mv$
 is of dimension at most one, while at the interior vertex $\mv_0$
 it is at most $J$. Let us present the following example.
 If $U_\pm$ is the  normalized eigenvector of $B$ associated with $\lambda_\pm$, then according to \eqref{eq:isotropic-vec},
 \[
 U=U_-+\imath\sqrt{\frac{\lambda_-}{\lambda_+}}U_+
 \]
 is an isotropic vector of the sesquilinear form associated with $B$.
 Therefore for all $j=1,\dots, J$, we take
 $Y_{\mv_j}$ as the vector space spanned by   $U$.
 We proceed similarly at $\mv_0$ by fixing
 $J$ isotropic vectors  constructed in the proof of \autoref{l:iso}.
 To have a coupling system, one possibility is the following one.
 We notice that
 the eigenvectors of $T_{\mv_0}$ associated with positive eigenvalues are
 \[
 U_1^+=(U_+^\top, 0, \dots, 0)^\top
 \]
 for the eigenvalue $\lambda_+$
 and
 \[
 U_2^+=(0, U_-^\top, 0, \dots, 0)^\top,\dots, U_J^+=(0, 0, \dots, 0, U_-^\top)^\top
 \]
 for the eigenvalue $-\lambda_-$. Similarly, the eigenvectors of $T_{\mv_0}$ associated with negative eigenvalues are
 \[
U_1^-=(U_-^\top, 0, \dots, 0)^\top
 \]
 for the eigenvalue $\lambda_-$
 and
 \[
U_2^-=(0, U_+^\top, 0, \dots, 0)^\top,\dots, U_J^-=(0, 0, \dots, 0, U_+^\top)^\top
 \]
 for the eigenvalue $-\lambda_+$.
We can now take  $Y_{\mv_0}$ as the vector space spanned by   
 $U_J^-+U_{1}^+$
 and by vectors
 $U_j^-+U_{j+1}^+$  for all $j=1,\cdots, J-1$.
 Then   $\widetilde{Y_{\mv_0}^\perp}$ is clearly independent of $\left\{\widetilde{Y_{\mv_i}^\perp}\;:\; i = 1,\dots,J\right\}$
which is also an independent set itself, and  \eqref{eq:basis}  holds since
  $\sum_{\mv\in\mV} \dim Y_\mv= 2J$.   Furthermore, each $Y_\mv$ is by construction a totally isotropic subspace associated with the quadratic form $q_\mv$, hence we are in the position to apply \autoref{prop:main1} and deduce that  the associated operator $\mathcal{A}$   generates a group, which is by \autoref{cor:main1} is unitary if and only if $V=0$.
 As before  the (forward) semigroup is real  and does not preserve positivity.

 In the subcritical case $gH-V^2>0$,  examples of boundary conditions leading to a semigroup can be
 easily built as before.
 \end{enumerate} 
 
\subsection{Wave type equations\label{exa:waves}}

Wave-type equations on  graphs have been intensively studied in the literature, let us mention \cite{Ali94,Klo12,LagLeuSch94,LeuSch89,Mug14,NicVal07}.  Here we focus on extending these results to rather general elastic systems modeled as 
\begin{equation}\label{eq:waves}
\ddot{u}_\me(t,x)=u_\me''(t,x)+\alpha_\me \dot{u}_\me'(t,x)+\beta_\me \dot{u}_\me(t,x)+\gamma_\me u'_\me(t,x),\qquad t\ge 0,\ x\in (0,\ell_\me),
\end{equation}
where $\alpha_\me \in C^1([0,\ell_\me])$
and $\beta_\me, \gamma_\me\in L^\infty (0,\ell_\me)$  are real-valued functions.  For the sake of simplicity, we hence restrict to stars as in Figure~\ref{fig:star}, which can be regarded as building blocks of more general networks.
It turns out that \eqref{eq:waves}  is equivalent to
\[
\dot{U}_\me= M_\me U_\me'+ N_\me U_\me,
\]
for the vector function $U_\me=(u_\me', \dot{u}_\me)^\top
$, where
\[
M_\me=\begin{pmatrix}
0 & 1 \\ 1 & \alpha_\me
\end{pmatrix},\qquad
N_\me=\begin{pmatrix}
 0 & 0\\
 \gamma_\me &  \beta_\me \\
\end{pmatrix}.
\]
As  $M_\me$ is symmetric,  \autoref{assum} are automatically satisfied by choosing  $Q_\me$ as the identity matrix.
As usual,  the boundary conditions at the vertices are related to the values of $M_\me$
at
the endpoints of the edge $\me$, that generically are given by
\[
M_\me(\mv)=\begin{pmatrix}
0& 1 \\ 1 & \alpha_\me(\mv)
\end{pmatrix},
\]
when $\mv$ is one of the endpoints of $\me$; hence $M_\me(\mv)$ has two real eigenvalues of opposite sign,  
\[
\lambda_\pm=\frac{1}{2}\left(\alpha_\me(\mv)\pm\sqrt{\alpha_\me(\mv)^2+4}\right).
\]

We are thus in the same situation as in the Saint-Venant model with the subcritical condition: we thus do not give any further details about the choice of boundary conditions at the interior vertices, since all ideas presented there carry over to the present case.
 Note that for an exterior vertex $\mv$, 
 Neumann condition
 \[
u_{\me}'(\mv)=0,
 \]
 can be equivalently described by means of the totally isotropic subspace
 $Y_\mv$ spanned by    $(0,1)^\top$. 
 In this case, $(0,1)^\top$ will be an isotropic vector if and only if $\alpha_\me(\mv)=0$.
 On the contrary, Dirichlet boundary condition
 at  an exterior vertex $\mv$ 
 \[
u_{\me}(\mv)=0,
 \]
 leads to
 $\dot{u}_{\me}(\mv)=0,$
and corresponds to the choice $Y_\mv$ spanned by   $(1,0)^\top$.
Our approach  also allows us to discuss absorbing boundary condition
(see~\cite{Chen:79} for instance))
\[
u_\me'(\ell_\me)=-\kappa \dot{u}_\me(\ell_\me)
\]
with $\kappa\in (0,\infty)$: indeed it then corresponds to the space spanned by
$U=(-\kappa, 1)^\top$,
hence
\[
M_\me(\mv) U \cdot \bar U=\alpha_\me(\mv)-  2\kappa \leq 0,
\]
that will be nonpositive as soon as $\kappa\geq \alpha_\me(\mv)/ 2$. 

 Let us finally notice that provided the boundary conditions  are nice enough
  to generate a group, it will be   unitary group if and only if $\gamma_\me=0$ 
  and $\beta_\me=\frac{\alpha'_{\me}}{2}$, for every edge $\me$.
  In case of a generation of a semigroup, it will be real provided \eqref{eq:realcond-0} holds,
  while  if all $\gamma_\me$ are nonpositive, it  will be never positive (this is in particular the case for the wave equation, as $\alpha_\me=\beta_\me=\gamma_\me=0$).

\subsection{Hybrid transport/string equations}

Network-like systems described by equations which are partially of diffusive and partially of transport type have been studied by Hussein and one of the  authors in~\cite{HusMug13}; characterization of the right transmission conditions leading to well-posedness has proved a difficult task. In the following we turn to the different but related task of connecting transport and wave equations; this can suggest natural ways of coupling first and second order differential operators.

{The simplest  toy model is to consider a  scalar wave equation set in an interval  $(0,\ell_2)$
and a scalar transport equation in $(0,\ell_1)$ coupled via their common endpoint that is assumed to be $0$.
This means that we consider the system
\begin{equation}
\label{eq:systwavetransport}
\left\{
\begin{array}{rcll}
\ddot{u}(t,x)&=& u''(t,x),\qquad &t\ge 0,\ x\in (0,\ell_2),\\
\dot{p}(t,x)&=& -p'(t,x),\qquad &t\ge 0,\ x\in (0,\ell_1),
\end{array}
\right.
\end{equation}
with Dirichlet boundary condition at $\ell_2$
\begin{equation}
\label{eq:Dirbc}
u(t,\ell_2)=0,
\end{equation}
and the 
transmission condition at $0$
\begin{equation}
\label{eq:transmissionbcwavetransport}
u'(t,0)=\alpha p(t,0), \quad \dot{u}(t,0)=\beta p(t,0),
\end{equation}
where $\alpha, \beta$ are two real numbers fixed below in order to guarantee well-posedness of our system.

As in \autoref{exa:waves}, by identifying the interval $(0,\ell_i)$ with $\me_i$, $i=1, 2$,
 introducing $u_{\me_1}=(u', \dot{u})^\top=(u_{\me_1, 1}, u_{\me_1, 2})^\top$, and setting $u_{\me_2}=p$,
 we can transform our system into a first order system   of the form \eqref{eq:max1}
 with 
\[
M_{\me_1}=\begin{pmatrix}
0 & 1\\
1 & 0
\end{pmatrix}\quad\hbox{and therefore}\quad 
Q_{\me_1}=\begin{pmatrix}
1 & 0\\
0 & 1
\end{pmatrix}
\]
and $M_{\me_2}=-1$ (hence $Q_{\me_2}=1$). 

 Notice that our network possesses three vertices, 
$\mv_0$,} corresponding to $0$ (with $E_{\mv_0}=\{\me_1, \me_2\}$),  $\mv_1$,}corresponding to the endpoint $\ell_1$ of $\me_1$  (with $E_{\mv_1}=\{\me_1\}$), and  $\mv_2$, corresponding to the endpoint $\ell_2$ of $\me_2$  (with $E_{\mv_2}=\{\me_2\}$).
Since there is no boundary condition at $\mv_1$, we  set
${Y_{\mv_1} = \C^2,}$
to take into account  \eqref{eq:Dirbc}, we set
$Y_{\mv_2} = \C\times \{0\},$
and
  \eqref{eq:transmissionbcwavetransport} requires to take
\[
Y_{\mv_0} = \lin\left\{ (\alpha,\beta,1)^\top \right\}.
\]
With this notation, we see that
\[
\begin{split}
q_{\mv_1}(x)&=-|x|^2,\qquad  x\in \mathbb{C}^2,
\\
q_{\mv_2}(\xi)&=0, \qquad \xi\in Y_{\mv_2},
\end{split}
\]
while 
\[
q_{\mv_0}(\xi)=(1-2\alpha\beta) |x_3|^2\quad \text{for all }\xi=(x_1,x_2, x_3)^\top \in Y_{\mv_0}.
\]
This means that $Y_{\mv_0}$ is a subspace of the nonpositive isotropic cone associated with 
$q_{\mv_0}$ if and only if $2\alpha\beta \geq 1$.

Finally, as 
\[
\widetilde{Y_{\mv_1}^\perp}= \{0\}\times  \{0\}\times  \{0\},\quad
\widetilde{Y_{\mv_2}^\perp}=  \{0\}\times \C \times  \{0\},\quad  
 \widetilde{Y_{\mv_0}^\perp}=\lin \{(1,0,-\alpha)^\top, (0,1,-\beta)^\top\},
\]
condition \eqref{eq:basis} is fulfilled
if $\beta\ne 0$. Therefore, system \eqref{eq:systwavetransport} with the boundary/transmission conditions
\eqref{eq:Dirbc}-\eqref{eq:transmissionbcwavetransport} is governed by a strongly continuous semigroup
if $2\alpha\beta \geq 1$.  The semigroup is real but not positive.
 
 \medskip
A toy model of a transport process sandwiched between two diffusive ones, all three taking place on intervals of length 1, was proposed in~\cite[\S~2]{HusMug13} (we refer to that paper for an interpretation of such a model in terms of delayed equations and for a possible biologic motivation). If diffusion is replaced by a wave equation, the corresponding hyperbolic system satisfies the \autoref{assum} with
\[
M:=\begin{pmatrix}
0 & 1 & 0 & 0 & 0\\
1 & 0 & 0 & 0 & 0 \\
0 & 0 & -1 & 0 & 0\\
0 & 0 & 0 & 0 & 1\\
0 & 0 & 0 & 1 & 0
\end{pmatrix}\quad\hbox{and }N\equiv 0,\hbox{ hence}\quad 
Q:=\begin{pmatrix}
1 & 0 & 0 & 0 & 0\\
0 & 1 & 0 & 0 & 0\\
0 & 0 & 1 & 0 & 0\\
0 & 0 & 0 & 1 & 0 \\
0 & 0 & 0 & 0 & 1
\end{pmatrix}:
\]
where the unknown is 
\[
\mmu:=(u', \dot{u}, p, v', \dot{v})^\top.
\]
Here we have  two interior vertices $\mv_1,\mv_2$: the former corresponds to the endpoint 0 of both $(0,\ell_1)$
and $(0,\ell_2)$, while the latter corresponds to the endpoint 0 of $(0,\ell_3)$ and the endpoint $\ell_1$ of $(0,\ell_1)$.
Hence imposing conditions \eqref{eq:Dirbc}-\eqref{eq:transmissionbcwavetransport} as before along with
 Dirichlet boundary condition at $\ell_3$
\begin{equation}
\label{eq:Dirbcone3}
v(t,\ell_3)=0,
\end{equation}
and the 
transmission condition at $\mv_2$
\begin{equation}
\label{eq:transmissionbcwavetransport3}
v'(t,0)=\gamma p(t,\ell_1), \quad \dot{v}(t,0)=\delta p(t,\ell_1),
\end{equation}
with   two real numbers  $\gamma, \delta$, one can show well-posedness of this problem
if  $2\alpha\beta \geq 1$, $2\gamma\delta \geq -1$ and $\gamma\ne 0$.   The semigroup is, again, real but not positive.

\subsection{Hybrid string/beam equations}
Ammari et al.~\cite{AmmMeh11,AmmMerReg12} have proposed models that consist of several combinations of strings and beams. 
In particular,  in~\cite{AmmMeh11} a collection of 1 wave and $N$ beam equations is considered on a star graph. The necessary $4N+2$ conditions consist of the following;
\begin{itemize}
\item $N+1$ transmission conditions: continuity of all solutions at the star's center along with a Kirchhoff-type condition on the beams' shear forces (third derivative of solutions) and the strings' flux (first derivative);
\item $3N+1$ boundary conditions: zero conditions on the beams' slopes and bending moments (first and second derivatives, respectively) along with closed feedbacks on the beams' shear forces and the strings' fluxes.
\end{itemize}
This model trivially satisfies  \autoref{assum} with  
\[
M_\me=\begin{pmatrix}
0 & 1\\
1 & 0\\
\end{pmatrix}
\quad
\hbox{and}
\quad
M_\me=\begin{pmatrix}
0 & 0 & 0 & 1\\
0 & 0 & 1 & 0\\
0 & 1 & 0 & 0\\
1 & 0 & 0 & 0\\
\end{pmatrix}
\]
in the case of the string-like and beam-like edges, respectively. In this case, \autoref{prop:main2} applies and we deduce that the hyperbolic system is governed by a strongly  continuous semigroup. Indeed, arguing as in \autoref{rk:globalbc} we deduce that this semigroup is contractive, hence the energy of solutions is decreasing.

The aim in~\cite{AmmMeh11} was to discuss the stabilization of an elastic system: the reason why it makes sense to consider closed feedbacks is that if they are replaced by zero conditions, then in view of the duality between continuity and Kirchhoff conditions, a direct computation shows the assumptions of \autoref{cor:main1} are satisfied and we conclude that the system is governed by a unitary group. 

\subsection{The Dirac equation}\label{sec:dirac}
The 1D Dirac equation is briefly discussed in~\cite[\S~1.1]{Tha92}: it was later extended to the case of networks and thoroughly studied by Bolte and his coauthors~\cite{BolHar03,BolSti07}, who also observed that it then takes on each edge the form
\[
\imath \hbar \frac{\partial}{\partial t} \psi=\left(\hbar c\begin{pmatrix}0 & -1\\ 1 & 0\end{pmatrix} \frac{\partial }{\partial x} +mc^2 \begin{pmatrix}1 & 0\\ 0 & -1\end{pmatrix}\right)\psi\qquad \hbox{in }(-\infty,\infty)\times (-\infty,\infty)
\]
for a $\C^2$-valued unknown $\psi=(\psi^{(1)},\psi^{(2)})$. A parametrization of skew-adjoint realizations on  a network has been presented in~\cite{BolHar03}; we are going to study the more general problem of finding boundary conditions that lead to a group or merely a semigroup,  which still yields \textit{forward} well-posedness of the Dirac equation.

To begin with, observe that \autoref{assum} are especially satisfied by taking
\[
M_\me =\begin{pmatrix}
0 & \imath c\\
-\imath c & 0
\end{pmatrix},\quad
Q_\me =\begin{pmatrix}
1 & 0\\
0 & 1
\end{pmatrix},\quad\hbox{and}\quad N_\me =\begin{pmatrix}
- \imath \frac{mc^2}{\hbar} & 0\\
 0 &  \imath \frac{mc^2}{\hbar}
\end{pmatrix},\qquad \me\in\mE.
\]
We hence deduce from Lemma~\ref{lem:cond4.4} that, no matter what the boundary/transmission conditions look like, a semigroup governing the Dirac equation cannot be real, let alone positive.

Let us now study the quadratic form $q_\mv$. We 
observe that
by~\eqref{eq:defTv},  {$T_\mv$ is $k_\mv\times k_\mv$ block-diagonal matrix whose diagonal blocks equal $\pm M_\me$, according to the appropriate value $\iota_{\mv \me}$ of the incidence matrix.
Hence, if we write
\[
\gamma_\mv(\psi)=(\psi^{(1)}_{\me}(\mv), \psi^{(2)}_{\me}(\mv))^\top_{\me\in \mE_{\mv}} =: (\xi_\me,\eta_\me)^\top_{\me\in \mE_{\mv}}\in\C^{k_\mv} ,
\]
then this vector  is an isotropic vector for the associated quadratic form $q_\mv$ if and only if 
\begin{equation}\label{eq:vanishingima-new}
\sum_{\me\in \mE} \iota_{\mv \me} \Im(\xi_{\me}  \cdot\bar \eta_{\me})=0.
\end{equation}
Any set of vector spaces $Y_\mv\subseteq\C^{k_\mv}$, consisting of vectors $(\xi_\me,\eta_\me)^\top_{\me\in \mE_{\mv}}$ that satisfy~\eqref{eq:vanishingima-new} possibly}
 induces boundary conditions that determine a realization of $\mathcal A$ generating a group; in fact, necessarily a \textit{unitary} group, since~\eqref{cond:iso} is clearly satisfied. A well-known example is that of vectors  $\xi$ being scalar multipliers of the ``characteristic function'' ${\mathbf 1}_{\mE_\mv}:={\mathbf 1}_{\{\me\in\mE_\mv\}}$ and $\eta$ with the same support $\mE_\mv$ and orthogonal to $\xi$, where $\mE_\mv\subset \mE$ is the set of edges incident with any given vertex $\mv$. {Recalling the notation $\iota_{\mE_\mv}$ from 
 \autoref{ex:Maxwell},}
  this gives rise to 
 \[ Y_\mv:=\lin\{{\mathbf 1}_{\mE_\mv}\} \oplus \lin\left\{\iota_{\mE_\mv}\right\}^\perp,
 \]
 corresponding to continuity of the first coordinate of $\psi$ across all vertices, and a Kirchhoff condition on the second coordinate,
  (clearly, swapping the transmission conditions in vertices satisfied by the two coordinates yields again a group generator); in this case, the   condition~\eqref{eq:basis} need not be satisfied. However, we can apply Theorem~\ref{thm:new-for-dirac}, since this choice of $Y_\mv$ implies that $T_\mv^{-1} Y_\mv^\perp=Y_\mv$. Another possibility is e.g.\ given by letting $Y_\mv:=\{(\alpha_1,\alpha_1,\ldots,\alpha_{|\mE_\mv|},\alpha_{|\mE_\mv|})^\top:\alpha \in \C^{k_\mv} \}$, corresponding to $\psi_\me^{(1)}(\mv)=\psi_\me^{(2)}(\mv)$.
  
If we turn to the issue of mere contractive well-posedness, then we observe that any non positive isotropic cone consists of vectors  such that
\begin{equation}\label{eq:positiveima-new}
  \sum_{\me\in \mE} \iota_{\mv \me} \Im(\xi_{\me} \bar \eta_{\me})\ge 0;
\end{equation}
accordingly,  any  $Y_\mv$ all of whose elements satisfy~\eqref{eq:positiveima-new} {is a candidate for generation of a semigroup. Indeed, by Theorem~\ref{prop:main2} such a choice of {$Y_\mv$}}
 induces a realization of the operator $\mathcal A$ that generates a strongly continuous contractive semigroup {if additionally 
 \eqref{eq:basis} holds 
or all elements of $T_\mv^{-1}Y_\mv^\perp$ satisfiy~\eqref{eq:positiveima-new}.} A somewhat trivial example is given by  $Y_\mv:=\{0\}\oplus 	\C^{|\mE_\mv|}$, corresponding to decoupled case of Dirichlet boundary conditions imposed at 
 {an} endpoint of each interval on both coordinates of the unknown; or, more generally (and interestingly) of  {$Y_\mv=\{(\imath B\eta,\iota_{\mE_\mv}\odot\eta):\eta\in \C^{k_\mv}\}$ for some  matrix $B$ of size {$k_\mv$ } with only purely imaginary eigenvalues
 (implying that $B$ and  $B^\ast$ are accretive and dissipative)
  since   $Y_\mv^\perp=\{(\eta,\imath \iota_{\mE_\mv}\odot B^\ast \eta):\eta\in \C^{k_\mv}\}$, 
   where for two vectors $a=(a_\me)_{\me\in \mE}$
 and  $b=(b_\me)_{\me\in \mE}$,  we recall that   $a\odot b$ means the Hadamard product of $a$ and $b$ defined by
 \[
 a\odot b=(a_\me b_\me)_{\me\in \mE}.
 \]}

\subsection{Second sound in networks}

So-called ``second sound'' is an exotic, wave-like phenomenon of heat diffusion that was first proposed by Landau to explain unusual behaviors in ultracold helium.
Second sound has ever since been observed in several materials -- most recently by Huberman et al.~\cite{HubDunChe19} also in graphite around cozy $130^\circ K$. As thoroughly discussed in~\cite{Racke}, one classical model going back to Lord and Shulman~\cite{LorShu67} boils down to the linear equations of thermoelasticity 
 \begin{equation}\label{systsecondsound}
\left\{
\begin{array}{rcll}
\ddot{z}-\alpha z'' + \beta \theta'&=&0  \quad &\hbox{ in  }  (0,\ell) \times(0,+\infty),\\
\dot{\theta}+\gamma q'+\delta \dot{z}'&=&0 \quad &\hbox{ in  }  (0,\ell) \times(0,+\infty),\\
\tau_0 \dot{q}+q + \kappa \theta'&=&0  \quad&\hbox{ in  }  (0,\ell) \times(0,+\infty),
\end{array}
\right.
\end{equation}
where  $z$, $\theta$, and $q$ represent the displacement, the temperature difference to a fixed reference temperature, and the heat flux, respectively, and $\alpha, \beta,\gamma, \delta, \tau_0,\kappa$ are positive constants. Racke has discussed in~\cite{Racke} the asymptotic stability of this system under three classes of boundary conditions:
\begin{enumerate}[(i)]
\item $z(0)=z(\ell)=q(0)=q(\ell)=0$,
\item $z(0)=z(\ell)=\theta(0)=\theta(\ell)=0$,
\item $\alpha z'(0)=\beta\theta(0)$, $\theta'(0)=0$, $z(\ell )=\theta(\ell)=0$, 
\end{enumerate}
 proving in detail well-posedness in the case of (i) and suggesting to use a similar strategy to study (ii) and (iii). In fact, the boundary conditions (iii) actually represents a dynamic condition for the unknown $q$ at 0, and hence seem to require a subtler analysis: we will consider them along with further hyperbolic systems with dynamic boundary conditions in a forthcoming paper~\cite{KraMugNic20b}.
We rewrite~\eqref{systsecondsound} as~\eqref{eq:max1} by letting
$u_\me := (z'_{\me}, \dot{z}_{\me}, \theta_\me, q_\me)$; then the  \autoref{assum} are satisfied taking 
\[
M_\me := \begin{pmatrix}
  0& 1& 0& 0\\
  \alpha& 0& -\beta&0\\
  0& -\delta& 0 &-\gamma\\
  0& 0& -\frac{\kappa}{\tau_0}& 0
\end{pmatrix},\quad
Q_\me := \begin{pmatrix}
  \alpha\delta& 0& 0& 0\\
 0& \delta& 0&0\\
  0&  0 &\beta &0\\
  0& 0& 0 &\frac{\beta\gamma\tau_0}{\kappa}
\end{pmatrix},\quad \text{and}\quad
N_\me := \begin{pmatrix}
  0 & 0& 0& 0\\
 0& 0 & 0&0\\
  0&  0 & 0 &0\\
  0& 0& 0 & -\frac{1}{\tau_0}
\end{pmatrix}.\]
The choice of $Q_\me$ is rather natural and indeed a similar term was also used to regularize the inner product by Racke, see~\cite[(18)]{Racke}. A direct computation shows that
\[
Q_\me M_\me=\begin{pmatrix}
0 & \alpha\delta & 0 & 0\\
\alpha\delta & 0 & -\beta\delta & 0\\
0 & -\beta\delta & 0 & -\beta \gamma\\
0 & 0 & -\beta\gamma & 0
\end{pmatrix}
\]
with  four eigenvalues of the form
$\pm\sqrt{\frac{H\pm2\sqrt{K}}{2}},$
where $H:= \alpha^2 \delta^2+\beta^2 \delta^2+ \beta^2\gamma^2$ and
$K:= H^2 - 4 \alpha^2\beta^2\gamma^2\delta^2$.
 Because $H^2>K$ whenever $\alpha,\beta,\gamma,\delta>0$, $Q_\me M_\me$ has two positive and two negative eigenvalues. 
 
 This is coherent with both above choices of boundary conditions (in the purely hyperbolic case of $\tau_0\ne 0$). For the sake of simplicity, let us focus for a while on the case of an individual interval that can be expressed in our formalism taking as $Y$ at each endpoint  the spaces
\[
\C\times\{0\}\times \C\times \{0\}
\quad\hbox{and}\quad  
 \C\times\{0\}\times \{0\}\times \C,
\]
respectively. A further possible choice for a subspace of the  null isotropic cone is e.g.
\[
\{0\}\times \C\times \{0\}\times \C,
\]
corresponding to 
\begin{itemize}
\item $z'(0)=z'(\ell)=\theta(0)=\theta(\ell)=0$.
\end{itemize}
If we however regard an interval as a loop (a network with one edge and one vertex), all these boundary conditions turn out to be only special cases of a more general setting. Indeed, a direct computation shows that a necessary condition for the vector
\[
\gamma(u):=\big(z'(0),\dot{z}(0),\theta(0),q(0), z'(\ell), \dot{z}(\ell),\theta(\ell),q(\ell) \big)^\top
\]
to lie in the  null isotropic cone of $T$ defined as in \eqref{eq:Tdef}
is that
\begin{equation}\label{eq:z1z2top}
\Re(Z_1 \cdot \bar Z_2 - Q\cdot\bar \Theta - Z_2 \cdot \bar  \Theta)=0
\end{equation}
where
\[
Z_1:=\begin{pmatrix}
-\alpha z'(0)\\ \alpha z'(\ell)
\end{pmatrix},\quad 
Z_2:=\begin{pmatrix}
\delta  \dot{z}(0)\\ \delta  \dot{z}(\ell)
\end{pmatrix},\quad \Theta:=\begin{pmatrix}
-\beta \theta(0)\\ \beta \theta(\ell)
\end{pmatrix},\quad Q:=\begin{pmatrix}
\gamma q(0)\\ \gamma q(\ell)
\end{pmatrix}.
\]
This is for instance the case if
\[
Z_1\perp Z_2\qquad\hbox{and}\qquad (Z_2+Q)\perp \Theta;
\]
this condition can e.g.\ be enforced by imposing
\[
Z_1,\Theta\in \lin\{ \mathbf 1_{\C^2}\}^\perp,\quad Z_2,Q\in \lin\{ \mathbf 1_{\C^2}\},
\]
where $\lin\{ \mathbf 1_{\C^2}\}$ is the subspace of $\C^2$ spanned by the vector $(1,1)^\top$. This is a hardly surprising choice for the reader familiar with evolution equations on networks which corresponds to periodic-type conditions
\begin{itemize}
\item $\dot{z}(0)=\dot{z}(\ell)$, $ z'(0)=z'(\ell)$, $\theta(0)=\theta(\ell)$, and $q(0)=q(\ell)$
\end{itemize}
and in turn to
\[
\gamma(u)\in Y:=\lin\{ \mathbf 1_{\C^2}\} \oplus \lin\{ \mathbf 1_{\C^2}\} \oplus \lin\{ \mathbf 1_{\C^2}\} \oplus \lin\{ \mathbf 1_{\C^2}\}.
\]
Indeed, $\dim(Y)=4$, hence condition \eqref{eq:basis_global}  is satisfied.

This paves the way to the study of second sound on collection of intervals with coupled boundary conditions, an especially interesting issue, as second sound has been conjectured in~\cite{HubDunChe19} to take place in graphene -- a network: more precisely, hexagonal lattice of carbon atoms --, already at  room temperature.

Indeed, one can apply our general theory in order to describe transmission conditions leading to well-posedness; an easy computation shows that the relevant equation is a higher dimensional counterpart of~\eqref{eq:z1z2top}. An educated guess suggests to study conditions of continuity (across the ramification nodes) on both displacement and temperature, i.e., on $z$ -- hence $\dot{z}$ -- and $\theta$, along with a Kirchhoff-type condition on $z'$ and $q$. It is remarkable that this choice does not satisfy~\eqref{eq:z1z2top}.
However, it is not difficult to see that all boundary values that satisfy either
\begin{itemize}
\item continuity on $z$ -- hence $\dot{z}$ -- as well as $q$, along with
\item Kirchhoff-type conditions
\[{ \sum_{\me\in \mE_\mv}{\iota}_{\mv \me} z'_{\me}(\mv)=0}\quad\hbox{and}\quad 
{ \sum_{\me\in \mE_\mv}{\iota}_{\mv \me} \theta_{\me}(\mv)=0}
\]
\end{itemize}
on $z'$ and $\theta$; or else
\begin{itemize}
\item continuity on $z'$ and $\theta$, along with
\item Kirchhoff-type conditions 
\[
{ \sum_{\me\in \mE_\mv}{\iota}_{\mv \me} {\dot{z}_{\me}}(\mv)=0}
 \quad\hbox{and}\quad 
 { \sum_{\me\in \mE_\mv}{\iota}_{\mv \me} q_{\me}(\mv)=0}
\]
\end{itemize}
on $z$ -- hence $\dot{z}$ -- as well as $q$ define a totally isotropic subspace of the null isotropic cone. (If the vertex $\mv$ has degree 1, then in both cases the first conditions become  void, whereas the second reduce to Dirichlet conditions.) Again, we see that
\eqref{eq:basis}  is satisfied and conclude that the system is governed by a strongly continuous group on $\bL^2(\mathcal G)$. 

All above spaces $Y$ are invariant under taking both the real and the positive part.
Furthermore, $Q,M,N$ are real valued and $Q,N$ are diagonal, but  $M$ is not, hence by \autoref{prop:real} and \autoref{prop:posit}  the semigroup generated by $\mathcal A$ with any of these transmission conditions is real but not positive.

Furthermore, $\mathcal A$ generates merely a semigroup whenever the space $Y$ defining the boundary conditions is a subspace of the nonpositive isotropic cone of $T$: this can e.g.\ enforced by assuming 
that
\[
Z_1=B Z_2\qquad\hbox{and}\qquad (Z_2+Q)=-C \Theta
\]
for some dissipative matrices $B,C$, provided $Y$ has the correct dimension. 
Because
\[
Q_\me N_\me = \begin{pmatrix}
0 & 0& 0& 0\\
 0& 0 & 0&0\\
  0&  0 & 0 &0\\
  0& 0& 0 &-\frac{\beta\gamma}{\kappa}
\end{pmatrix}
\]
is dissipative and $Q_{\me},M_\me$ are spatially constant, this semigroup is then automatically contractive.

\appendix
\section{Hyberbolicity revisited}

\bl\label{l:hyperbolic_Q-new} 
Let  $[0,\ell_\me]\ni x\mapsto M_\me(x)\in M_{k_\me}(\C)$ be a Lipschitz continuous matrix-valued function such that $M_\me(x)$ is invertible for each  $x\in  [0,\ell_\me]$.
Then matrix $M_\me(x)$ is \emph{Lipschitz-diagonalizable},
 i.e., there exist two Lipschitz continuous matrix-valued functions $[0,\ell_\me]\ni x\mapsto S_\me(x)\in M_{k_\me}(\C)$ and $[0,\ell_\me]\ni x\mapsto D_\me(x)\in M_{k_\me}(\C)$ such that for all $x\in [0,\ell_\me]$ both $S_\me(x),D_\me(x)$ are invertible,  $D_\me(x)$ is diagonal  and real, and furthermore
\be\label{assumMehyperbolic}
 M_\me(x)= S_\me^{-1}(x) D_\me(x) S_\me(x),
\ee
if only if \autoref{assum}.(\ref{assum:Q}) holds for some Lipschitz continuous, uniformly positive definite matrix-valued function $[0,\ell_\me]\ni x\mapsto Q_\me(x)\in M_{k_\me}(\C)$.
\el
\begin{proof}
If the matrix $M_\me$ can be diagonalized as above,
then we readily check 
that \autoref{assum}.(\ref{assumMe}) holds with the Hermitian matrices
\[
Q_\me(x):=S_\me^\ast(x) D_\me^{2n}(x) S_\me(x),\quad x\in [0,\ell_\me],
\]
for any $n\in \N_0:=\{0,1,2,\ldots\}$. This matrix is indeed uniformly positive definite, 
because for any $\xi\in \mathbb{C}^{k_\me}$, one has
\begin{equation}
\label{eq:uniforbdQme}
Q_\me(x)\xi\cdot \bar \xi=D_\me^{2n}(x) Y(x)\cdot \overline{Y(x)},
\end{equation}
where $Y(x):=S_\me(x) \xi$.
Since  the mappings
\[
[0,\ell_\me]\to [0,\infty): x\to \|D_\me^{-2n}(x) \| \quad\hbox{ and }\quad
[0,\ell_\me]\to [0,\infty): x\to \|S_\me^{-1}(x) \|
\]
{are} continuous and positive, there exists a positive constant $\alpha$ such that
\[
0< \|D_\me^{-2n}(x) \|\leq \alpha \hbox{ and } 0<\|S_\me^{-1}(x)\|\leq \alpha \text{ for all } x\in  [0,\ell_\me].
\]
These estimates in \eqref{eq:uniforbdQme} lead to
\[
Q_\me(x)\xi\cdot \bar \xi\geq {\alpha^{-1} \|Y(x)\|^2 \geq}
{\alpha^{-3}} \|\xi\|^2  \text{ for all } x\in  [0,\ell_\me].
\]
Finally as a composition of Lipschitz  continuous mappings, $Q_e{(\cdot)}$ is Lipschitz continuous, too.

Conversely, let us assume that {\autoref{assum}.(\ref{assum:Q})} holds. First, as each $Q_\me(x)$ is Hermitian and positive definite,  $Q_\me^{\frac{1}{2}}(x)$ is well-defined.  Now we notice that 
\autoref{assum}.(\ref{assumMe})  is equivalent to
\[
Q_\me^{\frac{1}{2}}(x) M_\me(x)Q_\me^{-\frac{1}{2}}(x)=Q_\me^{-\frac{1}{2}}(x)M_\me^\ast(x)Q_\me^{\frac{1}{2}}(x){  \text{ for all } x\in  [0,\ell_\me].}
\]
 Since this right-hand side is the adjoint of the left-hand side, each matrix $Q_\me^{\frac{1}{2}}(x) M_\me(x) Q_\me^{-\frac{1}{2}}(x)$, $x\in [0,\ell_\me]$,
 is Hermitian, hence it is diagonalizable by a family of unitary matrices $S_{0,\me}(x)$ and real diagonal matrices $D_{\me}(x)$  such that
\[
S_{0,\me} (x)\left(Q_\me^{\frac{1}{2}}(x)M_\me(x)Q_\me^{-\frac{1}{2}}(x)\right)S_{0,\me}^\ast(x)=D_\me(x) {\text{ for all } x\in  [0,\ell_\me],}
\]
which yields \eqref{assumMehyperbolic} with
$S_\me(x):= S_{0,\me}(x)Q_\me^{\frac{1}{2}}(x)$.
By assumptions on $M_\me$ and $Q_\me$, both  $S_\me(\cdot),D_\me(\cdot)$ are Lipschitz continuous functions and  $S_\me(x),D_\me(x)$ are invertible matrices for all $x\in  [0,\ell_\me]$.
 \end{proof}
 
 \section{Three versions of the Fundamental Lemma of Calculus of Variations}
  
We first prove a density result in the subset of positive integrable functions; we recall the notation in~\eqref{eq:lpc} and write likewise $\mathcal{D}(0,\ell; \mathbb{R})$ for the set of real-valued test functions.
\begin{lemma}\label{l:appendixB0}
Let $\ell>0$. The set 
\[
\{\varphi^2: \varphi\in \mathcal{D}(0,\ell; \mathbb{R})\}
\]
is dense in $L^1(0,\ell; \mathbb{R}_+)$.
\end{lemma}
\begin{proof}
Indeed let us fix $u \in L^1(0,\ell; \mathbb{R}_+)$, then $\sqrt{u}$ belongs to $L^2(0,\ell; \mathbb{R})$
and therefore there exists a sequence $(\varphi_n)_{n\in \mathbb{N}}$
of functions in  $\mathcal{D}(0,\ell; \mathbb{R})$ such that
\begin{equation}\label{eq:consqrtu}
\varphi_n\to \sqrt{u}\hbox{ in } L^2(0,\ell; \mathbb{R}) \hbox{ as } n\to\infty.
\end{equation}
Therefore by Cauchy--Schwarz's inequality we have
\begin{eqnarray*}
\int_0^\ell |u-\varphi_n^2|\,dx
&=& \int_0^\ell |(\sqrt{u}-\varphi_n)(\sqrt{u}+\varphi_n)|\,dx
\\
&\leq&\|\sqrt{u}-\varphi_n\|_{L^2(0,\ell)} \|\sqrt{u}+\varphi_n\|_{L^2(0,\ell)}
\\
&\leq&\|\sqrt{u}-\varphi_n\|_{L^2(0,\ell)} (\|\sqrt{u}\|_{L^2(0,\ell)}
+\|\varphi_n\|_{L^2(0,\ell)}.
\end{eqnarray*}
By \eqref{eq:consqrtu}, we conclude that this right-hand side tends to zero.
\end{proof}

Now we prove  two variants of
the fundamental lemma of the calculus of variations  (or ``du Bois--Reymond's lemma'').

\begin{lemma}\label{l:appendixB1}
 Let $h\in L^\infty(0,\ell; \mathbb{R})$ satisfy
\begin{equation}\label{eq:nullityassumptiondebase-le}
\int_0^{\ell}
h \varphi^2   \,dx\ge 0\qquad\hbox{for all } \varphi\in \mathcal{D}(0,\ell; \mathbb{R}),
\end{equation}
then $h\ge 0$. If in particular
\begin{equation}\label{eq:nullityassumptiondebase}
\int_0^{\ell}
h \varphi^2   \,dx= 0\qquad\hbox{for all } \varphi\in \mathcal{D}(0,\ell; \mathbb{R}),
\end{equation}
then $h=0$.
\end{lemma}
\begin{proof}
As the second assertion is a direct consequence of the first one, it remains to check
the first one. Let $h\in L^\infty(0,\ell; \mathbb{R})$ satisfy \eqref{eq:nullityassumptiondebase-le}, then as it is in $L^1(0,\ell; \mathbb{R})$, it can be split up as
\[
h=h^+-h^-,
\]
where $h^+, h^-\in L^1(0,\ell; \mathbb{R}_+)$ is the positive and negative part of $h$ respectively.
According to \eqref{eq:nullityassumptiondebase} and \autoref{l:appendixB0}
we have
\[
\int_0^{\ell}
h h^-  \,dx\geq 0,
\]
which in turn leads to
\[
\int_0^{\ell}
(h^-)^2  \,dx=0,
\]
and proves that $h=h^+$ is nonnegative.
\end{proof}

This Lemma allows to prove a matrix-valued version  of
the fundamental lemma of the calculus of variations.
\begin{lemma}\label{l:appendixB}
Let $k$ be a positive integer and let $A\in L^\infty(0,\ell; \mathbb{C}^{k\times k})$
be such that
\[
A^*(x)= A(x) \hbox{ for a.e. } x\in (0,\ell).
\]
If $A$ satisfies
\begin{equation}\label{eq:nullityassumption}
 \int_0^{\ell}
A u \cdot \bar u   \,dx=0\qquad\hbox{for all }u  \in \mathcal{D}(0,\ell)^{k},
\end{equation}
then $A=0$.
\end{lemma}  
\begin{proof}
First we show that the diagonal entries of   $A=(A_{i,j})_{1\leq i, j\leq k}$ are zero.
Indeed let us fix $i\in\{1,\cdots, k\}$
and in \eqref{eq:nullityassumption} take test functions $u$ in the form
$u=\varphi e_i$, where $\varphi\in \mathcal{D}(0,\ell; \mathbb{R})$ is abitrary and
$e_i$ is the $i$th element of the canonical basis of $\mathbb{C}^k$.
Then we get
\[
\int_0^{\ell}
A_{ii} \varphi^2   \,dx=0\qquad\hbox{for all } \varphi\in \mathcal{D}(0,\ell; \mathbb{R}),
\]
and   \autoref{l:appendixB1} yields $A_{ii}=0$ because $A_{ii}$ is real-valued.

Let us now manage the off-diagonal entries of $A$.
Fix $i,j \in  \{1,\cdots, k\}$ with $i<j$. Now we chose  two family of test-functions in \eqref{eq:nullityassumption}:
\\
1) First take test functions $u$ in the form
\[u=\varphi (e_i+e_j),
\]
 where $\varphi\in \mathcal{D}(0,\ell; \mathbb{R})$ is arbitrary. Then  \eqref{eq:nullityassumption}
 reduces to
\[
\int_0^{\ell}
(A_{ij}+A_{ji}) \varphi^2   \,dx=0\qquad\hbox{for all } \varphi\in \mathcal{D}(0,\ell; \mathbb{R}).
\]
Again 
$A_{ij}+A_{ji}=2\Re A_{ij}$ is real-valued because $A$ is hermitian,
and  \autoref{l:appendixB1} yields $\Re A_{ij}=0$.
\\
2) Second take test functions $u$ in the form
\[u=\varphi (e_i+\imath e_j),
\]
 where $\varphi\in \mathcal{D}(0,\ell; \mathbb{R})$ is arbitrary to obtain
 \[
\int_0^{\ell}
(A_{ij}-A_{ji}) \varphi^2   \,dx=0\qquad\hbox{for all } \varphi\in \mathcal{D}(0,\ell; \mathbb{R}).
\]
Since $A_{ij}-A_{ji}=2\imath \Im A_{ij}$, this means that
 \[
\int_0^{\ell}
\Im A_{ij}\varphi^2   \,dx=0\qquad\hbox{for all } \varphi\in \mathcal{D}(0,\ell; \mathbb{R}),
\]
and therefore  $\Im A_{ij}=0$ due to  \autoref{l:appendixB1}.
\end{proof}

 \section{On subspaces of isotropic cones associated with a quadratic form}

 In this section we fix a positive integer $k$
 and a hermitian and invertible matrix $P\in \mathbb{C}^{k\times k}$. Its  associated quadratic form
 $q$ is defined by 
 \[
 q(\xi)=P\xi\cdot \bar \xi, \quad \xi\in \mathbb{C}^k.
 \]
 Now we introduce some cones associated with $q$, see \cite[Def.~3.1]{Lam}.
 \begin{defi}\label{def:nonpositiveisocone}
 \begin{enumerate}[1)] 
\item The \emph{null isotropic cone}
associated with the quadratic form $q$ is defined as the set of isotropic  vectors associated with $q$, namely
the set of
 vectors $\xi\in \C^{k}$ such that
\begin{equation}\label{eq:isotropicvector}
 q(\xi)=0.
  \end{equation}
A subspace of the  null isotropic cone
associated with   $q$ is called  a \emph{totally isotropic subspace}
and the \emph{isotropy index} (of the quadratic space associated with $q$), denoted here by $i(q)$, is the maximum of the dimensions of the totally isotropic subspaces.
\item
The \emph{nonpositive   (resp.~nonnegative) isotropic cone}
associated with the quadratic form $q$  is defined as the set of vectors $\xi\in \C^{k}$ such that
\begin{equation}\label{eq:nonpositiveisotropicvector}
q(\xi)\leq 0  \hbox{ (resp.~} \geq 0).
\end{equation}
\end{enumerate}
\end{defi}

From Lemma 1.2 of \cite{Milnor-Husemoller} we know that 
$i(q)\leq k/2$ but, surprisingly, we could not find in the literature a reference that yields a characterization of $i(q)$. Hence
the goal of this appendix is to characterize this isotropic index as well as  the maximal dimension of any subspace
 of nonpositive isotropic cones.

Let $\{\lambda_i\}_{i=1}^k$ be the set of eigenvalues of $P$, repeated according to their multiplicities and
 enumerated in an increasing order, and denote by $\{u_i\}_{i=1}^k$  the set of the associated  normalized eigenvectors, i.e.,
 \[
 Pu_i=\lambda_i u_i\quad\hbox{and}\quad  u_i\cdot \bar u_j=\delta_{ij}\quad \text{ for all }  i\in \{1,\dots, k\}.
 \]
Denote by $k_-$ (resp. $k_+$) the number of negative (resp. positive) eigenvalues of $P$.
Without loss of generality, we may assume that $\lambda_i<0$ for $i\leq k_-$
and $\lambda_i>0$ for $i> k_-$.
We will see that the isotropic index $i(q)$ agrees with $\min\{k_-,k_+\}$.

\begin{lemma}\label{l:iso}
Any subspace of the {null}  isotropic  cone associated with the form $q$ has dimension at most $\kappa:=\min\{k_-,k_+\}$.
Furthermore if $\kappa\geq 1$, there exist   at least $2^\kappa$ subspaces of the null  anisotropic cone associated with $q$
 of dimension $\kappa$.
\end{lemma}
\begin{proof}
If $\kappa=0$, this means that $P$ is either positive definite or negative definite
and therefore the associated isotropic cone is reduced to  $\{0\}$. So the only case of interest is the case $\kappa\geq 1$.
By symmetry, we can assume that $\kappa=k_-$.
So let us now fix a subspace $I$   of the  isotropic cone  associated with $q$. 
Every nonzero $u\in I$ can be written as
\begin{equation}\label{eq:isotropicvectorinbasis}
u=\sum_{i=1}^k \alpha_i u_i,
\end{equation}
 for some $\alpha_i \in \mathbb{C}$ which are not all zero. Since  \eqref{eq:isotropicvector} is 
equivalent to
 \begin{equation}\label{eq:isotropiccond}
\sum_{i=k_-+1}^{k} |\alpha_i|^2\lambda_i=-\sum_{i=1}^{k_-} |\alpha_i|^2\lambda_i,
\end{equation}
we find that there exists at least one $i\leq k_-$
such that $\alpha_i\ne0$.

 Assume that $K:=\dim I>k_-$ and let $\{U_i\}_{i=1}^K$ be a basis of $I$. 
Let us write
\begin{equation}\label{eq:isotropicvectorbasis}
U_i=\sum_{j=1}^k\alpha_{ij} u_j,
\end{equation}
for some  $\alpha_{ij} \in \mathbb{C}$. By the previous remark, for all $i\leq K$, there exists
$j\leq k_-$ such that
$\alpha_{ij}\ne 0$.

Now we use a sort of Gram--Schmidt procedure: Starting with $i=1$
and without loss of generality (else we change the enumeration) we can assume that 
$\alpha_{11}\ne 0$ and, consequently, we have
\[
u_1=\delta_{1} \tilde U_1-\sum_{j=2}^k \frac{\alpha_{1j}}{\alpha_{11}} u_j.
\]
where 
we have set $\tilde U_1:=U_1$ and $\delta_{1}:=\frac{1}{\alpha_{11}}$.
Plugging this expression into \eqref{eq:isotropicvectorbasis} with $i=2$, we find that
\begin{equation}\label{eq:isotropicvectorbasis2}
U_2=\frac{\alpha_{21}}{\alpha_{11}} U_1+\sum_{j=2}^k\tilde \alpha_{2j} u_j,
\end{equation}
with some $\tilde \alpha_{2j}\in \mathbb{C}$.
This means that the new vector
$\tilde U_2:=U_2-\frac{\alpha_{21}}{\alpha_{11}} U_1$, that is still in $I$,
has at least one coefficient $\tilde \alpha_{2j}$ different from zero for $j\in\{2, \dots, k_-\}$. Again,  after a possible
change of enumeration, we can assume that $\tilde \alpha_{22}\ne 0$, hence
we have
 \[
u_2=\delta_{2}  \tilde U_2-\sum_{j=3}^k \frac{\alpha_{2j}}{\tilde \alpha_{22}} u_j.
\]
where $\delta_{2}=\frac{1}{\tilde \alpha_{22}}$.
Note that the new set
$\{\tilde U_1, \tilde U_2\}\cup \{U_i\}_{i=3}^K$ forms a basis of $I$.
By iterating this procedure, after $k_-$ steps, we will find a basis
$\{\tilde U_i\}_{i=1}^{k_-}\cup \{U_i\}_{i=k_-+1}^K$ of $I$ such that
\begin{equation}\label{eq:isotropicvectorbasisinverse}
u_i={\delta_{i} \tilde U_i} - \sum_{j=k_-+1}^k \beta_{ij}u_j \text{ for all }i=1,\dots, k_-,
\end{equation}
for some $\delta_{j}\in \mathbb{C}{\setminus\{0\}}$ and some $ \beta_{ij}\in  \mathbb{C}$.

By using the expansion \eqref{eq:isotropicvectorbasis} of $U_{k_-+1}$
and \eqref{eq:isotropicvectorbasisinverse}, we find
\[
U_{k_-+1}=\sum_{j=1}^{k_-} \alpha_{ij} \delta_{j} \tilde U_j+\sum_{j=k_-+1}^k\gamma_{ij} u_j,
\]
for some $\gamma_{ij} \in \mathbb{C}^*$.
We then arrive to a contradiction because on one hand the vector
{$V:=U_{k_-+1}-\sum_{j=1}^{k_-} \alpha_{ij} \delta_{j} \tilde U_j$ is in $I$, hence $q(V)=0$, 
while on the other hand $V\ne 0$ is a linear combination of the $u_j$'s for $j\geq k_-+1$, hence $q(V)>0$}. 

For the last assertion,  if in \eqref{eq:isotropicvectorinbasis}, for all $i=1,\dots, k_-$, 
we   chose
\[
\alpha_i=1 \quad\text{ and } \quad\alpha_{i'}=0 \text{ for all }i'\notin \{i, k_-+i\},
\]
condition 
\eqref{eq:isotropiccond} will hold if and only if
\[
|\alpha_{k_-+i}|^2=-\frac{\lambda_i}{\lambda_{k_-+i}},
\]
or equivalently
\[
\alpha_{k_-+i}=\pm \imath \sqrt{\frac{\lambda_i}{\lambda_{k_-+i}}}.
\]
This yields the isotropic vectors
\begin{equation}\label{eq:isotropic-vec}
U_i^{\pm}= u_i\pm \imath \sqrt{\frac{\lambda_i}{\lambda_{k_-+i}}} u_{k_-+i}.
\end{equation}
And, since the $U_i^{+}$'s and the $U_i^{-}$'s are linearly independent, we find
$2^{k_-}$ possibilities.
\end{proof}
 
 A similar assertion holds for subspaces of the nonpositive isotropic cone
associated with the quadratic form $q$ -- one such subspace is spanned by the first $k_-$ eigenvectors $\{u_i\}_{i=1}^{k_-}$ of $P$.

\begin{lemma}\label{l:nonpositiveiso}
Any subspace of the nonpositive isotropic subspace associated with the form $q$ has dimension at most $k_-$.
\end{lemma}

\begin{proof}
If $k_-=0$, this means that $P$ is positive definite  
and therefore the only possible choice for such a subspace is $\{0\}$. So the only case of interest is the case $k_-\geq 1$.
Now the proof is exactly the same as the one of the previous Lemma. Indeed let $I$ be such a subspace and  let 
$u\in I$ different from zero, then it admits the splitting  \eqref{eq:isotropicvectorinbasis}
with some $\alpha_i \in \mathbb{C}$ not all zeroes. Since the constraint 
\[
P u\cdot \bar u\leq 0,
\]
is equivalent to
 \begin{equation}\label{eq:nonpositiveisotropiccond}
\sum_{i=k_- +1}^{k} |\alpha_i|^2\lambda_i+\sum_{i=1}^{k_-} |\alpha_i|^2\lambda_i\leq 0,
\end{equation}
we again find that there exists at least one $i\leq k_-$
such that $\alpha_i\ne0$. The previous argument then leads to $\dim I\leq  k_-$.
\end{proof}


\begin{thebibliography}{10}

\bibitem{Ali89}
F.~{Ali Mehmeti}.
\newblock Regular solutions of transmission and interaction problems for wave
  equations.
\newblock {\em Math.\ Meth.\ Appl.\ Sci.}, 11:665--685, 1989.

\bibitem{Ali94}
F.~{Ali Mehmeti}.
\newblock {\em Nonlinear {W}aves in {N}etworks}, volume~80 of {\em Math.\
  Research}.
\newblock Akademie, Berlin, 1994.

\bibitem{AmmMeh11}
K.~Ammari and M.~Mehrenberger.
\newblock Study of the nodal feedback stabilization of a string--beams network.
\newblock {\em J.\ Appl.\ Math.\ Comput.}, 36:441--458, 2011.

\bibitem{AMbook}
K.~Ammari and S.~Nicaise.
\newblock {\em Stabilization of elastic systems by collocated feedback}, volume
  2124 of {\em Lecture Notes in Mathematics}.
\newblock Springer, Cham, 2015.

\bibitem{AmmMerReg12}
K.~Ammari, D.~Mercier, V.~R{\'e}gnier, and J.~Valein.
\newblock {Spectral analysis and stabilization of a chain of serially connected
  Euler-Bernoulli beams and strings}.
\newblock {\em Comm.\ Pure Appl.\ Anal.}, 11:785--807, 2012.

\bibitem{BasCor16}
G.~Bastin and J.-M. Coron.
\newblock {\em {Stability and Boundary Stabilization of 1-D Hyperbolic
  Systems}}, volume~88 of {\em Progress in Nonlinear Differential Equations}.
\newblock Birkh{\"a}user, Basel, 2016.

\bibitem{BatKraRha17}
A.~B\'{a}tkai, M.~Kramar~Fijav\v{z}, and A.~Rhandi.
\newblock {\em {Positive Operator Semigroups}}, volume 257 of {\em Operator
  Theory: Advances and Applications}.
\newblock Birkh\"{a}user, Cham, 2017.

\bibitem{BenGav07}
S.~Benzoni-Gavage and D.~Serre.
\newblock {\em Multidimensional Hyperbolic Partial Differential Equations --
  First-order Systems and Applications}.
\newblock Clarendon Press, Oxford, 2007.

\bibitem{BolHar03}
J.~Bolte and J.~Harrison.
\newblock {Spectral statistics for the Dirac operator on graphs}.
\newblock {\em J.\ Phys.\ A}, 36:2747--2769, 2003.

\bibitem{BolSti07}
J.~Bolte and H.-M. Stiepan.
\newblock {The Selberg trace formula for Dirac operators}.
\newblock {\em J.\ Math.\ Phys.}, 47:112104, 2007.

\bibitem{borichev:10}
A.~Borichev and Y.~Tomilov.
\newblock Optimal polynomial decay of functions and operator semigroups.
\newblock {\em Math. Ann.}, 347(2):455--478, 2010.

\bibitem{Bressanetal}
A.~Bressan, S.~{\v{C}}ani{\'c}, M.~Garavello, M.~Herty, and B.~Piccoli.
\newblock Flows on networks: recent results and perspectives.
\newblock {\em EMS Surv. Math. Sci.}, 1(1):47--111, 2014.

\bibitem{BCGHP14}
A.~Bressan, S.~\v{C}ani\'{c}, M.~Garavello, M.~Herty, and B.~Piccoli.
\newblock Flows on networks: recent results and perspectives.
\newblock {\em EMS Surv. Math. Sci.}, 1(1):47--111, 2014.

\bibitem{Bre10}
H.~Brezis.
\newblock {\em {Functional Analysis, Sobolev Spaces and Partial Differential
  Equations}}.
\newblock Universitext. Springer-Verlag, Berlin, 2010.

\bibitem{CarMug09}
S.~Cardanobile and D.~Mugnolo.
\newblock Parabolic systems with coupled boundary conditions.
\newblock {\em J.\ Differ.\ Equ.}, 247:1229--1248, 2009.

\bibitem{Car99}
R.~Carlson.
\newblock Inverse eigenvalue problems on directed graphs.
\newblock {\em Trans.\ Amer.\ Math.\ Soc.}, 351:4069--4088, 1999.

\bibitem{Car00}
R.~Carlson.
\newblock {Nonclassical Sturm--Liouville problems and Schr{\"o}dinger operators
  on radial trees}.
\newblock {\em Electronic J.\ Differ.\ Equ.}, 71:1--24, 2000.

\bibitem{Carlson:11}
R.~Carlson.
\newblock Spectral theory for nonconservative transmission line networks.
\newblock {\em Netw. Heterog. Media}, 6(2):257--277, 2011.

\bibitem{Chen:79}
G.~Chen.
\newblock Energy decay estimates and exact boundary value controllability for
  the wave equation in a bounded domain.
\newblock {\em J. Math. Pures Appl. (9)}, 58:249--273, 1979.

\bibitem{dagerzuazua}
R.~D{\'a}ger and E.~Zuazua.
\newblock {\em Wave propagation, observation and control in {$1\text{-}d$}
  flexible multi-structures}, volume~50 of {\em Math\'ematiques \& Applications
  (Berlin) [Mathematics \& Applications]}.
\newblock Springer-Verlag, Berlin, 2006.

\bibitem{Dor08b}
B.~Dorn.
\newblock Semigroups for flows in infinite networks.
\newblock {\em Semigroup Forum}, 76:341--356, 2008.

\bibitem{DorKraNag10}
B.~Dorn, M.~Kramar~Fijav{\v{z}}, R.~Nagel, and A.~Radl.
\newblock The semigroup approach to transport processes in networks.
\newblock {\em Physica D}, 239:1416--1421, 2010.

\bibitem{EndSte10}
S.~Endres and F.~Steiner.
\newblock {The Berry--Keating operator on $L^2(\mathbb R_>,dx)$ and on compact
  quantum graphs with general self-adjoint realizations}.
\newblock {\em J.\ Phys.\ A}, 43:095204, 2010.

\bibitem{Eng13}
K.-J. Engel.
\newblock Generator property and stability for generalized difference
  operators.
\newblock {\em J.\ Evol.\ Equ.}, pages 1--24, 2013.

\bibitem{EK17}
K.-J. Engel and M.~Kramar Fijav{\v{z}}.
\newblock Exact and positive controllability of boundary control systems.
\newblock {\em Netw. Heterog. Media}, 12 (2017): 319--337.

\bibitem{EKNS08}
K.-J. Engel, M.~Kramar Fijav{\v{z}}, R.~Nagel, and E.~Sikolya.
\newblock Vertex control of flows in networks.
\newblock {\em Netw. Heterog. Media}, 3 (2008): 709--722.

\bibitem{EKKNS10}
K.-J. Engel, M.~Kramar Fijav{\v{z}}, B.~Kl\"oss, R.~Nagel, and E.~Sikolya.
\newblock Maximal controllability for boundary control problems. 
\newblock {\em Appl. Math. Optim.},  62 (2010): 205--227.



\bibitem{EngNag00}
K.-J. Engel and R.~Nagel.
\newblock {\em One-{P}arameter {S}emigroups for {L}inear {E}volution
  {E}quations}, volume 194 of {\em Graduate Texts in Mathematics}.
\newblock Springer-Verlag, New York, 2000.

\bibitem{Eva10}
L.~Evans.
\newblock {\em Partial {D}ifferential {E}quations -- second edition}, volume~19
  of {\em Graduate Studies in Mathematics}.
\newblock Amer.\ Math.\ Soc., Providence, RI, 2010.

\bibitem{Exn13}
P.~Exner.
\newblock Momentum operators on graphs.
\newblock In H.~Holden, B.~Simon, and G.~Teschl, editors, {\em {Spectral
  Analysis, Differential Equations and Mathematical Physics: A Festschrift in
  Honor of Fritz Gesztesy's 60th Birthday}}, volume~87 of {\em Proc.\ Symp.\
  Pure Math.}, pages 105--118, Providence, RI, 2013. Amer.\ Math.\ Soc.

\bibitem{HubDunChe19}
S.~Huberman, R.~Duncan, K.~Chen, B.~Song, V.~Chiloyan, Z.~Ding, A.~Maznev,
  G.~Chen, and K.~Nelson.
\newblock Observation of second sound in graphite at temperatures above 100 k.
\newblock {\em Science}, 364:375--379, 2019.

\bibitem{HusMug13}
A.~Hussein and D.~Mugnolo.
\newblock Quantum graphs with mixed dynamics: the transport/diffusion case.
\newblock {\em J.\ Phys.\ A}, 46:235202, 2013.

\bibitem{ImpJol14}
S.~Imperiale and P.~Joly.
\newblock Mathematical modeling of electromagnetic wave propagation in
  heterogeneous lossy coaxial cables with variable cross section.
\newblock {\em Appl.\ Num.\ Math.}, 79:42--61, 2014.

\bibitem{JacMorZwa15}
B.~Jacob, K.~Morris, and H.~Zwart.
\newblock {$C_0$}-semigroups for hyperbolic partial differential equations on a
  one-dimensional spatial domain.
\newblock {\em J. Evol. Equ.}, 15:493--502, 2015.

\bibitem{JacZwa12}
B.~Jacob and H.~Zwart.
\newblock {\em Linear Port-Hamiltonian Systems on Infinite-dimensional Spaces},
  volume 223 of {\em Oper.\ Theory Adv.\ Appl.}
\newblock Birkh{\"a}user, Basel, 2012.

\bibitem{JorPedTia13}
P.~Jorgensen, S.~Pedersen, and F.~Tiang.
\newblock Momentum operators in two intervals: Spectra and phase transition.
\newblock {\em Compl.\ Anal.\ Oper.\ Theory}, 7:1735--1773, 2013.

\bibitem{Klo10}
B.~Kl{\"o}ss.
\newblock Difference operators as semigroup generators.
\newblock {\em Semigroup Forum}, 81:461--482, 2010.

\bibitem{Klo12}
B.~Kl\"oss.
\newblock The flow approach for waves in networks.
\newblock {\em Oper.\ Matrices}, 6:107--128, 2012.

\bibitem{KosSch99}
V.~Kostrykin and R.~{S}chrader.
\newblock Kirchhoff's rule for quantum wires.
\newblock {\em J.\ Phys.\ A}, 32:595--630, 1999.

\bibitem{KraPuc20}
M.~Kramar Fijav\v{z}, A.~Puchalska.
\newblock Semigroups for dynamical processes on metric graphs. 
\newblock{\em Phil. Trans. R. Soc. A},  378: 20190619, 2020.

\bibitem{KraMugNic20b}
M.~Kramar Fijav\v{z}, D.~Mugnolo, and S.~Nicaise.
\newblock Hyperbolic systems with dynamic boundary conditions.
\newblock (in preparation), 2020.

\bibitem{KraSik05}
M.~Kramar and E.~Sikolya.
\newblock Spectral properties and asymptotic periodicity of flows in networks.
\newblock {\em Math.\ Z.}, 249:139--162, 2005.

\bibitem{Kuc02}
P.~Kuchment.
\newblock Graph models of wave propagation in thin structures.
\newblock {\em Waves Random Media}, 12:1--24, 2002.

\bibitem{KurMugWol19}
P.~Kurasov, D.~Mugnolo, and V.~Wolf.
\newblock Analytic solutions for stochastic hybrid models of gene regulatory
  networks.
\newblock arXiv:1812.07788.

\bibitem{LagLeuSch94}
J.~Lagnese, G.~Leugering, and E.~Schmidt.
\newblock {\em Modeling, {A}nalysis, and {C}ontrol of {D}ynamic {E}lastic
  {M}ulti-{L}ink {S}tructures}.
\newblock Systems and Control: Foundations and Applications. Birkh{\"a}user,
  Basel, 1994.

\bibitem{Lam}
T.~Y. Lam.
\newblock {\em Introduction to quadratic forms over fields}, volume~67 of {\em
  Graduate Studies in Mathematics}.
\newblock American Mathematical Society, Providence, RI, 2005.

\bibitem{LaxPhi60}
P.~Lax and R.~Phillips.
\newblock Local boundary conditions for dissipative symmetric linear
  differential operators.
\newblock {\em Comm.\ Pure Appl.\ Math.}, 13:427--455, 1960.

\bibitem{LeuSch89}
G.~Leugering and E.~Schmidt.
\newblock On the control of networks of vibrating strings and beams.
\newblock In {\em IEEE Conference on Decision and Control}, pages 2287--2290,
  Providence, RI, 1989. IEEE.

\bibitem{LorShu67}
H.~Lord and Y.~Shulman.
\newblock A generalized dynamical theory of thermoelasticity.
\newblock {\em Journal of the Mechanics and Physics of Solids}, 15:299--309,
  1967.

\bibitem{Lum80}
G.~Lumer.
\newblock Connecting of local operators and evolution equations on networks.
\newblock In F.~Hirsch, editor, {\em Potential Theory (Proc.\ Copenhagen
  1979)}, pages 230--243, Berlin, 1980. Springer-Verlag.

\bibitem{MaffucciMiano:06}
A.~Maffucci and G.~Miano.
\newblock A unified approach for the analysis of networks composed of
  transmission lines and lumped circuits.
\newblock In {\em Scientific computing in electrical engineering}, volume~9 of
  {\em Math. Ind.}, pages 3--11. Springer-Verlag, Berlin, 2006.

\bibitem{MatSik07}
T.~M{\'a}trai and E.~Sikolya.
\newblock Asymptotic behavior of flows in networks.
\newblock {\em Forum Math.}, 19:429--461, 2007.

\bibitem{Milnor-Husemoller}
J.~Milnor and D.~Husemoller.
\newblock {\em Symmetric bilinear forms}.
\newblock Springer-Verlag, New York-Heidelberg, 1973.
\newblock Ergebnisse der Mathematik und ihrer Grenzgebiete, Band 73.

\bibitem{Mug14}
D.~Mugnolo.
\newblock {\em {Semigroup Methods for Evolution Equations on Networks}}.
\newblock Underst.\ Compl.\ Syst. Springer-Verlag, Berlin, 2014.

\bibitem{nicaise:93}
S.~Nicaise.
\newblock {\em Polygonal Interface Problems}, volume~39 of {\em Methoden und
  Verfahren der mathematischen Physik}.
\newblock Peter Lang GmbH, Europ\"aischer Verlag der Wissenschaften,
  Frankfurt/M., 1993.

\bibitem{Nic:2017}
S.~Nicaise.
\newblock Control and stabilization of {$2\times 2$} hyperbolic systems on
  graphs.
\newblock {\em Math. Control Relat. Fields}, 7:53--72, 2017.

\bibitem{NicVal07}
S.~Nicaise and J.~Valein.
\newblock Stabilization of the wave equation on 1-d networks with a delay term
  in the nodal feedbacks.
\newblock {\em Networks Het.\ Media}, 3:425--479, 2007.

\bibitem{Ouh05}
E.~Ouhabaz.
\newblock {\em Analysis of {H}eat {E}quations on {D}omains}, volume~30 of {\em
  Lond.\ Math.\ Soc.\ Monograph Series}.
\newblock Princeton Univ.\ Press, Princeton, NJ, 2005.

\bibitem{PavFad83}
B.~Pavlov and M.~D. Faddeev.
\newblock Model of free electrons and the scattering problem.
\newblock {\em Theor.\ Math.\ Phys.}, 55:485--492, 1983.

\bibitem{Paz83}
A.~Pazy.
\newblock {\em Semigroups of linear operators and applications to partial
  differential equations}, volume~44 of {\em Appl.\ Math.\ Sci.}
\newblock Springer-Verlag, New York, 1983.

\bibitem{pruss:84}
J.~Pr{\"u}ss.
\newblock On the spectrum of {$C_{0}$}-semigroups.
\newblock {\em Trans. Amer. Math. Soc.}, 284(2):847--857, 1984.


\bibitem{Racke}
R.~{Racke}.
\newblock {Thermoelasticity with second sound -- exponential stability in
  linear and nonlinear 1D.}
\newblock {\em {Math. Methods Appl. Sci.}}, 25:409--441, 2002.

\bibitem{Rau85}
J.~Rauch.
\newblock Symmetric positive systems with boundary characteristic of constant
  multiplicity.
\newblock {\em Trans.\ Amer.\ Math.\ Soc.}, 291:167--187, 1985.

\bibitem{Rau12}
J.~Rauch.
\newblock {\em Hyperbolic partial differential equations and geometric optics},
  volume 133 of {\em Graduate Studies in Mathematics}.
\newblock Amer.\ Math.\ Soc., Providence, RI, 2012.

\bibitem{SchSeiVoi15}
C.~Schubert, C.~Seifert, J.~Voigt, and M.~Waurick.
\newblock Boundary systems and (skew-) self-adjoint operators on infinite
  metric graphs.
\newblock {\em Math.\ Nachr.}, 288:1776--1785, 2015.

\bibitem{Tha92}
B.~Thaller.
\newblock {\em {The Dirac Equation}}.
\newblock Springer-Verlag, New York, 1992.

\bibitem{WauWeg19}
M.~Waurick and S.-A. Wegner.
\newblock Dissipative extensions and port-hamiltonian operators on network.
\newblock 2019.

\bibitem{Yok01}
T.~Yokota.
\newblock {Invariance of closed convex sets under semigroups of nonlinear
  operators in complex {H}ilbert spaces}.
\newblock {\em SUT J. Math.}, 37:91--104, 2001.

\bibitem{ZwaLeGMasVil10}
H.~Zwart, Y.~Le~Gorrec, B.~Maschke, and J.~Villegas.
\newblock Well-posedness and regularity of hyperbolic boundary control systems
  on a one-dimensional spatial domain.
\newblock {\em ESAIM Control Optim. Calc. Var.}, 16:1077--1093, 2010.

\end{thebibliography}
\end{document}